\documentclass[12pt,
  oneside]{amsart}

\usepackage{amscd, amsmath, colonequals}

\usepackage[usenames]{xcolor}

\title[]{Real submanifolds of maximum complex tangent space at a CR singular point, I}
\author[]{Xianghong Gong}

\address{Department of Mathematics,
 University of Wisconsin, Madison, WI 53706, U.S.A.}
 \email{gong@math.wisc.edu}

\author{Laurent Stolovitch}
\address{CNRS and Laboratoire J.-A. Dieudonn\'e
U.M.R. 6621, Universit\'e de Nice - Sophia Antipolis, Parc Valrose
06108 Nice Cedex 02, France.}
\email{stolo@unice.fr}
\thanks{Research of L. Stolovitch was partially supported by ANR grant ``ANR-10-BLAN 0102'' for the project DynPDE and ANR grant ``ANR-14-CE34-0002-01'' for the project ``Dynamics and CR geometry''}


\newtheorem{thm}{Theorem}[section]
\newtheorem{cor}[thm]{Corollary}
\newtheorem{prop}[thm]{Proposition}
\newtheorem{lemma}[thm]{Lemma}

\newcommand{\diag}{\operatorname{diag}}

\theoremstyle{definition}

\newtheorem{defn}[thm]{Definition}
\newtheorem{exmp}[thm]{Example}

\newtheorem{rem}[thm]{Remark}

\renewcommand{\th}[1]{\begin{thm}\label{#1}}
\newcommand{\eth}{\end{thm}}
\newcommand{\co}[1]{\begin{cor}\label{#1}}
\newcommand{\eco}{\end{cor}}
\renewcommand{\le}[1]{\begin{lemma}\label{#1}}
\newcommand{\ele}{\end{lemma}}
\newcommand{\pr}[1]{\begin{prop}\label{#1}}
\newcommand{\epr}{\end{prop}}

\newcommand{\ga}{\begin{gather}}
\newcommand{\ega}{\end{gather}}
\newcommand{\gan}{\begin{gather*}}
\newcommand{\egan}{\end{gather*}}
\newcommand{\al}{\begin{align}}
\newcommand{\eal}{\end{align}}
\newcommand{\aln}{\begin{align*}}
\newcommand{\ealn}{\end{align*}}
\newcommand{\eq}[1]{\begin{equation}\label{#1}}
\newcommand{\eeq}{\end{equation}}

\newcommand{\ci}{~\cite}

\newcommand{\f}[2]{\frac{#1}{#2}}

\newcommand{\fix}{\operatorname{Fix}}

\newcommand{\cc}{{\bf C}}
\newcommand{\nn}{{\bf N}}
\newcommand{\zz}{{\bf Z}}
\newcommand{\rr}{{\bf R}}

\newcommand{\ov}{\overline}

\newcommand{\id}{\operatorname{I}}

\newcommand{\RE}{\operatorname{Re}}
\newcommand{\IM}{\operatorname{Im}}

\newcommand{\cL}{\mathcal}

\newcommand{\I}{\operatorname{I}}

\newcommand{\all}{\alpha}

\newcommand{\gaa}{\gamma}

\newcommand{\del}{\delta}
\newcommand{\Del}{\Delta}
\newcommand{\var}{\varphi}
\newcommand{\e}{\epsilon}
\newcommand{\om}{\omega}
\newcommand{\Om}{\Omega}

\newcommand{\la}{\lambda}

\newcommand{\pd}{\partial}

\newcommand{\re}[1]{(\ref{#1})}
\newcommand{\rea}[1]{$(\ref{#1})$}
\newcommand{\rl}[1]{Lemma~\ref{#1}}

\newcommand{\rp}[1]{Proposition~\ref{#1}}
\newcommand{\rt}[1]{Theorem~\ref{#1}}
\newcommand{\rd}[1]{Definition~\ref{#1}}

\newcommand{\rla}[1]{Lemma~$\ref{#1}$}

\newcommand{\rta}[1]{Theorem~$\ref{#1}$}

\newcounter{pp}
\newcommand{\bpp}{\begin{list}{$\hspace{-1em}\alph{pp})$}{\usecounter{pp}}}
\newcommand{\epp}{\end{list}}

\newcounter{ppp}
\newcommand{\bppp}{\begin{list}{$\hspace{-1em}(\roman{ppp})$}{\usecounter{ppp}}}
\newcommand{\eppp}{\end{list}}

\def\beq{\begin{equation}}
\def\eeq{\end{equation}}

\begin{document}


\begin{abstract}
We study a germ of  real analytic $n$-dimensional submanifold  of ${\mathbf C}^n$ that has a complex tangent space of maximal dimension at a CR singularity.
Under some assumptions, we show its equivalence to a normal form under a local biholomorphism at the singularity.
 We also show that if a real submanifold is formally equivalent 
 to a quadric,  it is actually holomorphically equivalent to it,  if a small divisors  condition is satisfied.
 Finally, we investigate the existence of   a complex submanifold of positive dimension in ${\mathbf C}^n$  that intersects
 a real submanifold along two totally and real analytic submanifolds that intersect transversally at a possibly non-isolated CR singularity.
\end{abstract}

\date{\today}
 \maketitle
\tableofcontents

\addtocontents{toc}{\protect\setcounter{tocdepth}{1}}

\setcounter{section}{0}
\setcounter{thm}{0}\setcounter{equation}{0}
\section{Introduction and main results}
\label{sect1}

\subsection{Introduction}
We are concerned with the local holomorphic invariants of a real analytic submanifold $M$ in $\cc^n$.
The tangent space of $M$ at a point $x$ contains a maximal complex subspace of dimension $d_{x}$. When  
$d_x$ is constant,
$M$ is called a  Cauchy-Riemann (CR) submanifold. The CR  submanifolds have been extensively studied since E.~Cartan \cite{Ca32}, \cite{Ca33},  Tanaka~\cite{Ta62}, and Chern-Moser~\cite{chern-moser}.

We say that a point $x_0$ in
 a  real submanifold
 $M$ in $\cc^n$ is  a 
 CR singularity,
 if the complex tangent spaces $T_xM\cap J_xT_xM$  do not
have a constant dimension in any neighborhood of $x_0$. A real submanifold with a CR singularity must have codimension at least $2$.
The study of real submanifolds with CR singularities was initiated by E.~Bishop in his pioneering work~\cite{Bi65}. He investigated  a $C^\infty$ real submanifold $M$ of which
the complex tangent space  at a CR singularity
is minimal, that is exactly one-dimensional.
 The very elementary models of this kind of manifolds are    the Bishop quadrics $Q_\gaa$ that depends on the
 Bishop invariant $0\leq\gaa\leq\infty$, given by
$$
Q_\gaa\subset\cc^2\colon z_2=|z_1|^2+\gaa(z_1^2+\ov z_1^2), \quad 0\leq\gaa<\infty; 
\quad Q_\infty\colon z_2=z_1^2+\ov z_1^2.
$$
 The complex tangent at the origin  is said to be {\it elliptic} if $0\leq\gaa<1/2$,  {\it parabolic} if $\gaa=1/2$, or {\it hyperbolic} if $\gaa>1/2$.
In ~\cite{MW83}, Moser and Webster studied the normal form problem of  a real analytic surface
$M$ in $\cc^2$
which is the higher order perturbation of $Q$. They showed that
 when $0<\gaa<1/2$, $M$ is holomorphically equivalent, near the origin,  to a normal form which
is  an algebraic surface that depends only on $\gaa$ and two discrete invariants.
We  mention that  the Moser-Webster normal form  theory, as in Bishop's work,
actually deals with an $n$-dimensional
real submanifold $M$ in $\cc^n$,  of which
the complex tangent space has   (minimum) dimension $1$
at a CR singularity.

The main purpose of this paper is to investigate an  $n$-dimensional real analytic  submanifold $M$ in $\cc^n$, which is totally real outside a proper analytic subset and of which the complex tangent space has the
{\it largest}
possible
dimension at a  given CR singularity. 
 We shall say that the singularity is a (maximal) complex tangent. The dimension must be $p=n/2$. Therefore,  $n=2p$ is even.
   We are interested in
the normal form problem, the rigidity property,  and  the local analytic geometry  of such real analytic manifolds.

In  suitable holomorphic coordinates, a $2p$-dimensional real analytic submanifold $M$ in $\cc^{2p}$
 that has a complex tangent space of maximum dimension at the origin  is given by
\eq{mzpjintr}
M\colon z_{p+j}=E_j(z',\ov z'),
\quad 1\leq j\leq p,
\eeq
where $z'=(z_1,\ldots, z_p)$ and
$$
E_j(z',\ov{z'})=h_j(z',\ov z')+q_j(\ov z') + O(|(z',\ov z')|^3). 
$$
Moreover,  each  $h_j(z',\ov z')$ is a homogeneous quadratic polynomial in $z',\ov z'$ without holomorphic or anti-holomorphic terms,
  and each $q_j(\ov z')$ is a homogeneous quadratic polynomial in $\ov z'$.
  One of our goals is to seek suitable normal forms
  of perturbations of quadrics at the CR singularity (the origin). 



 \subsection{Basic invariants}
To study $M$, we consider its complexification in $\cc^{2p}\times\cc^{2p}$ defined by
\begin{equation}
{\mathcal M}\colon
\begin{cases}
z_{p+i} = E_{i}(z',w'), & i=1,\ldots, p,
\\
w_{p+i} = \ov{ E_i}(w',z'),& i=1,\ldots, p.\\
\end{cases}
\nonumber
\end{equation}
 It is a complex submanifold of complex
dimension $2p$ with coordinates $(z',w')\in\cc^{2p}$.  Let $\pi_1,\pi_2$ be the restrictions of the projections $(z,w)\to z$
and   $(z,w)\to w$ to $\cL M$, respectively.
Note that $\pi_2=C\pi_1\rho_0$, where  $\rho_0$ is the restriction to $\cL M$ of the anti-holomorphic involution $(z,w)\to(\ov w,\ov z)$ and $C$ is the complex conjugate.

  Our basic assumption is the following condition.

\medskip
\noindent
{\bf Condition B.}  $q(z')=(q_1(z'),\ldots, q_p(z'))$ satisfies $q^{-1}(0)=\{0\}.$
\medskip

When $p=1$, condition B corresponds to 
 the non-vanishing of  the Bishop invariant $\gaa$.
  When $\gaa=0$,  Moser \cite{moser-zero} obtained a  formal  normal form that is still subject to
  further  formal changes of coordinates.
In \cite{HY09},  Huang and Yin obtained a formal normal form and a complete holomorphic classification for real analytic surfaces with $\gaa=0$.
    The formal normal forms for co-dimension two real submanifolds  in $\cc^n$ have been further studied
  by Huang-Yin~\cite{HY12} and Burcea~\cite{Bu13}.
  Coffman~\cite{Co06} showed that  any
 $m$ dimensional real analytic submanifold in $\cc^n$ of  one-dimensional complex tangent
 space at  a  CR singularity satisfying certain non-degeneracy conditions
 is locally holomorphically equivalent to a unique algebraic submanifold,
 provided $2(n+1)/3\leq m<n$.

	
When $M$ is a {\it quadric}, i.e. all $E_j$ in \re{mzpjintr} are   quadratic  polynomials, our basic condition~B
is equivalent to $\pi_1$ being a $2^p$-to-1 branched covering. 
Since $\pi_2=C\pi_1\rho_0$, then $\pi_2$ is also a $2^p$-to-$1$ branched covering.
We will see that the CR singularities of the real submanifolds are closely connected with these branched coverings and their deck transformations.


\medskip

We now introduce our main results.
Some  of them are analogous to the Moser-Wester theory.
 We will also describe new situations which arise with the maximum complex tangency.

\subsection{Branched coverings and deck transformations}
In section~\ref{secinv}, we study the existence of deck transformations for $\pi_1$.
We   show that they must be involutions generating 
 an abelian group of order $2^k$ for some $0\leq k\leq p$. The latter is a major difference with the case $p=1$.
 Indeed, 
in  the Moser-Webster theory, 
the group of deck transformations is generated by a unique non-trivial involution $\tau_1$.
Therefore, we will
impose the following condition.

\medskip
\noindent
{\bf Condition D.}  {\it $M$ satisfies condition $B$ and the branched   covering   $\pi_1$ 
of $\cL M$
admits the maximum $2^p$ deck transformations.}
\medskip

Condition D gives rise to two families of commuting involutions $\{\tau_{i1},\ldots, \tau_{i2^p}\}$
intertwined by the anti-holomorphic
involution $\rho_0\colon(z',w')\to(\ov w',\ov z')$ such that $\tau_{2j}=\rho_0\tau_{1j}\rho_0$ $(1\leq j\leq 2^p)$ are deck
transformations of $\pi_2$.    We will call $\{\tau_{11},\ldots, \tau_{12^p},\rho_0\}$ the set of {\it Moser-Webster involutions}.
We will show that there is a unique set of $p$ generators for the deck transformations  of $\pi_1$, denoted by $\tau_{11},\ldots, \tau_{1p}$, such  that each $\tau_{1j}$   fixes a hypersurface in $\cL M$ pointwise.  Then
$$\tau_1=\tau_{11}\circ\cdots\circ\tau_{1p}$$
 is the unique deck transformation of which the fixed-point set   has the smallest dimension $p$.
 Let  $\tau_2=\rho_0\tau_1\rho_0$ and
  $$
 \sigma=\tau_1\tau_2.
 $$
 Then $\sigma$ is {\it reversible} by $\tau_j$ and $\rho_0$, i.e.
 $\sigma^{-1}=\tau_j\sigma\tau_j^{-1}$ and $\sigma^{-1}=\rho_0\sigma\rho_0$.

 As in the Moser-Webster theory,  we will show that the existence of such $2^p$ deck transformations    transfers the normal form problem for the real submanifolds into the normal form problem for the sets
  of involutions $\{\tau_{11},\ldots,\tau_{1p},
  \rho_0\}$.

In this paper we will make the following assumption.

 \medskip
 \noindent
{\bf Condition J.} {\it   $M$ satisfies condition D and  $M$ is diagonalizable, i.e.
   $\sigma'(0)$ is diagonalizable.}


\medskip


Note that the condition excludes the higher dimensional analogous complex tangency of {\it parabolic} type, i.e. of $\gaa=1/2$.
The normal form problem for the parabolic complex tangents has been studied by Webster~\cite{We92}, and in~\cite{Go96} where the normalization is divergent in general.  In~\cite{AG09}, Ahern and Gong constructed a moduli space for real analytic submanifolds that are formally equivalent to the Bishop quadric with $\gamma=1/2$.

\subsection{Product quadrics} In this paper,
the basic model for quadric manifolds with   a CR singularity satisfying condition J is a product of  3 types of quadrics defined by
\begin{gather}\label{Qgs2}
 Q_{\gamma_e}\subset\cc^2 \colon z_{2}= (z_1+2\gaa_e\ov z_1)^2;\\
 Q_{\gamma_h}\subset\cc^2\colon z_{2}= (z_{1}+2\gaa_{h}\ov z_1)^2, \ 1/2<\gaa_h<\infty;\quad Q_\infty\colon z_2=z_1^2+\ov z_1^2;\\
 Q_{\gamma_s}\subset\cc^4\colon z_{3}= (z_1+2\gamma_s\ov z_{2})^2,
\quad z_{4}=( z_{2}+2
(1-\ov\gamma_{s})  \ov z_{1})^{2}.
\label{Qgs2+}
\end{gather}
Here $\gaa_s\in\cc$ and
\eq{0ge1}
0<\gamma_e<1/2, \quad
1/2<\gamma_h\leq\infty, \quad \RE
\gaa_s\leq1/2, \quad  \IM\gaa_s\geq0, \quad\gaa_s\neq0,1/2.
\eeq
Note that $Q_{\gaa_e}, Q_{\gaa_h}$ are elliptic and hyperbolic Bishop quadrics, respectively.  Realizing
a type of pairs
 of involutions  introduced in~\cite{St07}, we will say that the complex tangent of $Q_{\gaa_s}$ at the origin is {\it complex}.  We emphasize that this last type of quadric
 is new 
as it 
 is not holomorphically equivalent to a product of two Bishop surfaces.
  A product of the above quadrics
 will be called a {\em product of quadrics}, or a {\it product quadric}.
 We denote by $e_*,h_*, 2s_*$
  the number of elliptic, hyperbolic and complex coordinates, respectively.  We remark that the complex tangent of complex type has another basic model $Q_{\gamma_s}$ with $\gamma_s=1/2$, which is excluded by condition J (see \rp{sigs}).

\bigskip

%
%

This  is the first part of two papers devoted to the local study of real analytic manifold at maximal complex tangent point.
To limit its scope, we have to leave the complete classification of quadratic submanifolds of maximum deck transformations to  the second paper~\cite{part2} (see Theorem 1.1 therein), showing that there are  quadratic manifolds which are not holomorphically equivalent to a product quadric.  In~\cite{part2}, we also  show that all Poincar\'e-Dulac normal forms of the $\sigma$ of a general   higher order perturbation of a product quadric are divergent when $p>1$. With the divergent Poincar\'e-Dulac normal forms at our disposal,  we seek types of CR singularities that ensure  the convergent normalization and the analytic structure of the hull of holomorphy  associated with the types of CR singularities.

 We now introduce our main geometrical and dynamical results for analytic higher order perturbations of product quadric.
We first turn to  a holomorphic normalization of a real analytic submanifold $M$ with the so-called abelian CR singularity.  This will be achieved by
studying an integrability problem on a general family of commuting biholomorphisms described below. The holomorphic normalization will
be used to construct the local hull of holomorphy of $M$.
 We will also study the
 rigidity problem 
 of a quadric under higher order analytic perturbations,
 i.e.  the problem if such a perturbation  remains holomorphically equivalent to the quadric if  it is   formally equivalent
 to the quadric.
  The rigidity problem is  reduced to a theorem of holomorphic linearization of one or several commuting diffeomorphisms   that was devised in \cite{stolo-bsmf}.
 Finally,  we will study the existence of
 holomorphic  submanifolds attached to the real submanifold $M$. These are complex submanifolds of dimension $p$ intersecting $M$ along two totally real analytic submanifolds that intersect  transversally at a CR singularity.  Attaching complex submanifolds has less constraints than
 finding a convergent normalization.
 A remarkable feature
of   attached complex submanifolds is that their existence depends only on the existence of suitable (convergent)
 invariant submanifolds of $\sigma$. 

\subsection{Normal form of commuting biholomorphisms}
\begin{defn}
Let $\cL F=\{F_1,\ldots, F_\ell\}$ be a finite family of germs of biholomorphisms of $\cc^n$
fixing the origin.
 Let $D_m$ be the linear part of $F_m$ at the origin.  We say that the family $\cL F$ is (resp. formally)  {\it completely integrable}, if there is a (resp. formal) biholomorphic mapping $\Phi$ such that
$\{\Phi^{-1}F_m\Phi\colon 1\leq m\leq \ell\}=\{\hat F_m\colon 1\leq m\leq \ell\}$ satisfies
\bppp
\item $\hat F_m(z)=(\mu_{m 1}(z)z_1,\ldots, \mu_{m n}(z)z_n)$ where $\mu_{mj}$ are germs of holomorphic (resp. formal) functions such that $\mu_{m j}\circ D_{m'}=\mu_{m j}$ for
 $1\leq m,m'\leq \ell$  and $ 1\leq j\leq n$.   In particular,  $\hat F_m$ commutes with $D_{m'}$ for all $1\leq m,m'\leq\ell$.
\item  For each $j$ and each $Q\in\nn^n$ with $|Q|>1$,  $\mu_m^Q(0)=\mu_{m j}(0)$ hold for all $m$ if and only if
 $\mu_{m}^Q(z)=\mu_{mj}(z)$ hold  for all $m$.
\eppp
\end{defn}
A necessary condition for $\cL F$ to be formally completely integrable is that $F_1,\ldots, F_\ell$ commute pairwise.
The main result of section~\ref{nfcb} is the following.
\begin{thm} \label{lFba}
Let $\cL F$ be a family of finitely many
germs of biholomorphisms  at the origin.
If $\cL F$ is formally completely integrable and its linear part $\cL D$  has the Poincar\'e type, then it is holomorphically  completely integrable.
\end{thm}
The definition of   Poincar\'e type is in Definition~\ref{small divisors}.
Such a formal integrability condition can hold under some geometrical properties.
 For instance, for a single germ of real analytic hyperbolic area-preserving mapping, the result  was due to
Moser~\cite{moser-hyperbolic},
and for a single germ of reversible hyperbolic  holomorphic mapping $\sigma=\tau_1\tau_2$
of which $\tau_1$ fixes a hypersurface,
 this result was due to Moser-Webster~\cite{MW83}.
Such  results for commuting germs of vector fields were obtained in \cite{St00,  stolo-annals} under a   collective small divisors Brjuno-type condition.  
  Our result is inspired by these results.

\subsection{Holomorphic normalization for the abelian CR singularity} In section~\ref{sectabel}, we  obtain the convergent normalization for an  {\it abelian} CR singularity which we now define.
 We first consider
a  product quadric $Q$ which    satisfies condition~J. 
  So the deck transformations of
 $\pi_1$
for  the complexification of
$Q$ are generated by $p$ involutions of which each fixes a hypersurface pointwise. We denote them by $T_{11}, \ldots, T_{1p}$.
 Let $T_{2j}=\rho T_{1j}\rho$.
It turns out that
  each $T_{1j}$ commutes with all $T_{ik}$
   except one,  $T_{2k_j}$ for some $1\leq k_j\leq p$. When we formulate $S_j=T_{1j}T_{2k_j}$ for $1\leq j\leq p$,
  the  $S_1, \ldots, S_p$ commute pairwise. 
  Consider a general  $M$ that  is a third-order perturbation of product quadric $Q$ and
  satisfies  condition~J.  
   We   define $\sigma_j=\tau_{1j}\tau_{2k_j}$. In suitable coordinates, $T_{ij}$
  (resp. $S_j$) is the linear part of $\tau_{ij}$ (resp. $\sigma_j$) at the origin.
  We say that the complex tangent of
  a third order perturbation $M$ of a product quadric
  at the origin is of {\it abelian type},  if
  $\sigma_1, \dots, \sigma_p$
  commute pairwise.
 If each linear part  $S_j$ of $\sigma_j$ has exactly two eigenvalues $\mu_j,
\mu_j^{-1}$ that are different from $1$,  then $\cL S:=\{S_1,\ldots, S_p\}$ is of Poincar\'e type if and only if $|\mu_j|\neq1$ for all $j$.
 As mentioned previously,  Moser and Webster
 actually dealt with $n$-dimensional real submanifolds in $\cc^n$ that
have the minimal dimension of complex tangent subspace at a CR singular point.
In their situation, there is only one possible composition, that is $\sigma=\tau_1\tau_2$. When the complex tangent
has an elliptic but non-vanishing Bishop invariant, $\sigma$ has exactly two positive  eigenvalues that are separated by
$1$, while the remaining eigenvalues are $1$ with multiplicity $n-2$.

  As an application of \rt{lFba}, we will prove  the following convergent normalization.
  \begin{thm}\label{iabelm}
Let $M$  be a germ of real analytic submanifold  in $\cc^{2p}$
that is a third order perturbation of a product quadric given by \rea{Qgs2}-\rea{0ge1} with an abelian CR singularity. Suppose that $M$ has all eigenvalues of modulus different from one, i.e. it has no hyperbolic component $(h_*=0)$ while each $\gaa_s$ in \rea{Qgs2+} satisfies $\RE\gamma_s<1/2$ additionally.
Then $M$ is holomorphically equivalent to
\begin{gather}\nonumber 
\widehat M\colon
z_{p+j}=\Lambda_{1j}(\zeta)\zeta_j,\quad \Lambda_{1j}(0)=\la_j,\quad 1\leq j\leq
p,
\end{gather}
where  $\zeta=(\zeta_1,\ldots, \zeta_p)$ are the   solutions to
\begin{align*}\nonumber
\zeta_e&=
 A_e(\zeta)z_e\ov z_e-
 B_e(\zeta)(z_e^2+\ov z_e^2),\quad 1\leq e\leq
e_*,\\
\zeta_s&=
 A_s(\zeta)z_s\ov z_{s+s_*}-
 B_s(\zeta)(z_s^2+\Lambda_{1s}^2(\zeta)\ov z_{s+s_*}^2),\quad  e_*< s\leq e_*+s_*,\\
\zeta_{s+s_*}&= 
A_{s+s_*}(\zeta) \ov z_s z_{s+s_*}-
B_{s+s_*}(\zeta)(z_{s+s_*}^2+\Lambda_{1(s+s_*)}^2(\zeta)\ov z_{s}^2),
\nonumber
\end{align*}
 while 
$ \Lambda_{1j}$ satisfies \rea{lam1e}-\rea{rhoz5},
and $A_j,B_j$  are rational functions in $\Lambda_{1j}$ defined  by \rea{AeAj}-\rea{AeAj+}.
\end{thm}

There are many non-product real submanifolds  of abelian CR singularity.
\begin{exmp} Let $0<\gaa_i<\infty$.
Let $R(z_1,\ov z_1)=|z_1|^2+\gaa_1(z_1^2+\ov z_1^2)+O(3)$ be a real-valued power series in $z_1,\ov z_1$ of real coefficients. Then the origin is an abelian CR singularity of
$$
M\colon z_3=R(z_1,\ov z_1), \quad z_4=(z_2+2\gaa_2 \ov z_2+z_2z_3)^2.
$$
\end{exmp}
  We will also present a more direct proof of \rt{iabelm} by using a    convergence theorem of Moser and Webster~\cite{MW83}
and some formal results from
section~\ref{nfcb}.  The above $\Lambda_{11}, \dots, \Lambda_{1p}$ satisfy conditions $\Lambda_{1j}(0)=\lambda_j$ and \re{lam1e}-\re{rhoz5} and are otherwise {\it arbitrary} convergent power series. The $\Lambda_{11}, \dots, \Lambda_{1p}$ may be subjected to further normalization.
In~\cite{part2}, we find a  unique holomorphic normal form by refining the above normalization for $M$ satisfying a non resonance condition and a third order non-degeneracy condition (see Theorem 5.6 in [GS15]); in particular, it shows the existence of infinitely many formal invariants and non-product structures of the manifolds when $p>1$.

 As an application of \rt{iabelm},
we will prove the following flattening result.
\begin{cor} Let $M$ be as in \rta{iabelm}. In suitable holomorphic coordinates, $M$
is contained in the linear subspace   defined by $z_{p+e}=\ov z_{p+e}$ and $z_{p+s}=\ov z_{p+s+s_*}$ where $1\leq e\leq
e_*$ and $ e_*< s\leq e_*+s_*$.
\end{cor}


\subsection{Analytic hull of holomorphy}
One of significances of the Bishop quadrics is that their  higher order analytic perturbation at an elliptic    complex tangent has
 a non-trivial hull of holomorphy.
As another application of the above normal form, we will  construct the local hull of holomorphy of $M$
via higher dimensional non-linear analytic polydiscs.
\begin{cor} \label{andiscver}
Let $M$ be as in \rta{iabelm}.  Suppose that $M$
  has   only elliptic component of complex tangent. Then   in suitable holomorphic coordinates,
$\cL H_{loc}(M)$,  the local hull of holomorphy of $M$, is filled
by a real analytic family of analytic polydiscs of dimension $p$.
\end{cor}
For a precise statement of the corollary, see \rt{hullj}.
The hulls of holomorphy for real submanifolds with a CR singularity have been studied extensively, starting with the work of Bishop. In the real analytic case with minimum complex tangent space at an elliptic
complex tangent, 
we refer to Moser-Webster~\cite{MW83} for $\gaa>0$, and Krantz-Huang \cite{HK95} for $\gaa=0$.  For the smooth case, see Kenig-Webster~\cite{KW82,KW84}, Huang~\cite{H98}.
For global
results on hull of holomorphy, we refer to \cite{bedford-gaveau, bedford-kling}.

\subsection{Rigidity of quadrics}

  In  Section~\ref{rigidquad},
we  prove the following theorem. 
\begin{thm}
 Let $M$ be a germ of analytic submanifold that is an higher order perturbation of a product quadric $Q$ in $\cc^{2p}$ given by \rea{Qgs2}-\rea{0ge1}. Assume that $M$ is formally equivalent to $Q$.
 Suppose that  each hyperbolic component has an eigenvalue $\mu_h$ which
is either a root of unity or satisfies Brjuno condition, and each complex component has an eigenvalue $\mu_s$ is not a root of unity and satisfies the Brjuno condition. 
Then $M$ is holomorphically equivalent to the product quadric.
\end{thm}
We emphasize that condition (1.5) ensure that $M$ is diagonalizable (condition J). It is plausible that theorem remains valid when $M$ satisfies condition J and $\mu_s$ satisfy the Brjuno condition or are roots of unity; however, the resonance condition requires some tedious changes of computation in section 6. The proof uses
 a theorem of linearization of holomorphic mappings
in \cite{stolo-bsmf}. 
Brjuno small divisors condition is defined by
\re{omnu}, with $\nu=\mu_h$ and $p= 
1$. When $p=1$, the result
 under 
 the stronger 
 Siegel condition is in
  \cite{Go94}.  This last statement requires a small divisors condition to be true as shown in \cite{Gon04}.  
   When $p=1$ with a vanishing Bishop invariant, such rigidity result was obtained by
Moser~\cite{moser-zero}  and by Huang-Yin \cite{HY09b} in a more general context. 

%
%

\subsection{Attached complex submanifolds}
We now describe  convergent results for attached complex submanifolds.  The
results are for a   general
 $M$, including the one of which the complex tangent might not be of abelian type.

We say that a formal complex submanifold $K$ is  {\it attached} to $M$ if $K\cap M$
 contains at least two germs of  totally real and formal submanifolds $K_1, K_2$
that intersect
 transversally at a given CR singularity. In~\cite{Kl85}, Klingenberg showed that when
$M$ is non-resonant and $p=1$, there is   a unique   formal holomorphic
curve attached to $M$ with a hyperbolic complex tangent. He also proved the convergence of the attached formal holomorphic curve under
a Siegel small divisors
 condition.  When $p>1$, we will show that generically there is no formal complex submanifold that can be attached
to $M$   if the CR singularity has
an elliptic component.   When $p>1$ and $M$ is a higher order perturbation of a product quadric  of
 $Q_{\gaa_h},
Q_{\gaa_s}$,
we will encounter  various  interesting situations.

By adapting Klingenberg's proof for $p=1$
and using a theorem of P\"oschel~\cite{Po86}, we will prove the following.
\begin{thm}\label{1conv}
Let $M$ be a germ of analytic submanifold that is an higher order perturbation of a product quadric $Q$ in $\cc^{2p}$ without elliptic components.
Assume that the eigenvalues $\mu_1,\ldots, \mu_p$,  $\mu_1^{-1}, \ldots, \mu_p^{-1}$ of $D\sigma(0)$ are distinct. Let $\e_h^2,\e_s^2=1$, $\nu_h:=\mu_h^{\e_h}$, $\nu_s:=\bar\mu_s^{\e_s}$ and $\nu_{s+s_*}:= \bar\nu_s^{-1}$.
Assume $\nu=(\nu_1,\ldots,\nu_p)$ is weakly non resonant and Diophantine in the sense of P\"oschel.
Then $M$ admits an attached complex submanifold $M_{\e}$. 
\end{thm}

Weak non  resonance is defined in \rea{nuqn0}, while Diophantine condition in the sense of P\"oschel is defined in \rea{omnu}.

Finally, we prove the convergence of {\it all} attached  formal 
submanifolds: 
\begin{thm}\label{allconv}
Let $M$ be as in \rta{1conv}. 
Suppose that the $2p$ eigenvalues of $\sigma$ are non-resonant.
 If 
  the eigenvalues of $\sigma$ satisfy a Bruno type condition,
 all attached formal submanifolds are convergent.
\end{thm}
The Brjuno-type condition,  defined in \re{bnI},  was introduced in~\cite{stolo-bsmf} for linearization on ideals.


%
%
%

\subsection{Notation}
We denote the identity map   by $I$ 
and 
by $LF$ the linear part at the origin of a mapping $F\colon(\cc^m,0)\to(\cc^n,0)$.
We also denote by 
$DF(z)$ or $F'(z)$, the Jacobian matrix of $F$
at 
$z$.
 By an analytic (or holomorphic) function, we mean a {\it germ} of analytic function at a point (which will be defined by the context) otherwise stated.
 We  denote by ${\cL O}_n$ (resp. $\widehat {\cL O}_n$, $\mathfrak M_n$, $\widehat{\mathfrak M}_n$) the space of germs of holomorphic functions of $\cc^n$ at the origin (resp. of formal power series in $\cc^n$,  holomorphic
germs, and formal germs vanishing at the origin). If $Q=(q_1,\dots, q_k)\in\nn^k$, then $|Q|=q_1+\dots+q_k$
and $x^Q=x_1^{q_1}\cdots x_k^{q_k}$.

\medskip
\noindent
{\bf Acknowledgment.} This joint work was completed while X.G. 
was visiting at  SRC-GAIA of
POSTECH.  He is grateful to Kang-Tae Kim for hospitality.
\setcounter{thm}{0}\setcounter{equation}{0}

\section{
CR singularities and deck transformations}
\label{secinv}

We consider a real submanifold $M$ of $\cc^n$.
Let $T_{x_0}^{(1,0)}M$ be the  space of tangent vectors of $M$ at $x_0$   of the form $\sum_{j=1}^na_j\f{\pd}{\pd z_j}$.   Let $M$ have dimension $n$. In this paper, we assume that
$T_{x_0}^{(1,0)}M$
 has the largest possible dimension $p=n/2$ at a given   point $x_0$.
  In suitable
holomorphic affine coordinates, we have $x_0=0$ and
\begin{equation}\label{variete-orig}
M\colon 
z_{p+j} = E_{j}(z',\bar z'), \quad 1\leq j\leq p.
\end{equation}
Here we   set $z'=(z_1,\ldots,z_p)$ and we will denote $z''=(z_{p+1},\ldots,z_{2p})$.
Also, the $E_j$ together with
their first order derivatives vanish at $0$.
 The tangent space $T_0M$ is then the $z'$-subspace.

  The main purpose of this section is to obtain some basic invariants and a relation between two families of involutions
 and the real analytic submanifolds which we want to normalize.


 Note that $M$ is totally real at $(z',z'')\in M$ if and only if 
 $C(z',\ov z')\neq0$, where $C(z',\ov{z'}):= \det(\frac{\partial E_i}{\partial \ov {z_j}})_{1\leq i,j\leq p}$.
We will assume that
 $C(z',\ov z')$ is not identically zero in any neighborhood of the origin.
 Then the zero set of $C$ on $M$, denoted by $M_{CRsing}$,  is called {\it CR singular set}
 of $M$, or the set of {\it complex tangents} of $M$.
   We assume that $M$ is real analytic. Then
 $M_{CRsing}$ is a possibly singular
  proper real analytic subset of $M$ that contains the origin.

\subsection{Existence of deck transformations and examples}

We first derive some quadratic invariants.
Applying a quadratic change of holomorphic coordinates, we obtain
\eq{variete-orig+}
E_j(z',\ov z')=h_j(z',\ov z')+q_j(\ov z')+O(|(z',\ov z')|^3).
\eeq
Here we have used the convention that if $x=(x_1,\ldots, x_n)$,
then $O(|x|^k)$ denotes a formal power series in $x$ without
terms of order  $<k$.
A biholomorphic map $f$ that preserves the form of the above submanifolds
$M$ and fixes the origin must preserve their complex tangent spaces
at the origin, i.e.  $z''=0$. Thus if $\tilde z$
denote the old coordinates and $z$ denote the new coordinates then $f$ has the form $$
 \tilde z'=  \mathbf{A}  z'+\mathbf{B}   z''+O(|  z|^2), \quad \tilde z''= \mathbf{U}  z''+O(|  z|^2).
$$
Here $\mathbf{A}$ and $\mathbf{U}$ are non-singular  $p\times p$ complex matrices.
Now $f(M)$ is given by
$$
\mathbf{U}  z''=h(\mathbf{A}  z',\ov{\mathbf{A}}\ov{  z}')+q(\ov{\mathbf{A}}\ov z')+O(|z|^3).
$$
We multiply the both sides by $\mathbf{U}^{-1}$ and
solve  for $z''$;  
the vectors of $p$ quadratic forms  $\{h(\tilde z',\ov {\tilde z'}), q(\tilde z')\}$ are transformed into
\eq{hhzp}
\{\hat h(z',z'),\hat q(\ov z')\}=\{\mathbf{U}^{-1}h(\mathbf{A}z',\ov{\mathbf{A}}\ov z'),\mathbf{\mathbf{U}}^{-1}q(\ov {\mathbf{A}}\ov z')\}.
\eeq
This shows that if $M$ and $\hat M$ are holomorphically equivalent,  their corresponding quadratic terms  are
equivalent via \re{hhzp}.
 Therefore, we obtain a holomorphic
invariant
$$
q_*=\dim_{\cc}\{z'\colon q_1(z')=\cdots =q_p(z')=0\}.
$$
We remark that when $M,\hat M$
are quadratic (i.e. when their corresponding $E, \hat E$ are homogeneous quadratic polynomials),
the equivalence relation \re{hhzp} implies that
$M, \hat M$ are linearly equivalent,
Therefore,  the above transformation
of $h$ and $q$
via $\mathbf{A}$ and $\mathbf{U}$ determines the classifications of  the quadrics under  local biholomorphisms as wells as under linear biholomorphisms.
We have shown that the two classifications for the quadrics are identical.


Recall that $M$ is real analytic.
Let us complexify such a real submanifold $M$ by replacing $\bar z'$ by $w'$
to obtain a complex   $n$-submanifold of $\cc^{2n}$, defined by
\begin{equation}\nonumber
{\mathcal M}\colon
\begin{cases}
z_{p+i} = E_{i}(z',w'), 
\\
w_{p+i} = \bar E_i(w',z'),\quad i=1,\ldots, p.\\
\end{cases}
\end{equation}
We  use $(z',w')$ as holomorphic coordinates of $\mathcal M$ and
 define the anti-holomorphic involution $\rho$ on it by
\begin{equation}\label{antiholom-invol}
\rho(z',w')=(\bar w',\bar z').
\end{equation}
  Occasionally we will also denote the above $\rho$ by $\rho_0$
for clarity.
We will  identify $M$ with a totally real and real analytic submanifold of $\cL M$ via embedding $z\to (z,\ov z)$.
We have $M={\mathcal M}\cap \text{Fix}(\rho)$ where $\text{Fix}(\rho)$ denotes the set of fixed points of $\rho$.
Let $\pi_1\colon \mathcal M\mapsto{\cc}^n$ be the restriction of the projection $(z,w)\to z$
and let $\pi_2$ be the restriction of $(z,w)\to w$. It is clear that $\pi_2=\ov{\pi_1\rho}$ on $\cL M$. Throughout the paper, $\pi_1,\pi_2,\rho$ are restricted on $\cL M$
unless stated otherwise.

%
%
%

Condition B that $q_*=0$,  introduced in section~\ref{sect1},  ensures that $\pi_1$ is a branched covering.
A necessary condition for $q_*=0$
is that functions $q_1(z')$, $ q_2(z'), \ldots, q_p(z')$
are linearly independent, since the intersection
of $k$ germs of holomorphic hypersurfaces at $0$ in $\cc^p$ has dimension at least
$p-k$. (See~\cite{Ch89}, p. 35;  \cite{Gunning2}[Corollary 8, p. 81].)

  When $\pi_1\colon\cL M\to\cc^{2p}$ is a branched covering,  we define a {\it deck transformation} on $\cL M$ for $\pi_1$ to
be   a germ of biholomorphic mapping $F$ defined at $0\in\cL M$ that satisfies  $\pi_1\circ F=\pi_1$. In other words, 
$F(z',w')=(z',f(z',w'))$  and
$$ 
E_i(z',w')=E_i(z',f(z',w')),\quad i=1,\dots, p.
$$ 

\begin{lemma}\label{2p1} Suppose that $q_*=0$.   Then $M_{CRsing}$ is a proper real analytic subset of $M$ and
 $M$ is totally real away from $M_{CRsing}$, i.e. the CR dimension of $M$
is zero.  Furthermore, $\pi_1$ is a
$2^p$-to-$1$ branched   covering.  The group of deck transformations  of $\pi_1$
consists of $2^\ell$ commuting  involutions with $0\leq \ell\leq p$.
\end{lemma}
\begin{proof}  Since $q^{-1}(0)=\{0\}$, then  $z'\to q(z')$ is a finite holomorphic map; see~\cite{Ch89}, p.~105.
 Hence its Jacobian determinant
is not identically zero.  In particular, 
$\det(\frac{\partial E_i}{\partial \ov {z_j}})_{1\leq i,j\leq p}$ 
is not identically zero. This shows that $M$ has CR dimension $0$.

Since $w'\to q(w')$ is a homogeneous quadratic
mapping of the same space  which vanishes only at the origin, then
$$
|q(w')|\geq c|w'|^2.
$$
We want to verify that
$\pi_1$ is a $2^p$--to--$1$
branched   covering.  Let $\Delta_r=\{z\in\cc\colon|z|<r\}$.
 We choose  $C>0$ such that
$\pi_1(z,w)=(z',E(z',w'))$
defines a proper and onto mapping
\eq{pi1M}
\pi_1\colon \cL M_1:=  \mathcal M\cap((\Delta_{\delta}^{p}\times \Delta_{\delta^2}^{p})
\times(\Delta_{C\delta}^{p}\times \Delta_{C\delta^2}^{p}))\mapsto \Delta_{\delta}^{p}\times \Delta_{\delta^2}^{p}.
\eeq
By Sard's theorem, the regular values of $\pi_1$ have the full measure. For each regular
value $z$,  $\pi_1^{-1}(z)$ has exactly $2^p$ distinct points (see  \cite{Ch89}, p.~105 and
p.~112). It is obvious that  $\mathcal M_1$ is smooth and connected.
We   fix a fiber $F_z$ of $2^p$ points.   Then 
 the group of deck transformations of $\pi_1$  acts on $F_z$ in such a way that if a deck transformation fixes
a point in $F_z$, then it must be the identity.  Therefore, the number of
deck transformations divides $2^p$
and each deck transformation has period $2^\ell$ with $0\leq \ell\leq p$.

We first show that each deck transformation $f$ of $\pi_1$
is an involution.
We know that  $f$ is periodic and
has the form $$z'\to z',\quad
w'\to \mathbf{A}w'+\mathbf{B}z'+O(2),$$ where $\mathbf{A},\mathbf{B}$ are  matrices.
Assume that $f$  has period $m$. Then
$\hat f(z',w')=(z',\mathbf{A}w'+\mathbf{B}z')$ satisfies $\hat f^m= I $ and $f$  is locally equivalent to $\hat f$; indeed $\hat fgf^{-1}=g$ for
$$
g=\sum_{i=1}^m(\hat f^{i})^{-1}\circ f^i.
$$
Therefore, it suffices to show that $\hat f$ is an involution.
 We have
$$
\hat f^m(z',w')=(z',\mathbf{A}^mw'+(\mathbf{A}^{m-1}+\cdots+\mathbf{A}+\mathbf{I})\mathbf{B}z').
$$
Since $f$ is a deck
transformation, then $E(z',w')$ is invariant under $f$.
Recall from \re{variete-orig+} that $E(z',\ov z')$
starts with quadratic terms of the form $h(z',\ov z')+q(\ov z')$.
Comparing quadratic terms in
$
E(z',w')=E\circ\hat f(z',w'),
$
we see  that the linear map $\hat f$ has   invariant functions
$$
z''=h(z',w')+q(w').
$$
We know that $\mathbf{A}^m=\mathbf{I}$. By the Jordan normal form,
we choose a linear transformation $\tilde w'= \mathbf{S}w'$  such that $\mathbf{S}\mathbf{A}\mathbf{S}^{-1}$ is the diagonal matrix $\diag \mathbf{a}$
with $\mathbf{a}=(a_1,\ldots, a_p)$.
In  $(z',\tilde w')$ coordinates, the
mapping $\hat f$ has the form $( z',\tilde w')\to(z',(\diag \mathbf{ a})\tilde w'+\mathbf{S}\mathbf{B} z')$. Now  $$\tilde h_j( z',\tilde w')+\tilde q_{j}(\tilde w'):=h_j( z',\mathbf{S}^{-1}\tilde w)+q_{ j}(\mathbf{S}^{-1}\tilde w')$$  are invariant under $\hat f$. Hence
  $\tilde q_j(\tilde w')$   are invariant under    $\tilde w'\mapsto(\diag \mathbf{a})\tilde w'$. Since the common zero set of $q_1(w'),\ldots, q_p(w')$ is the origin, then
$$
V=\{\tilde w'\in\cc^p \colon \tilde q(\tilde w')=0\}=\{0\}.
 $$
We conclude that $\tilde q(\tilde w_1,0,\ldots,0)$ is not identically zero; otherwise $V$ would contain the $\tilde w_1$-axis.
 Now
$ \tilde q((\diag \mathbf{a})\tilde w')=\tilde q(\tilde w')$,  restricted to $\tilde w'=(\tilde w_1,0,\ldots, 0)$,
implies that $a_1=\pm1$.  By the same argument, we get  $a_j=\pm1$
for all $j$.
This shows that $\mathbf A^2=\mathbf I$. Let us combine it with
$$
\mathbf{A}^m=\mathbf{I}, \quad (\mathbf{A}^{m-1}+\cdots+\mathbf{A}+\mathbf{I})\mathbf{B}=\mathbf{ 0}.
$$
If $m=1$, it is obvious that $\hat f=I$.
If $m=2^\ell>1$, then $(\mathbf{A}+\mathbf{I})\mathbf{B}=\mathbf{ 0}$. Thus $\hat
f^2(z',w')=(z',\mathbf{A}^2w'+(\mathbf{A}+\mathbf{I})\mathbf{B}z')=(z',w')$.  This shows that every deck transformation of $\pi_1$ is an involution.

For any two deck transformations $f$ and $g$, $fg$ is still a deck
transformation. Hence $(fg)^2= I $ implies that $fg=gf$.
\end{proof}

Before we proceed   to discussing
the deck transformations,
we give some examples. The first example turns out to be
a holomorphic equivalent form of a real submanifold that
admits the maximum number of deck transformations
and satisfies other mild conditions.
\begin{exmp}\label{sqex}
Let $\mathbf{B}=(b_{jk})$ be a non-singular $p\times p$ matrix. Let $M$ be defined by
\eq{reformu}
z_{p+j}=\left(\sum_k b_{jk}\ov z_k+R_j(z',\ov z')\right)^2, \quad 1\leq j\leq p,
\eeq
  where each $R_j(0,\ov z')$ starts with terms of order at least $2$.
Then $M$ admits $2^p$
deck transformations for $\pi_1$.   Indeed,  let $\mathbf E_1,\ldots,\mathbf E_{2^p}$ be the set of
 diagonal $p\times p$ matrices with $\mathbf E_j^2=
\mathbf I$, and let $\mathbf R$ be the column vector $(R_1,\dots, R_p)^t$.
For each $\mathbf E_j$ let us show that there is a deck transformation $(z',w')\to(z',\tilde w')$ satisfying
\eq{inveq}
\mathbf B\tilde w'+\mathbf R(z',\tilde w')=\mathbf E_j(\mathbf Bw'+\mathbf R(z',w')).
\eeq
  Since $\mathbf B$ is invertible, it has a unique solution
$$
 \tilde w'=\mathbf B^{-1}\mathbf E_j\mathbf Bw'+O(|z'|)+O(|w'|^2).
$$
Finally, $(z',w')\to (z',\tilde w')$ is an involution, as if $(z',w',\tilde w')=(z',w',f(z',w'))$ satisfy \re{inveq} if and only if $(z',f(z', w'),w')$,
substituting for $(z',w',\tilde w')$ in \re{inveq},
satisfy \re{inveq}.  \end{exmp}

We now present an example  to show that the deck
transformations can be destroyed by perturbations
when $p>1$. This is the major difference between real submanifolds with $p>1$ and the ones with $p=1$.
%
%
The   example shows that
the number of deck transformations can be reduced to any number $2^\ell$ by a higher order perturbation.
\begin{exmp}Let $N_{\gaa,\e}$ be a perturbation
 of $Q_\gaa$ defined by
$$
z_{p+j}=z_j\ov z_j+\gaa_j\ov z_j^2+\e_{j-1} \ov z_{j-1}^3,
\quad 1\leq j\leq p.
$$
Here $\e_j\neq0$ for all $j$,  $\e_0=\e_p$ and $z_0=z_p$.
Let $\tau$ be a deck transformation of $N_{\gaa,\e}$ for $\pi_1$. We know that $\tau$ has the form
$$
z_j'=z_j, \quad w_j'=A_j(z',w')+B_j(z',w')+O(|(z',w')|^3).
$$
Here $A_j$ are linear and $B_j$ are homogeneous
quadratic polynomials.
We then have
\ga\label{zjBj-}
z_jA_j(z',w')+\gaa_jA_j^2(z',w')
=z_jw_j+\gaa_jw_j^2, \\
\label{zjBj} z_jB_j(z',w')+2\gaa_j (A_jB_j)(z',w')+ \e_{j-1}A_{j-1}^3(z',w')=\e_{j-1} w_{j-1}^3.
\end{gather}
We know that $L\tau$ is a deck transformation for
$Q_\gaa$. Thus $a_j(w'):=A_j(0,w')=\pm w_j$.
Set $z_j=0$ in 
\re{zjBj} to get $a_j(w')|\e_{j-1}(w_{j-1}^3-a_{j-1}^3 (w'))$. Thus $a_{j-1}(w')=w_{j-1}$. Hence, the matrix
of $L\tau$ is triangular and its diagonal entries are
$1$. Since $L\tau$
is periodic then $L\tau= I$. Since $\tau$ is periodic,
then $\tau= I$.
\end{exmp}

Based the above  example,  we impose the  basic condition  D that  the branched   covering $\pi_1$  of $\cL M$
admits the maximum $2^p$ deck transformations.
\medskip

We first derive some  properties of
real submanifolds under 
 condition D.

\subsection{
Real submanifolds and 
Moser-Webster involutions}

The main result of this subsection is to show the equivalence
of classification of the real submanifolds  satisfying condition $D$ with that of  families of involutions $\{\tau_{11}, \ldots, \tau_{1p},\rho\}$. More precisely, condition J is not imposed.
 The relation between two classifications plays an important role
in the Moser-Webster theory for $p=1$. This
 will be the base of our approach to the normal form problems.

Let $\cL F$ be a family of holomorphic maps in $\cc^n$ with coordinates $z$.  Let $L\cL F$ denote the set of linear maps $z\to f'(0)z$ with $f\in \cL F$.
Let $\cL O_n^\cL F$ denote the set of germs of holomorphic functions
$h$ at $0\in\cc^n$
so that $h\circ f=h$ for each $f\in \cL F$.
Let $[\mathfrak M_n]_1^{L\cL F}$ be the subset of linear functions   of
$\mathfrak M_n^{L\cL F}$.
\begin{lemma}\label{sd}
Let $G$ be an abelian group of
 holomorphic $($resp. formal$)$ involutions fixing $0\in\cc^n$. Then $G$ has $2^\ell$ elements 
 and they are simultaneously
diagonalizable by a holomorphic $($resp. formal$)$ transformation.
If $k=\dim_\cc[\mathfrak M_n]_1^{LG}$ then  $\ell\leq n-k$.  Assume furthermore that $\ell=n-k$. In suitable holomorphic $(z_1,\ldots, z_n)$ coordinates, the group $G$ is generated by $Z_{k+1},\ldots, Z_n$ with
\eq{zjzjp}
Z_j\colon z_j'=-z_j, \quad z_i'=z_i, \quad i\neq j, \quad 1\leq i\leq n.
\eeq
In the $z$ coordinates,
 the set of convergent $($resp. formal$)$
  power series in $z_1,\ldots, z_k$, $z_{k+1}^2,\ldots,z_{n}^2$  is equal
 to $\cL O_n^{G}$
 $($resp. $\widehat{\cL O}_n^{G})$, and with $Z=Z_{n-k}\cdots Z_n$,
 \eq{Mn1g}
 [\mathfrak M_n]_1^{G}=[\mathfrak M_n]_1^{Z}, \quad
 \fix(Z)=\bigcap_{j=k+1}^n\fix(Z_j).\eeq
\end{lemma}
\begin{proof}
We first want to show that $G$ has $2^\ell$ elements.
  Suppose that it has more than one element and
we have already found a subgroup of $G$ that
has $2^{i}$ elements $f_1,\ldots, f_{2^i}$. Let $g$ be an element
in $G$ that is different from the $2^i$ elements. 
Since $g$ is an involution and commutes
with  each $f_j$, then
$$
f_1,\ldots, f_{2^{i}},\quad  gf_1,
\ldots, gf_{2^{i}}
$$
form a group of   $2^{i+1}$ elements.
We have proved that every finite subgroup of $G$ has exactly $2^\ell$
elements. Moreover, if $G$ is infinite then it contains a subgroup of $2^\ell$ elements   for every $\ell\geq0$. Let $\{f_1,\ldots, f_{2^\ell}\}$ be such a subgroup of $G$. It suffices to show that $\ell\leq n-k$.
We first linearize  all $f_j$ simultaneously.
We know that $Lf_1, \ldots, Lf_{2^\ell}$  commute pairwise. Note that $ I +f_1'(0)^{-1}f_1$ linearizes $f_1$.
Assume that $f_1$ is linear. Then $f_1=Lf_1$ and $Lf_2$ commute, and
$ I+f_2'(0)^{-1}f_2$ commutes with $f_1$ and linearizes $f_2$. Thus $f_j$ can be simultaneously linearized by a holomorphic (resp. formal) change of coordinates. Without loss of generality, we may assume that each $f_j$ is linear. We
want to diagonalize all $f_j$ simultaneously. Let $E_i^{\,1}$ and
$E_i^{-1}$ be the eigenspaces of $f_i$
with eigenvalues $1$ and $-1$, respectively. Since $f_i=f_j^{-1}f_if_j$, each eigenspace of $f_i$
is invariant under $f_j$. Then we can decompose
\eq{ccn}
\cc^n=\bigoplus_{(i_1, \ldots, i_s)}E_1^{i_1}\cap\cdots\cap E_s^{i_s}, \quad s= 2^{\ell}.
\eeq
Here $
(i_1,\ldots, i_s)$ runs over $\{-1,1\}^s$ with subspaces $E^{(i_1,\ldots, i_s)}:= E_{1}^{i_1}\cap\cdots\cap E_s^{i_s}$. On each of these subspaces,
$f_j= I$ or $- I$.
We are ready to choose
 a new basis for $\cc^n$ whose elements are in the subspaces. Under
the new basis, all $f_j$ are diagonal.

Let us rewrite \re{ccn} as
$
\cc^n=V_1\oplus V_2\oplus\cdots\oplus V_{d}.
$
Here  $V_j=E^{I_j}$ and $I_1=(1,\ldots, 1)$. Also,
$I_j\neq (1,\ldots, 1)$
and $\dim V_j>0$ for $j>1$. We have $\dim_\cc\fix (G)=\dim_\cc V_1=\dim_\cc[\mathfrak M_n]_1^{LG}=k$.  Therefore, $d-1\leq
n-\dim_\cc V_1\leq n-k$.
  We have proved that in suitable coordinates $G$ is contained in the group  generated by $Z_{k+1}, \ldots, Z_n$. The remaining assertions follow easily.
\end{proof}

We will need an elementary result about invariant functions.
\le{twosetin} Let  $Z_{k+1}, \ldots, Z_n$ be defined by \rea{zjzjp}.
Let $F=\{f_{k+1}, \dots, f_n\}$ be a family of germs of  holomorphic
mappings at the origin $0\in\cc^n$. Suppose that the family $F$ is holomorphically
equivalent to $\{Z_{k+1}, \dots, Z_n\}$.
Let $b_{1}(z), \dots, b_n(z)$ be germs of holomorphic functions that are invariant under $F$.
Suppose that  for $1\leq j\leq k$, $b_j(0)=0$ and the linear part of $b_j$ at the origin is $\tilde b_j$.
Suppose that
for $i>k$,
$b_i(z)=O(|z|^2)$ and the quadratic part of $b_i$ at the origin is $b_i^*$.  Suppose that $\tilde b_1, \dots, \tilde b_k$ are linear independent,  and that $b_{k+1}^*,
\dots, b_n^*$ are linearly independent modulo $\tilde b_1, \dots, \tilde b_k$,  i.e.
$$
\sum c_ib_i^*( z)=\sum d_j(z)\tilde b_j(z) + O(|z|^3)
$$
holds for some constants $c_i$ and formal power series $d_j$,
 if and only if all $c_i$ are zero.  Then  invariant functions
of $F$ are power series in $b_1, \dots, b_n$. Furthermore, $F$ is uniquely determined by $b_1,\ldots, b_n$.
The same conclusion holds if $F$ and $b_j$ are given by formal power series.
\ele
\begin{proof} Without loss of generality, we  assume that $F$ is $\{Z_{k+1}, \dots, Z_n\}$.  Hence, for all $j$, there is a formal power series $a_j$ such that $b_j(z)= a_j(z_1, \dots, z_k, z_{k+1}^2, \dots, z_n^2)$. Let us show that the map $w\to a(w)=(a_1(w),\ldots, a_n(w))$ is invertible.

By \rl{sd},
$\tilde b_1(z), \dots,\tilde b_k(z)$ are linear combinations of $z_1,\ldots, z_k$, and
vice versa. By \rl{sd} again, $b_{k+1}^*,
\dots, b_n^*$ are linear combinations of $z_{k+1}^2, \dots, z_n^2$ modulo $z_1,\ldots, z_k$. This shows
that
\eq{bjsz}\nonumber
b_i^*(z)=\sum_{j>k} c_{ij}z_j^2+\sum_{\ell\leq k} d_{i\ell}(z)\tilde b_\ell(z), \quad i>k.
\eeq
Since $b_{k+1}^*, \dots, b_n^*$ are linearly independent modulo $\tilde b_1,\dots, \tilde b_k$. Then $(c_{ij})$
is invertible; so is the linear part of $a$.

To show that $F$ is uniquely determined by its invariant functions, let $\tilde F$ be another such
family that is equivalent to $\{Z_{k+1},\dots, Z_n\}$.
 Assume that $F$ and $\tilde F$
have the same invariant functions.  Without loss of generality, assume that $\tilde F$
is $\{Z_{k+1},\dots, Z_n\}$.
 Then $z_1,\ldots, z_k$ are invariant by each $F_j$, i.e. the $i$th component of $F_j(z)$
is $z_i$ for $i\leq k$. Also $F_{j,\ell}^2(z)=z_\ell^2$ for $\ell>k$. We get $F_{j,\ell}=\pm z_\ell$. Since $z_\ell$
is not invariant by $\tilde F$, then it is not invariant by $F$ either. Then $F_{j_\ell,\ell}(z)=-z_{\ell}$ for
some $\ell_j>k$.  Since $F_{j_\ell}$ is equivalent to some   $Z_i$, the set of fixed points of $F_{j_\ell}$
is a hypersurface. This shows that $F_{j_\ell}=Z_{\ell}$. So the family $F$ is $\{Z_{k+1},\dots, Z_n\}$.
\end{proof}

 We now want to find a special set of generators
 for the deck transformations  and its basic properties, which will be important
 to our study of the normal form problems.
\begin{lemma}\label{sd1} Let $M$ be defined by \rea{variete-orig} and
\rea{variete-orig+} with $q_*=0$.
Suppose that  the group $\cL T_i$ of deck transformations
 of $\pi_i\colon \cL M\to\cc^p$  has
 exactly $2^p$ elements. Then the followings hold.
 \bppp
\item $\cL T_1$ is generated by $p$
 distinct involutions $\tau_{1j}$ such that $\fix(\tau_{11})$,
 $\dots$, $ \fix(\tau_{1p})$ are hypersurfaces
 intersecting transversally at $0$.
And $\tau_1=\tau_{11}\cdots\tau_{1p}$   is the unique deck transformation
 of which the set of fixed points
 has dimension $p$. Moreover, $\fix(\tau_1)=\bigcap\fix(\tau_{1j})$.
 \item $\cL O_n^{\cL T_1}$ $($resp. $\widehat{\cL O}_n^{\cL T_1})$
is precisely  the set of convergent $($resp. formal$)$ power series in
 $z'$ and $ E(z',w')$. $\cL O_n^{\cL T_2}$ $($resp. $\widehat{\cL O}_n^{\cL T_2})$
is the set of convergent $($resp. formal$)$ power series in
 $w'$ and $ \ov E(w',z')$. In particular, in $(z',w')$ coordinates   of $\cL M$, $\cL T_1$ and $\cL T_2$ satisfy
 \ga\label{i1i1}
 [\mathfrak M_n]_1^{L\cL T_1}\cap[\mathfrak M_n]_1^{L\cL T_2}
 =\{0\},\\  
 \dim\fix(\tau_i)=p, \quad \fix(\tau_1)\cap\fix(\tau_2)=\{0\}.\nonumber
 \end{gather}
 Here $[\mathfrak M_n]_1$ is the set of linear functions in $z',w'$ without constant terms.
  \eppp
\end{lemma}
\begin{proof} (i).
%
 Since $z_1,\ldots, z_p$ are invariant under deck transformations of $\pi_1$, we have $p'=\dim_\cc  [\cL O_n]_1^{L\cL T_1}\geq p$. By \rl{sd}, $\pi_1$ has at most $2^{2p-p'}$
deck transformations. Therefore, $p'=p$.
By \rl{sd} again,  we may assume that
 in suitable
 $(\xi,\eta)$ coordinates, the deck transformations are generated by $Z_{p+1},\dots, Z_{2p}$ defined by \re{zjzjp} in which $z=(\xi,\eta)$.
It follows that $Z=Z_{p+1}\cdots Z_{2p}$ is the unique  deck transformation of $\pi_1$,  of which the set of fixed points has dimension $p$.

(ii). We have proved that in $(\xi,\eta)$ coordinates the deck transformations are generated by the above $Z_{p+1}, \ldots,Z_{2p}$. Thus, the
 invariant holomorphic functions of  $Z_{p+1}, \ldots, Z_{2p}$ are precisely the holomorphic
functions in $\xi_1, \ldots, \xi_p, \eta^2_{1}, \ldots, \eta^2_{p}$.
Since $z_1, \ldots, z_p$ and $E_i(z',w')$ are invariant under deck transformations,
then on $\cL M$
\eq{zpfx}
z'=f(\xi, \eta_1^2, \ldots, \eta_p^2), \quad
E(z',w')=g(\xi,\eta_1^2, \ldots, \eta_p^2).
\eeq
Since $(z',w')$ are local coordinates of $\cL M$,
  the differentials of $z_1,\ldots, z_p$ under any coordinate
  system of $\cL M$
 are linearly independent. Computing the differentials of $z'$ in variables $\xi,\eta$ by using \re{zpfx}, we see that
the mapping
 $\xi \to f(\xi,0)$ is a local biholomorphism. Expressing both sides of
  the second identity in \re{zpfx} as power series in $\xi,\eta$, we obtain
\gan
E(f(\xi,0),w')=g(\xi,\eta_1^2,\ldots, \eta_p^2)+O(|(\xi,\eta)|^3).
\end{gather*}
We  set $\xi=0$, compute the left-hand
side, and rewrite the identity as
\ga\label{g0et} g(0,\eta_1^2,\ldots,\eta_p^2)
=q(w')+O(|(\xi,\eta)|^3).
\end{gather}
 As coordinate systems, $(z',w')$ and $(\xi,\eta)$
 vanish at $0\in\cL M$. We now use $(z',w')=O(|(\xi,\eta)|)$. By \re{zpfx}, $f(0)=g(0)=0$ and
 $g(\xi,0)=O(|\xi|^2)$.
Let us   verify that  the linear parts of $g_1(0,\eta), \ldots, g_p(0,\eta)$ are
linearly independent. Suppose that
$\sum_{j=1}^pc_jg_j(0,\eta)=O(|\eta|^2)$. Replacing  $\xi,\eta$ by $O(|(z',w')|)$ in \re{g0et} and setting $z'=0$, we obtain
$$
\sum_{j=1}^pc_jq_j(w')=O(|w'|^3),\quad i.e. \quad\sum_{j=1}^pc_jq_j(w')=0.
$$
As remarked after condition B was introduced, $q_*=0$ implies that $q_1(w'), \ldots, q_p(w')$ are linearly independent. Thus all $c_j$ are $0$. We have verified that $\xi\to f(\xi,0)$
is biholomorphic near $\xi=0$. Also $\eta\to g(0,\eta)$ is biholomorphic near $\eta=0$
and  $g(\xi,0)=O(|\xi|^2)$.
Therefore, $(\xi,\eta)\to(f,g)(\xi,\eta)$ is invertible near $0$. By solving \re{zpfx},
 the functions $\xi,\eta_1^2,\ldots, \eta_p^2$ are expressed as power series
  in $z'$ and $E(z',w')$.

It is clear that $z_1,\ldots, z_p$ are invariant under $\tau_{1j}$. From linearization of $\cL T_1$,
we know that the space of invariant linear functions of $L\cL T_1$ is the same
as the space  of linear invariant functions of $L\tau_1$, which
has dimension $p$. This shows that $z_1,\ldots, z_p$ span the space of linear invariant functions of $L\tau_1$. Also $w_1,\ldots, w_p$ span
the space of linear invariant functions of $L\tau_2$. We obtain $[\mathfrak M_n]_1^{L\cL T_1}\cap[\mathfrak M_n]_1^{L\cL T_2}=\{0\}$.  We have verified   \re{i1i1}.

In view of the linearization of $\cL T_1$ in (i), we obtain $\dim\fix(\tau_1)=\dim\fix(\cL T_1)=p$.
Moreover,  $\fix(\tau_i)$ is a smooth submanifold of which the tangent space at the origin
is   $\fix(L\tau_i)$.
We choose a basis $u_1,\ldots, u_p$ for  $\fix(L\tau_1)$. Let $v_1,\ldots, v_p$ be any $p$
vectors such that $u_1,\ldots, u_p$,  $v_1,\ldots, v_p$ form a basis of $\cc^n$. In new coordinates defined by
$\sum\xi_iu_i+\eta_iv_i$, we know that linear invariant functions of $L\tau_1$ are spanned by $\xi_1,\ldots, \xi_p$. The linear invariant functions in $(\xi,\eta)$
that are invariant by $L\tau_2$ are spanned by
$f_j(\xi,\eta)=\sum_k (a_{jk}\xi_k+b_{jk}\eta_k)$
for $1\leq j\leq p$. Since $[\mathfrak M_n]^{L\tau_1}\cap[\mathfrak M_n]^{L\tau_2}=\{0\}$,
 then $\xi_1,\ldots, \xi_p, f_1,\ldots$,   $ f_p$ are linearly
independent. Equivalently, $(b_{jk})$ is non-singular.  Now $\fix(L\tau_2)$ is spanned by vectors  $\sum_k (a_{jk}u_k+b_{jk}v_k)$.
This shows that $\fix(L\tau_1)\cap\fix(L\tau_2)=\{0\}$. Therefore, $\fix(\tau_1)$ intersects
 $\fix(\tau_2)$  transversally at the origin and the intersection must
be the origin.
\end{proof}

Note that the proof of the above lemma actually gives us a more general result.

\begin{cor} \label{sd2} 
Let $\mathfrak{I}$ be a group of commuting holomorphic $($formal$)$
involutions on $\cc^n$.
 \bppp
\item $\fix(L  \mathfrak{I})=\{0\}$ if and only if  $[\mathfrak M_n]_1^{L  \mathfrak{I}}$  has dimension $0$.
 \item Let $\widetilde{ \mathfrak{ I}}$ be another family of commuting holomorphic $($resp. formal$)$ involutions such that
 $[\mathfrak M_n]^{L  \mathfrak{I}}\cap[\mathfrak M_n]^{L\widetilde{  \mathfrak{I}}}=\{0\}$. Then
 $\fix(L  \mathfrak{I} )\cap\fix(L\widetilde{  \mathfrak{I}})=\{0\}$. Moreover,  $\fix(  \mathfrak{I} )\cap\fix(\widetilde{  \mathfrak{I}})=\{0\}$ if $  \mathfrak{I}$
 and $\widetilde{  \mathfrak{I}}$ consist of convergent involutions.
   \eppp
\end{cor}

In view of  \rl{sd1},
we will refer  to the family
 $$
 \cL T_1:=\{\tau_{11}, \ldots, \tau_{1p}\},\quad
 \cL T_2:=\{\tau_{21}, \ldots, \tau_{2p}\},\quad  \rho $$
as the Moser-Webster involutions, where $\tau_{2j}=\rho\tau_{1j}\rho$.
%
  The significance of the two sets of
involutions is the following proposition that transforms the normalization of the real manifolds
into that 
of two families of commuting involutions.

 For clarity, recall the anti-holomorphic involution $\rho_0\colon(z',w')\to(\ov{w'},\ov{z'})$.
\begin{prop}\label{inmae}
Let $M$ and $\widetilde M$ be two real analytic submanifolds of the form \rea{variete-orig}
and \rea{variete-orig+}
 that
admit Moser-Webster involutions $\{\cL T_1,\rho_0\}$ 
and $\{\widetilde{\cL T_1},\rho_0\}$,
respectively. Then $M$ and $\widetilde M$ are holomorphically equivalent
if and only if $\{\cL T_1,\rho_0\}$ and $\{\widetilde{\cL T_1},\rho_0\}$
are holomorphically   equivalent, i.e. if there is a biholomorphic map $f$ commuting with $\rho_0$
such that
 $f\tau_{1j}f^{-1}=\tilde\tau_{1i_j}$ for $1\leq j\leq p$. Here $\{i_1,\ldots, i_p\}=\{1,\ldots, p\}$.

Let  $\cL T_1=\{\tau_{11},\ldots,\tau_{1p}\}$ be a family of $p$ distinct
commuting   holomorphic
involutions.  Suppose that $\fix(\tau_{11}), \ldots, \fix(\tau_{1p})$ are hypersurfaces
 intersecting transversely at the origin.
 Let $\rho$ be an anti-holomorphic involutions and let $\cL T_2$ be the family of
involutions $\tau_{2j}=\rho\tau_{1j}\rho$ with $1\leq j\leq p$.
Suppose that
\eq{ilii+}
 [\mathfrak M_n]_1^{L\cL T_1}\cap[\mathfrak M_n]_1^{L\cL T_2}
 =\{0\}.
 \eeq
There exists a real analytic real $n$-submanifold
\eq{realM2}
M\subset\cc^{2p}
\colon z_{p+j}=A_j^2(z',\bar z'), \quad1\leq j\leq p
\end{equation} such that the set
of Moser-Webster involutions
$\{\widetilde{\cL T_1},\rho_0\}$ of $M$
is holomorphically equivalent to
 $\{\cL T_1,\rho\}$.
\end{prop}
\begin{proof} We  recall from \re{pi1M} the branched  covering
\eq{pi1M+}\nonumber
\pi_1\colon \cL M_1:= \mathcal M\cap((\Delta_{\delta}^{p}\times \Delta_{\delta^2}^{p})
\times(\Delta_{C\delta}^{p}\times \Delta_{C\delta^2}^{p}))\longrightarrow \Delta_{\delta}^{p}\times \Delta_{\delta^2}^{p}.
\eeq
Here $C\geq1$.
  Let $\pi_1$
be restricted to $\cL M_1$.  
  Then $\pi_2=\ov{\pi_1\circ\rho}$ is defined on $\rho(\cL M_1)$. Note that
\eq{pi2M}\nonumber
\pi_2\colon \cL \rho(\cL M_1)
\longrightarrow \Delta_{\delta}^{p}\times \Delta_{\delta^2}^{p}.
\eeq
We have $\pi_1^{-1}(z)\cap \fix(\rho)=\{(z,\ov z)\}$ for $z\in M$
and $\pi_1(\fix(\rho))=M$.
Let   $\cL B_0\subset\Delta_{\delta}^{p}\times \Delta_{\delta^2}^{p}$
be the branched  locus. Take $\cL B=\pi_1^{-1}(\cL B_0)$.
We will denote by $\widetilde {\cL M}_1,\tilde{\cL B}$ and $\tilde{\cL B_0}$ the corresponding
  data for $\widetilde M$. Here $\widetilde {\cL M}_1$ is an analogous
   branched  covering over $\pi_1(\widetilde{\cL M}_1)$. We assume that
  the latter contains $f(\pi_1(\cL M_1))$ if $\widetilde M$ is equivalent to $M$
  via $f$.

Assume that $f$ is a biholomorphic map sending $M$ into $\widetilde M$. Let
$f^c$ be the restriction of biholomorphic map
$f^c(z,w)=(f(z),\ov f(w))$ to $\cL M$. Let $M$ be defined by $z''=E(z',\ov z')$
 and $\widetilde M$ be defined by $z''=\tilde E(z',\ov z')$.
 By $f(M)\subset\widetilde M$, $f=(f',f'')$ satisfies
$$
f''(z',E(z',\ov z'))=\tilde E(f'(z', E(z',\ov z')),\ov{f'}(\ov z',\ov E(\ov z',z'))).
$$
Using the defining equations for $\cL M$,  we get
$f^c(\cL M)\subset\widetilde{\cL M}$
and $\rho f^c=f^c\rho$ on $\cL M\cap\rho(\cL M)$.
We will also assume that $ f^c(\cL M_1)$ is contained in $\widetilde{\cL  M}_1$. It is clear that $f^c$
sends a fiber $\pi_1^{-1}(z)$ onto the fiber $\pi_1^{-1}(f(z))$ for $z\in \Om=
\cL \pi_1(\cL M_1)\setminus(\cL B_0\cup f^{-1}
(\tilde B_0))$, since the two fibers have the same
number of points and $f$ is injective.
 Thus $f^c\tau_{1j}=\tilde\tau_{1i_j}f^c$ on $\pi_1^{-1}(\Om)$. Here $i_j$ is of course locally determined on
 $\pi_1^{-1}(\Om)$.
 Since $\cL B$ has positive codimension in $\cL M_1$
then $\cL M_1\setminus\cL B$ is connected. Hence $i_j$ is well-defined on $\pi_1^{-1}(\Om)$.
Then $f^c\tau_{1j}=\tilde\tau_{1i_j} f^c$ on $\cL M_1\setminus B$. This
shows that
$f^c$ conjugates simultaneously the deck transformations of ${\mathcal M}$ to the deck transformations of $\widetilde {\mathcal M}$ for $\pi_1$.  The same conclusion holds for
 $\pi_2$.

Conversely, assume that there is a biholomorphic map $g\colon \cL M\to \widetilde{\cL  M}$
such that $\rho g=g\rho$ and $g\tau_{1i}=\tilde\tau_{1j_i}g$. Since $\tau_{11}, \ldots, \tau_{12^p}$ are distinct
and $M_1\setminus\cL B$ is connected, then $\bigcup_{j\neq i}\{ x\in\cL M_1\setminus\cL B\colon\tau_{1i}(x)=\tau_{1j}(x)\}$
is a complex subvariety  of positive codimension in
$\cL M_1\setminus\cL B$. Its image under the proper projection $\pi_1$ is  a subvariety of positive codimension in $
\Delta_{\delta}^p\times\Delta_{\delta^2}^p\setminus\cL B_0$. This shows that the latter contains a non-empty open subset $\omega$ such that $\{\tau_{11}(x),\ldots,\tau_{12^p}(x)\}=\pi_1^{-1}\pi_1(x)$ has $2^p$ distinct points
 for each $\pi_1(x)\in\omega$. Therefore, $\tau_{11},\ldots, \tau_{12^p}$ are
  all deck transformations of $\pi_1$ over $\om$. Hence they are
   all deck transformations of $\pi_1\colon\cL M_1\setminus\cL B\to\Delta_{\delta}^p\times\Delta_{\delta^2}^p\setminus\cL B_0$, too.
 This shows that $\pi_1^{-1}(\pi_1(x))=\{\tau_{1j}(x)\colon1\leq j\leq 2^p\}$ for $x\in\cL M_1\setminus\cL B$.
Now, $g$ sends $\tau_{1j}(x)$ to $\tilde\tau_{1i_j}(g(x))$ for each $j$. Hence
  $f(z)=\pi_1g\pi_1^{-1}(z)$
is well-defined and holomorphic for $z\in\Delta_{\delta}^p\times\Delta_{\delta^2}^p\setminus\cL B_0$.
By the Riemann extension for bounded holomorphic functions, $f$ extends to a holomorphic mapping, still denoted by $f$,
which is defined near the origin.   We know that
 $f$ is invertible and  in fact the inverse
  can be obtained by extending  the mapping $z\to
 \pi_1g^{-1}\pi_1(z)$.
 If $z=(z',E(z',w'))\in M$, then $w'=\ov z'$
and
$f(z)=\pi_1g\pi_1^{-1}(z)=\pi_1g(z,\ov z)$ with $(z,\ov z)\in \fix(\rho)$. Since $\rho g=g\rho$,
then $g(z,\ov z)\in \fix(\rho)$. Thus $f(z)=\pi_1g(z,\ov z)\in \widetilde M$.

Assume that $\{\tau_{1j}\}$ and $\rho$
are germs of involutions defined at the origin of $\cc^n$. Assume that they satisfy the conditions in the proposition.   From 
\rl{sd} it follows that $\tau_{11}, \ldots, \tau_{1p}$ generate a group of $2^p$ involutions, while the $p$ generators are the only elements of which each
fixes a hypersurface pointwise.
To realize them as deck transformations of the complexification of a real analytic submanifold, we apply   \rl{sd}
to find  a
coordinate map $(\xi,\eta)\to \phi(\xi,\eta)=(A,B)(\xi,\eta)$  such that invariant holomorphic functions of $\{\tau_{1j}\}$ are precisely holomorphic functions in
$$
 z'=(A_1(\xi,\eta),\ldots,
A_p(\xi,\eta)), \quad  z''=(B_1^2(\xi,\eta), \ldots, B_p^2(\xi,\eta)).
$$
Note that $B_j$ is skew-invariant under $\tau_{1j}$ and
is invariant under $\tau_{1i}$
for $i\neq j$ and $A$ is invariant under all $\tau_{1j}$.
Set $$
w_j'=\ov{A_j\circ\rho(\xi,\eta)}, \quad w''_j=\ov{B_j^2\circ\rho(\xi,\eta)}.
$$
Since $\tau_{2j}=\rho\tau_{1j}\rho$, the holomorphic functions invariant under all $\tau_{2j}$ are precisely the holomorphic functions in the above $w',w''$. We now draw conclusions for the linear parts of invariant functions and
involutions. Since $\phi$ is biholomorphic, then
$LA_1,\ldots, LA_p$ are linearly independent. They are also invariant under $L\tau_{1j}$. Since $\tau_{2j}=\rho\tau_{1j}\rho$,  the $p$ functions $L\ov {A_i\circ\rho}$
 are linearly independent and   invariant under $L\tau_{2j}$. Thus $$LA_1,\ldots, LA_p,\
 L\ov {A_1\circ\rho}, \ldots, L\ov{A_p\circ\rho}$$ are linearly independent,
since $[\mathfrak M_n]_1^{L\cL T_1}\cap[\mathfrak M_n]_1^{L\cL T_2}=\{0\}$. This shows that
  the map $(\xi,\eta)\to  (z',w')=(A(\xi,\eta),\ov{A\circ\rho(\xi,\eta)})$
  has
an inverse $(\xi,\eta)=\psi(z',w')$. Define
$$
M\colon z''=(B_1^2,\ldots, B_p^2)\circ\psi(z',\ov z').
$$
The  complexification of $M$ is given by
$$
\cL M\colon
z''=(B_1^2,\ldots, B_p^2)\circ\psi(z',w'),\quad
w''=(\ov B_1^2,\ldots, \ov B_p^2)\circ\ov\psi(w',  z').
$$
Note that $\phi\circ\psi(z',w')=(z',B\circ\psi(z',w'))$ is biholomorphic. In particular, we can write
$$
B_j^2\circ\psi(z',\ov z')=h_j(z',\ov z')+q_j(\ov z')+b_j(z')+O(|(z',\ov z')|^3).
$$
Here $q_j(\ov z')=\tilde q^2_j(\ov z')$, and $\tilde q(w')$
is the linear part of $w'\to B\circ\psi(0,w')$. Therefore, $|q(w')|
\geq c|w'|^2$ and $q_*=0$.  By \rl{2p1}, $\pi_1\colon\cL M\to\cc^p$
is a $2^p$-to-$1$ branched covering defined near
$0\in\cL M$.
Since  $B^2$ is invariant by $\tau_{1j}$, then $z''=B^2\circ\psi(z',w')$ is invariant by $\psi^{-1}\tau_{1j}\psi(z',w')$.
Also $A$ is invariant under $\tau_{1j}$. Then $z'=
A\circ\psi(z',w')$ is invariant by $\psi^{-1}\tau_{1j}\psi(z',w')$.
 This show that $\{\psi^{-1}\tau_{1j}\psi\}$ has the same invariant functions
as of the deck transformations of $\pi_1$. By \rl{twosetin}, $\{\psi^{-1}\tau_{1j}\psi\}$
agrees with the set of deck transformations of $\pi_1$.
For $\rho_0(z',w')=(\ov{w'},\ov{z'})$ we have $\rho_0\psi^{-1}=\psi^{-1}\rho.$
This shows that $M$ is a realization for    $\{\tau_{11},\ldots, \tau_{1p},\rho\}$.
\end{proof}

\begin{rem}
We choose the realization in such a way that $z_{p+j}$ are square functions.
This particular holomorphic equivalent form of $M$ will be crucial
to study the asymptotic manifolds in section~\ref{secideal}.
\end{rem}

Next we want to compute the deck transformations for a product quadric.
We will  first recall the Moser-Webster involutions for elliptic and hyperbolic complex tangents.  We will then
compute the deck transformations for complex tangents of complex type.

%
%
%

The Moser-Webster theory deals with the case $p=1$ for a real
analytic surface
\eq{z2z1}
\nonumber
z_2=|z_1|^2+\gaa (z_1^2+\ov z_1^2)+O(|z_1|^3), \quad \text{or}\quad  z_2=z_1^2+\ov z_1^2+O(|z_1|^3).
\eeq
Here $\gaa \geq0$ is the Bishop invariant of $M$.
One of most important properties of the Moser-Webster theory
is the existence of the above mentioned deck transformations.
  When $\gamma \neq0$, there is a pair of Moser-Webster involutions $\tau_{1}, \tau_2$ with $\tau_2=\rho\tau_1\rho$
such that
 $\tau_1$ generates the deck transformations of $\pi_1$.    In fact, $\tau_1$ is the only possible non-trivial deck
transformation of $\pi_1$.
When $\gamma \neq1/2$,  in suitable coordinates
their composition $\sigma=\tau_1\tau_2$ is  of the form
$$
\tau\colon\xi'=\mu\xi+O(|(\xi,\eta)|^2), \quad \eta'=\mu^{-1}\eta+O(|(\xi,\eta)|^2). 
$$
 Here $\rho(\xi,\eta)=(\ov\eta,\ov\xi)$ when $0<\gaa<1/2$, and $\rho(\xi,\eta)
=(\ov\xi,\ov\eta)$ when $\gaa>1/2$.
  When the complex tangent is elliptic 
  and $0<\gaa<1/2$, $\sigma$ is
{\it hyperbolic} with $\mu>1$;  when the complex tangent is  hyperbolic, i.e. $1/2<\gaa\leq\infty$,  then
 $\sigma$ is {\it elliptic } with $|\mu|=1$.
When the complex tangent is parabolic,  the linear part of $\sigma$ is not diagonalizable and $1$ is the eigenvalue.

We will see later that with $p\geq2$, there is yet another simple
model that is not in the product. This is the quadric
in $\cc^4$ defined by
\eq{z3z4} Q_{\gaa_s}\colon
  z_3=z_1\ov z_2+\gaa_s\ov z_2^2+(1- \gaa_s)z_1^2, \quad z_4=\ov z_3.
\eeq
Here $\gaa_s$ is a complex number.
We will, however, exclude $\gaa_s=0$ or equivalently $\gaa_s=1$   by condition B. We also exclude $\gaa_s=1/2$
by condition J. Note that $\gaa_s=1/2$ does not correspond to a product  Bishop  quadrics either,   by
examining the CR singular sets.
Under these mild non degeneracy conditions, we will show that $\gaa_s$ is an invariant when it is normalized to the range
     \eq{gsra}
\RE\gaa_s\leq 1/2, \quad \IM\gamma_s\geq0, \quad\gamma_s\neq0.
  \eeq
In this case, the complex  tangent is said of {\it complex} type.

We have introduced the types of the complex tangent at the origin. Of course a product of quadrics, or a product quadric,
can exhibit a combination of the above basic $4$ types.
We will see soon that 
quadrics have other invariants when $p>1$. Nevertheless, in our results,
the above invariants
that describe the types of the complex tangent will
play a   major role in the convergence or divergence of normalizations. 

Let us first recall involutions in \cite{MW83} where the complex tangents
are elliptic (with non-vanishing Bishop invariant) or hyperbolic.
When $\gaa>0$, the non-trivial
 deck transformations of
\eq{z2z12}
\nonumber
Q_{\gaa}\colon
z_2=|z_1|^2+\gaa (z_1^2+\ov z_1^2)
\eeq
 for $\pi_1,\pi_2$ are $\tau_1,\tau_2$,  respectively. They are
\eq{tau10e}
\tau_1\colon  z_1' = z_1, \quad w_1'=-w_1 -\gamma ^{-1}z_1; \quad \tau_2=\rho\tau_1\rho
\eeq
with $\rho$ being defined by (\ref{antiholom-invol}). Here the formula is valid for $\gaa =\infty$
(i.e. $\gaa ^{-1}=0$).  Note that $\tau_1$ and $\tau_2$
do not commute and $\sigma=\tau_1\tau_2$ satisfies
 \eq{} \nonumber
 \sigma^{-1}=\tau_i\sigma\tau_i=\rho\sigma\rho, \quad \tau_i^2= I,\quad\rho^2= I.
 \eeq
When the complex tangent is not parabolic, the eigenvalues of $\sigma$ are   $\mu,\mu^{-1}$ with
$\mu=\la^2$ and
$
\gaa\la^2 -\la +\gaa=0.
$
For the elliptic complex tangent, we can choose a solution  $\la>1$, and in suitable coordinates we obtain
\begin{gather}
\label{tau1e}
\tau_1\colon\xi'=\la\eta+O(|(\xi,\eta)|^2), \quad \eta'=\la^{-1}\xi+O(|(\xi,\eta)|^2),\\
\tau_2=\rho\tau_1\rho, \nonumber
\quad 
\nonumber
\rho(\xi,\eta)=(\ov\eta,\ov\xi),\\
\sigma\colon\xi'=\mu\xi+O(|(\xi,\eta)|^2), \quad \eta'=\mu^{-1}\eta+O(|(\xi,\eta)|^2),\quad \mu=\la^2.\nonumber
\end{gather}
When the complex tangent is hyperbolic, i.e.
$1/2<\gaa\leq\infty$,    $\tau_i$ and  $\sigma$ still have the above form, while $|\mu|=1=|\la|$ and
\eq{rhohy}
\nonumber
\rho(\xi,\eta)
=(\ov\xi,\ov\eta).
\eeq
When the complex tangent is parabolic, i.e. $\gaa=1/2$, the pair of involutions
still exists. However, $L\sigma$ is not diagonalizable and $1$ is its only eigenvalue. 
 We recall from ~\cite{MW83} that
\eq{gaala1}
\gaa =\f{1}{\la+\la^{-1}}.
\eeq

 For the complex type, new situations arise.   Such a quadric has the form
 \eq{z3z42} Q_{\gaa_s}\colon  z_3=z_1\ov z_2+\gaa_s\ov z_2^2+(1- \gaa_s)z_1^2, \quad z_4=\ov z_3.
\eeq
Here $\gaa_s$ is a complex number.
Let us first check  that such a quadric   is not the product of two Bishop quadrics~:
 Its CR singular set is defined by
$$
(z_1+2\gaa_s \ov z_2)(z_2+2(1-\ov\gaa_s)\ov z_1)=0,
$$
  which is the union of a complex line and a real surface when $\gaa_s=0,1$, or a union of  two totally real surfaces.
The CR singular set of  a quadric defined by $z_3=|z_1|^2+\gaa_1(z_1^2+\ov z_1^2)$
and $z_4=|z_2|^2+\gaa_2(z_2^2+\ov z_2^2)$ is given by
$$
(z_1+2\gaa_1 \ov z_1)(z_2+2\gaa_2\ov z_2)=0.
$$
 It is the union of two Bishop  surfaces when $\gaa_1\neq1/2$ and $\gaa_2\neq1/2$, and it contains
a submanifold of dimension $3$ otherwise.

By condition B, we know that $\gaa_s\neq0,1$.   Let us compute the deck transformations of the complexification of  \re{z3z42}.
 According to   \rl{sd1} (i), the deck transformations for $\pi_1$ are generated by two involutions
\ga\tau_{11}
\label{tau1112}
\nonumber
\colon \begin{cases}
z_1'=z_1,\\
z_2'=z_2,\\
w_1'=-w_1-(1-\ov\gaa_s)^{-1} z_2,\\
w_2'=w_2;
\end{cases}\quad
\tau_{12}\colon \begin{cases}
z_1'=z_1,\\
z_2'=z_2,\\
w_1'=w_1,\\
w_2'=-w_2- \gaa_s^{-1}z_1.
\end{cases}
\end{gather}
We still have $\rho$ defined by (\ref{antiholom-invol}).  Let $\tau_{2j}=\rho\tau_{1j}\rho$. Then $\tau_{21},\tau_{22}$
generate the deck transformations of $\pi_2$. Note that
\ga
\nonumber
\tau_{21}\colon \begin{cases}
z_1'=-z_1- (1-\gaa_s)^{-1} w_2,\\
z_2'=z_2,\\
w_1'=w_1,\\
w_2'=w_2;
\end{cases}\quad
\tau_{22}\colon \begin{cases}
z_1'=z_1,\\
z_2'=-z_2-\ov\gaa_s^{-1} w_1,\\
w_1'=w_1,\\
w_2'=w_2.
\end{cases}
\end{gather}
Recall that $\tau_i=\tau_{i1}\tau_{i2}$ is the unique deck transformation of $\pi_i$ that has the smallest dimension of the set of fixed-points among all deck transformations. They are
\ga\tau_{1}\colon \begin{cases}
z_1'=z_1,\\
z_2'=z_2,\\
w_1'=-w_1-(1-\ov\gaa_s)^{-1} z_2,\\
w_2'=-w_2-\gaa_s^{-1} z_1;
\end{cases}
\tau_{2}\colon   \begin{cases}
z_1'=-z_1- (1-\gaa_s)^{-1} w_2,\\
z_2'=-z_2-\ov\gaa_s^{-1} w_1,\\
w_1'=w_1,\\
w_2'=w_2.
\end{cases}
\label{tau12s}
\end{gather}
  And $\tau_1\tau_2$ is given by
\ga\sigma_s\colon \begin{cases}
z_1'=-z_1- (1-\gaa_s)^{-1}w_2,\\
z_2'=-z_2- \ov\gaa_s^{-1}w_1,\\
w_1'= (1-\ov\gaa_s)^{-1} z_2+((\ov\gaa_s-\ov\gaa_s^2)^{-1}-1)w_1,\\
w_2'=\gaa_s^{-1}z_1+((\gaa_s-\gaa_s^2)^{-1}-1)w_2.
\end{cases}
\nonumber
\end{gather}

In contrast to the elliptic and hyperbolic cases, $\tau_{11}$  and $\rho\tau_{11}\rho$ commute; in other words, $\tau_{11}\rho\tau_{11}\rho$ is actually an involution.
And $\tau_{12}$ and $\rho\tau_{12}\rho$ commute, too.  However, $\tau_{11}$ and $ \tau_{22}$ do not commute,  and $\tau_{12}, \tau_{21}$ do not commute either.  Thus, we form compositions $$\sigma_{s1}=\tau_{11}\tau_{22}, \quad \sigma_{s2}=\tau_{12}\tau_{21}, \quad\sigma_{s2}^{-1}=\rho\sigma_{s1}\rho.$$
By a simple computation, we have
\ga  \nonumber
\sigma_{s1}\colon \begin{cases}
z_1'=z_1,\\
z_2'=-z_2- \ov\gaa_s^{-1} w_1,\\
w_1'=(1-\ov\gaa_s)^{-1} z_2+((\ov\gaa_s-\ov\gaa_s^2)^{-1}-1)w_1,\\
w_2'=w_2;
\end{cases}\\
\sigma_{s2}\colon \begin{cases}
z_1'=-z_1-(1-\gaa_s)^{-1}w_2,\\
z_2'=z_2,\\
w_1'=w_1,\\
w_2'=\gaa_s^{-1} z_1+((\gaa_s-\gaa_s^2)^{-1}-1)w_2.
\end{cases}
\nonumber
\end{gather}
We verify that $
\sigma_{s1}\sigma_{s2}=\sigma_s=\tau_1\tau_2.
$
This allows us to compute the eigenvalues of $\sigma_{s1}\sigma_{s2}$ easily:
\begin{gather}\label{msms-}
\mu_s, 
\quad  \mu_s^{-1}, 
\quad
  \ov \mu_s^{-1}, \quad 
  \ov\mu_s, \\  
\mu_s=\ov\gaa_s^{-1}-1.
\label{msms-+}
\end{gather}
In fact we compute them by observing that the first two  in \re{msms-} and $1$ with multiplicity are eigenvalues of $\sigma_{s1}$, while the last two  in \re{msms-} and $1$ with multiplicity are eigenvalues of $\sigma_{s2}$.  Note that   $\sigma_s$ is diagonalizable if and only if $\gaa_s\neq 1/2$.
When $\gamma_s<1/2$ and $\gamma_s\neq0$, $\mu_s$ and $\mu_s^{-1}$ are eigenvalues of multiplicity two while condition J holds.
Therefore,  for $\gaa_s\neq1/2$, i.e.
 $\mu_s\neq1$, we can find a linear transformation of the form
$$
\psi\colon (z_1,w_2)\to(\xi_2,\eta_2)=\ov\phi(z_1,w_2), \quad
(z_2,w_1)\to(\xi_1,\eta_1)=\phi(w_1,z_2)$$
 such that $\sigma_{s1},\sigma_{s2},\sigma_s=\sigma_{s1}\sigma_{s2}$ are simultaneously diagonalized as
 \begin{equation}\label{linears}
\begin{array}{rllll}
\sigma_{s1}\colon & \xi_1' =\mu_s\xi_1, \quad & \eta_1' =\mu_s^{-1}\eta_1, \quad& \xi_2' =\xi_2,\quad& \eta_2' =\eta_2, \\
\sigma_{s2}\colon &\xi_1' =\xi_1,\quad &\eta_1' =\eta_1,\quad&\xi_2' =\ov\mu_s^{-1}\xi_2,\quad&\eta_2' =\ov\mu_s\eta_2,\\
\sigma_s\colon&\xi_1' =\mu_s\xi_1, \quad&\eta_1' =\mu_s^{-1}\eta_1, \quad& \xi_2' =\ov\mu_s^{-1}\xi_2,\quad&\eta_2' =\ov\mu_s\eta_2.
\end{array}
\end{equation}
Under the transformation $\psi$, the involution $\rho$, defined by (\ref{antiholom-invol}),
takes the form
 \eq{rhos}
\rho(\xi_1,\xi_2,\eta_1,\eta_2)=(\ov\xi_2,\ov\xi_1,\ov\eta_2,\ov\eta_1).\eeq
  Moreover, for $j=1,2$, we have $\tau_{2j}=\rho\tau_{1j}\rho$ and
\ga
\tau_{1j}\colon\xi_j'=\la_j\eta_j, \quad \eta_j'=\la_j^{-1}\xi_j; \quad\xi_i'=\xi_i, \quad\eta_i'=\eta_i, \quad i\neq j;\label{taus}\\
\la_1=\la_s, \quad \la_2=\ov\la_s^{-1}, \quad\mu_s=\la_s^2.\label{taus+}
\end{gather}
Condition J on $Q_{\gaa_s}$ is equivalent to $\gamma_s\neq 1/2$.
By a permutation of coordinates that preserves $\rho$, we obtain a unique holomorphic invariant $\mu_s$ satisfying
\eq{mu1im}
|\mu_s|\geq1, \quad \IM\mu_s\geq0,\quad\mu_s\neq0,-1, \quad 0\leq \arg\la_s\leq\pi/2,\quad\la_s\neq0, i.
\eeq
Note that $\RE\gaa_s<1/2$ implies that $|\mu_s|\neq1$.

For later purpose, we summarize some   facts for complex type in the following.
\begin{prop}\label{sigs} Let $Q_{\gaa_s}\subset\cc^4$ be the   quadric defined by \rea{z3z4} and \rea{gsra} with $\gaa_s\neq0, 1$.  Then $\pi_1$ admits two deck transformations $\tau_{11},\tau_{12}$ such that 
the set of fixed points of each $\tau_{1j}$ has dimension $3$. Also, $\tau_{2j}=\rho\tau_{1j}\rho$ are the deck transformations of $\pi_2$ and \ga
\tau_{11}\tau_{21}=\tau_{21}\tau_{11}, \quad
\tau_{12}\tau_{22}=\tau_{22}\tau_{12}.
 \nonumber\end{gather}
Let $\sigma_{s1}=\tau_{11}\tau_{22}$, $\sigma_{s2}=\tau_{12}\tau_{21}$, $\tau_i=\tau_{i1}\tau_{i2}$, and $\sigma_s=\tau_1\tau_2$. Then
\eq{}  \nonumber
\sigma=\sigma_{s1}\sigma_{s2}=\sigma_{s2}\sigma_{s1}, \quad \sigma_{s2}^{-1}=\rho\sigma_{s1}\rho, \quad \sigma_s^{-1}=\rho\sigma_s\rho.
\eeq
Assume further that $\gaa_s\neq1/2$.
In suitable coordinates $\sigma_{s1},\sigma_{s2},\sigma,\rho_s$ are given by \rea{linears}-\rea{rhos}, while
$\sigma_s$ has  $4$ $($possibly repeated$)$ eigenvalues  given by \rea{msms-}-\rea{msms-+}.  The $\sigma$ has four distinct eigenvalues if and only if $Q_{\gaa_s}$ can be holomorphically normalized so that
$\IM\gaa_s>0$ and $\RE\gaa_s<1/2$.
\end{prop}

In summary, we have found linear coordinates for the product quadrics such that the normal forms of
$S$, $T_{ij}$, $\rho$ of the corresponding $\sigma,\sigma_j,\tau_{ij},\rho_0$ are given by
\begin{align}
&S\colon\xi_j'=\mu_j\xi_j,\quad \eta_j'=\mu_j^{-1}\eta_j;\label{sxij-mv}\\
\label{rSxi-mv}
&S_j\colon\xi_j'=\mu_j\xi_j,\quad \eta_j'=\mu_j^{-1}\eta_j,\quad \xi_k'=\xi_k,\quad \eta_k'=\eta_k,
\quad k\neq j;\\
\label{rTij-mv}
&T_{ij}\colon\xi_j'=\la_{i j}\eta_j,\quad\eta_j'=\la_{i j}^{-1}\xi_j,\quad\xi_k'=\xi_k,\quad
\eta_k'=\eta_k, \quad k\neq j; 
\\ 
\label{rRho-mv}
& \rho\colon \left\{\begin{array}{ll}
(\xi_e',\eta_e',\xi_h',\eta_h')=
(\ov\eta_e,\ov\xi_e,\ov\xi_h,\ov\eta_h),\vspace{.75ex}
\\
(\xi_{s}', \xi_{s+s_*}',\eta_{s}',\eta_{s+s_*}')=(\ov\xi_{s+s_*},
\ov\xi_{s},\ov\eta_{s+s_*}, \ov\eta_{s}).
\end{array}\right.
\end{align}
Throughout the paper, the indices $h,e,s$ have the ranges $1\leq e\leq e_*$, $e_*<h\leq e_*+h_*$, and $e_*+h_*<s\leq p-s_*$.
Notice that we can always normalize $\rho_0$ into the above normal form $\rho$.

The commutativity of $\sigma_h,\sigma_e,\sigma_{s1},\sigma_{s2}$ will be important to understand the convergence of normalization for the  abelian CR singularity.

\setcounter{thm}{0}\setcounter{equation}{0}
\section{
Formal deck
transformations  and centralizers}\label{fsubm}

In section~\ref{secinv} we  show 
 the equivalence of   the classification of real analytic submanifolds $M$ that admit the maximum number of deck transformations and the classification of the families of
involutions $\{\tau_{11}, \ldots, \tau_{1p},\rho\}$ that satisfy some mild conditions  (see \rp{inmae}).
As a consequence we show that a real submanifold is formally equivalent to a quadric if and only if its
family  of Moser-Webster involutions is formally linearizable.



\subsection{Formal submanifolds}
We  first need some notation.
Let $I$ be an ideal
of the ring $\rr[[x]]$ of formal power series in   $x=(x_1,\ldots, x_N)$.
Since $\rr[[x]]$ is noetherian, then $I$ and its radical $\sqrt I$ are finitely generated. We say that $I$ defines a formal  submanifold $M$ of dimension $N-k$ if $\sqrt I$ is generated by $r_1,\ldots, r_k$ such that at the origin all $r_j$ vanish and $dr_1,\ldots, dr_k$
  are linearly independent. For such an $M$, let $I(M)$ denote $\sqrt I$ and let   $T_0M$ be defined by $dr_1(0)=\cdots=dr_k(0)=0$. If $F=(f_1,\ldots, f_N)$ is a formal mapping
  with $f_j\in\rr[[x]]$, we say that its set of (formal) fixed points is a submanifold if
  the ideal generated by $f_1(x)-x_1, \ldots, f_N(x)-x_N$ defines a submanifold.
Let $I,\tilde I$ be   ideals of $\rr[[x]], \rr[[y]]$ and
 let $\sqrt I, \sqrt{\tilde I}$ define two formal submanifolds $M, \tilde M$,
 respectively. We say
  that a formal map $y=G(x)$ maps $M$ into $\widetilde{M}$ if $\tilde I\circ G\subset \sqrt I$.
If $M,\tilde M$ are   in the same space, we write $M\subset\tilde M$ if $\tilde I\subset\sqrt I$. We say that a formal map $F$ fixes $M$
pointwise if $I(M)$ contain each component of the mapping $F-\id$.

\subsection{Formal deck transformations}
Consider a   formal $(2p)$-submanifold in $\cc^{2p}$
defined by
\eq{fmzp}
M\colon
z_{p+j} = E_{j}(z',\bar z'), \quad 1\leq j\leq p.
\eeq
Here $E_j$ are   formal power series in $z',\ov z'$.  We  assume that
\eq{fmzp+}
E_j(z',\bar z')=h_j(z',\ov z')+q_j(\ov z')+O((|(z',\ov z')|^3)
\eeq
and $h_j, q_j$ are homogeneous quadratic
polynomials. The formal complexification of $M$ is  defined by
\begin{equation}\nonumber
\label{variete-complex}
\begin{cases}
z_{p+i} = E_{i}(z',w'),\quad i=1,\ldots, p,\\
w_{p+i} = \bar E_i(w',z'),\quad i=1,\ldots, p.\\
\end{cases}
\end{equation}
 We define a {\it formal deck transformation}
of $\pi_1$ to be a formal biholomorphic map
$$
\tau\colon (z',w')\to (z',f(z',w')), \quad \tau(0)=0
$$
such that $\pi_1\tau=\pi_1$, i.e. $E\circ\tau=E$. Recall that condition B says that $q_*=\dim\{z'\in\cc^n\colon q(z')=0\}$ is zero, i.e.  $q$ vanishes only at
the origin
in $\cc^p$.
\begin{lemma}
Let $M$ be a formal submanifold defined by \rea{fmzp}-\rea{fmzp+}. Suppose that
$M$ satisfies condition {\rm B}. Then formal
deck transformations of $\pi_1$ are commutative  involutions. Each formal deck transformation $\tau$ of $\pi_1\colon \cL M\to \cc^p$ is
uniquely determined by its linear part $L\tau$
in the $(z',w')$ coordinates, while $L\tau$ is a deck transformation for the complexification for $\pi_1\colon \cL Q\to\cc^p$, where $\cL Q$ is the complexification of
the quadratic part  $Q$ of $M$.
 If $M$ is real analytic, all
formal
deck transformations of $\pi_1$ are convergent.
\end{lemma}
\begin{proof} Let us recall some results about the quadric $Q$.
We already know that $q_*=0$ implies that $\pi_1$ for the complexification of
$Q$ is a branched   covering. As used in the proof of \rl{2p1},  $\pi_1$  is an open mapping near
 the origin and its regular values are dense.
In particular, we have
\eq{dwph}
\det\pd_{w'}\{h(z',w')+  q(w')\}\not\equiv0.
\eeq

Let $\tau$ be a formal deck transformation for  $M$. To show that $\tau$ is an involution, we note that its linear part at the origin, $L\tau$, is a deck transformation
of $Q$. Hence $L\tau$ is an involution.
Replacing $\tau$ by the deck transformation $\tau^2$, we may assume that
$\tau$ is tangent to the identity. Write
$$
\tau(z',w')=(z',w'+u(z',w')).
$$
We want to show that $u=0$. Assume that $u(z',w')=O(|(z',w')|^k)$  and let $u_k$ be
 homogeneous and of degree $k$ such that $u(z',w')=u_k(z',w')+O(|(z',w')|^{k+1})$. We have
$$
E(z',w'+u(z',w'))=E(z',w').
$$
Comparing terms of order $k+1$, we get
$$
\pd_{w'}\{h(z',w')+ q(w')\}u_k(z',w')=0.
$$
By \re{dwph}, $u_k=0$. This shows that each formal deck transformation $\tau$ of $\pi_1$ for $M$ is an involution. As mentioned above, $L\tau$ is a deck transformation of $\pi_1$ for $Q$.
Also if $\tau,\tilde\tau$  are commuting formal involutions then $\tau^{-1}\tilde\tau$ is an involution and $\tau=\tilde\tau$ if and only if $L\tau= L\tilde\tau$.

 Assume now that $M$ is real analytic. We want to show that  each formal deck transformation $\tau$ is convergent.
By a theorem of Artin~\cite{artin68} applied to the solution $u$, there is a convergent $\tilde\tau(z',w')=\tau(z',w')
+O(|(z',w')|^2)$ such that $E\circ\tilde\tau=E$, i.e. $\tilde\tau$ is a deck
transformation. Then $\tilde\tau^{-1}\tau$ is a deck transformation   tangent to the identity. Since it is a
formal involution by the above argument,
then it must be identity. Therefore,
 $\tau=\tilde\tau$ converges.
\end{proof}

Analogous to real analytic submanifolds, we say that a formal manifold defined
by \rea{fmzp}-\rea{fmzp+} satisfies condition D if its  formal branched  covering $\pi_1$ admits
$2^p$ formal deck transformations.

 Recall from section~\ref{secinv}
that
it is crucial to distinguish a special set of generators
for the deck transformations in order to relate the classification
of real analytic manifolds to the classification of
  certain $\{\tau_{11}, \ldots, \tau_{1p},\rho\}$.
The set of generators is uniquely determined by
the dimension of  fixed-point sets.
We want to extend these results at the
formal level.

\begin{prop} \label{mmtp}
Let $M,\tilde M$ be formal $p$-submanifolds
in $\cc^n$ of the form \rea{fmzp}-\rea{fmzp+}.
Suppose that   $M,\tilde M$ satisfy condition  {\rm D}. Then the following hold~$:$
\bppp
\item $M$ and $\tilde M$ are
formally equivalent
if and only if  their associated families of involutions
$\{\tau_{11}, \ldots, \tau_{1p},
\rho\}$ and $\{\tilde\tau_{11}, \ldots, \tilde\tau_{1p},\rho\}$
are   formally equivalent.
\item
Let  $\cL T_1=\{\tau_{11}, \ldots,\tau_{1p}\}$ be a family of
 formal holomorphic
involutions which commute pairwise. Suppose that
 the tangent spaces of $\fix(\tau_{11}),\ldots,
 \fix(\tau_{1p})$  are hyperplanes
   intersecting  transversally at the origin. Let $\rho$
 be an anti-holomorphic formal involution and let $\cL T_2=\{\tau_{21},\ldots,\tau_{2p}\}$   with $\tau_{2j}=\rho\tau_{1j}\rho$.
Suppose that  $\sigma=\tau_1\tau_2$ has distinct eigenvalues for $\tau_i=\tau_{i1}\cdots\tau_{ip}$, and
$$
 [\mathfrak M_n]_1^{L\cL T_1}\cap[\mathfrak M_n]_1^{L\cL T_2}
 =\{0\}.
 $$
There exists a formal  submanifold  defined by
\begin{equation}\label{masym}
 z''=(B_1^2,\ldots, B_p^2)(z',\ov z')
\end{equation}
for some formal power series $B_1,\ldots, B_p$
such that $M$ satisfies condition  {\rm D}.  The set
of involutions  $\{ \tilde\tau_{11},\ldots, \tilde\tau_{1p},   \rho_0\}$  of $M$ is formally equivalent to
 $\{\tau_{11}, \ldots, \tau_{1p},\rho\}$.
 \eppp
\end{prop}
\begin{proof} (i) Let $M$ and $\tilde M$ be given by $z''=E(z',\ov{z'})$ and $\tilde z''=\tilde E(\tilde z',\ov{\tilde z'})$, respectively.
Suppose that $f$ is a formal holomorphic
 transformation sending $M$ into $\tilde M$.  We have
\eq{fppz}
f''(z',E(z',w'))=\tilde E(f'(z',E(z',w')),\ov f'(w',\ov E(w',z'))).
\eeq
Here $f=(f',f'')$. Recall that 
 $\rho_0(z',w')=(\ov{w'},\ov{z'})$. Define a formal mapping $(z',w')\to (\tilde z',\tilde w')=F(z',w')$ by
\eq{Fzpwp}
F(z',w'):= (f'(z',E(z',w')),\ov {f'}(w',\ov E(w',z'))).
\eeq
It is clear that $ F\rho_0=\rho_0 F$.
  By \rl{twosetin}, we know that $\tilde z'$ and $\tilde z''=\tilde E(\tilde z', \tilde w')$ generate invariant formal
power series of $\{\tilde\tau_{1j}\}$.
Thus, $\tilde z'\circ F(z',w')=f'(z',E(z',w'))$ and $\tilde E\circ F(z',w')$  are invariant by  $F^{-1}\circ\tilde\tau_{1j}\circ F$. By \re{fppz} and the definition of $F$,
 $$\tilde E\circ F(z',w')=f''(z',E(z',w')).$$
This shows that $f(z', E(z',w'))$  is invariant under
 $F^{-1}\circ\tilde\tau_{1j}\circ F$. Since $f$ is invertible, then $z'$ and $E(z',w')$
  are invariant under $F^{-1}\circ\tilde\tau_{1j}\circ F$.
 Therefore, $\{\tau_{1j}\}$ and $\{F^{-1}\circ\tilde\tau_{1i}\circ F\}$  are the same by \rl{twosetin} as
 they
 have the same invariant functions.  

Assume now that $\{\tau_{1j}\}=\{F^{-1}\circ\tilde\tau_{1i}\circ F\}$ for some
formal biholomorphic map $F$ commuting with $\rho_0$.  Recall that $\tilde z',\tilde z''$
are invariant by $ \tilde\tau_{1j}$. Then $\tilde z'\circ F$ and $\tilde E\circ F$ are
 invariant by $\{\tau_{1j}\}$.    By \rl{twosetin}, invariant power series of ${\tau_{1j}}$
are generated by $z',E(z',w')$. Thus
we can write
\ga \nonumber
\tilde z'\circ F(z',w')=f'(z',E(z',w')), \\
\label{hefp} \tilde E\circ F(z',w')=f''(z',E(z',w'))
 \end{gather}
for some formal power series map $f=(f',f'')$.   Since $\rho_0 F=F\rho_0$,
then by \re{Fzpwp} $$F(z',w') 
=(f'(z',0),\ov f'(w',0))+O(|(z',w')|^2).$$
Since $F$ is (formally) biholomorphic then $z'\to f'(z',0)$ is biholomorphic.
Then
$$
f''(0,E(0,w'))=\tilde E(0,\ov f'(w',0))+O(|w'|^3).
$$
 We have $E(0,w')=q(w')+O(|w'|^3)$ and $\tilde E(0,w')=\tilde q(w')+O(|w'|^3)$. Here $q(w'),\tilde q(w')$ are quadratic. By condition $q_*=0$, we know that $\tilde
 q_1, \ldots,\tilde q_p$ and hence $\tilde q_1\circ L, \ldots, \tilde q_p\circ L$ are linearly independent.
 Here $L$ is the linear part of the mapping $w'\to \ov f'(w',0)$, which is invertible. This shows that
the linear part of $w'\to f''(0,w')$
 is biholomorphic. 
 By \re{hefp}, $f''(z',0)=O(|z'|^2)$. Hence $f=(f',f'')$ is biholomorphic.
By a simple computation, we have $f(M)=\tilde M$, i.e.
$$
\tilde E(f'(z),\ov{f'(z)})=f''(z)
$$
for $z''=E(z',\ov z')$.

(ii)
Assume that $\{\tau_{1j}\}$ and $\rho$
are given in the $(\xi,\eta)$
space.  We want to show that a formal holomorphic
 equivalence class of $\{\tau_{1j},\rho\}$ can be realized by a formal submanifold satisfying condition D.
The proof is almost identical to the realization proof of \rp{inmae} and we will be brief.  Using a formal, instead of convergent,
change of coordinates,  we
know that invariant formal power series of $\{\tau_{1j}\}$ are generated by
$$
 z'=(A_1(\xi,\eta),\ldots,
A_p(\xi,\eta)), \quad  z''=(B_1^2(\xi,\eta), \ldots, B_p^2(\xi,\eta)),
$$
where $B_j$ is skew-invariant by $  \tau_{1j}$, and $A,B_i$ are invariant under $\tau_{1j}$
for $i\neq j$.  Moreover, $\phi(\xi,\eta)=(A,B)(\xi,\eta)$ is formally biholomorphic.
Set $$
w_j'=\ov{A_j\circ\rho(\xi,\eta)}, \quad w''_j=\ov{B_j^2\circ\rho(\xi,\eta)}.
$$
Then $(\xi,\eta)\to (A(\xi,\eta),\ov{A\circ\rho(\xi,\eta)})$
has an inverse $\psi$.  Define
$$
M\colon z''=(B_1^2,\ldots, B_p^2)\circ\psi(z',\ov z').
$$
The complexification of $M$ is given by
$$
\cL M\colon
z''=(B_1^2,\ldots, B_p^2)\circ\psi (z',w'),\quad
w''=(\ov B_1^2,\ldots, \ov B_p^2)\circ\ov\psi(w',  z').
$$
Note that $\phi\circ\psi (z',w')=(z',B\circ\psi(z',w'))$. Since $\phi\psi$ is invertible,
the linear part $D$ of $B\circ\psi$ satisfies
$
|D(0,w')|\geq |w'|/C.
$
This shows that $q_*=0$.
As in the proof of \rp{inmae}, we can verify that $M$ is a realization for $\{\tau_{1j},\rho\}$.\end{proof}
\setcounter{thm}{0}\setcounter{equation}{0}

\section{
Normal forms of
commuting biholomorphisms}\label{nfcb}

In this section, we shall consider a family of commuting germs of  holomorphic diffeomorphisms at a common fixed point, say $0\in \cc^n$.
We shall give conditions that ensure that the family can be transformed simultaneously and holomorphically to a normal form. This means that there exists a germ of biholomorphism at the origin which conjugates each   germ
of biholomorphism in the family to a mapping that commutes with the linear part of every mapping of the family.

\subsection{Centralizers and Decomposition} 

\begin{defn}\label{def-centr}Let $\cL  F$ be a family of formal mappings  on $\cc^n$  fixing
the origin.
Let $\cL C(\cL  F)$
 be the  {\it centralizer} of $ \cL   F$, i.e.
 the set of formal  holomorphic mappings $g$ that fix  the origin and
  commute with each element $f$ of $\cL  F$, i.e., $f\circ g=g\circ f$.
\end{defn}
Let ${\mathcal C}_2(\cL F)$ be  the ``higher order 
 formal centralizer'' of $\cL F$, that is
$$
{\mathcal C}_2(\cL F)=\{H\in (\widehat{\mathfrak M}^2_n)^n\,|\;H\circ F=F\circ H,\; F\in\cL F \}.
$$

We now deal with the following decomposition problem: Let $\mathcal C$ be a set of analytic mappings. We  want to decompose an arbitrary invertible
 mapping into the composition of an element of a centralizer of $\mathcal C$ and an element which is normalized with respect to $\mathcal C$. 
We first prove a general convergence decomposition,  which  will be used several times.
Let  $e_j$ denote the standard $j$th unit vector of $\cc^n$.
\begin{defn}
Let $\cL A$ be a group of permutations of $\{1,\ldots, n\}$. Then
$\cL A$ acts on the right (resp. on the left) on $\widehat{\cL O}_n^n$ by permutation of variables $z=(z_1,\ldots, z_n)$ as follows: Let $F(z)=\sum_{|Q|>0}F_Qz^Q$ be a formal mapping from $\cc^n$ to $\cc^n$,   and let $\nu, \mu\in \cL A$; set
$$\nu\circ F\circ\mu(z):=\sum_i \sum_{Q\in \nn^n}F_{\nu(i),\mu^{-1}(Q)}z^Q
e_{ i}.
$$
Define the components $(\cL AF)_i, (F\cL A)_i$, and consequently $(\cL AF\cL A)_i$ by
\aln
(\cL A F)_i(z)& :=\sum_{Q\in \nn^n}\max_{\nu\in \cL A}|F_{\nu(i),Q}|z^Q, \qquad
(F\cL A)_i(z):=\sum_{Q\in \nn^n}\max_{\mu\in\cL A}|F_{i,\mu^{-1}(Q)}|z^Q,
\\
(\cL AF\cL A)_i(z) &\, =\sum_{Q\in \nn^n}\max_{(\nu,\mu)\in \cL A^2}|F_{\nu(i),\mu^{-1}(Q)}|z^Q.
 \end{align*}
\end{defn}
We see that $F\cL A$ is the smallest (w.r.t. $\prec$) power series mapping that majorizes $F$
and is right-invariant under   $\cL A$, while $\cL AF$ is the smallest power series
mapping that majorizes $F$ and is left-invariant under   $\cL A$.
In particular, if $F, G$ are mappings without constant or linear terms, then
\ga
\label{FiGp}
\cL A(F\circ (I+G))\cL A\prec (\cL AF\cL A)(\cL A I\cL A+\cL AG\cL A),
\end{gather}
where the last relation holds if the composition is well-defined.

To simply our notation, we will take $\cL A$ to be the full permutation group of $\{1, \ldots, n\}$. We will denote
$$
F_{sym}=\cL AF\cL A.
$$


\begin{lemma}\label{fhg-}  Let  $\hat {\cL H}$ be  a real subspace of $({\widehat { \mathfrak M}}_n^2)^n$.
Let $\pi : ({\widehat { \mathfrak M}_n^2)^n}\rightarrow \hat  {\cL H}$ be
a $\rr$ linear projection $($i.e. $\pi^2=\pi)$
 that preserves the degrees of the mappings and let $\hat  {\cL G}:= (\I-\pi)({\widehat { \mathfrak M}}_n^2)^n$.
Suppose that there is a positive constant $C$ such that
\eq{pifp}
\pi(E)\prec  CE_{sym}
\eeq
for any $E\in (\widehat { \mathfrak M}_n^2)^n$.
 Let $F$ be a formal map tangent to the identity.
There exists a unique decomposition
\eq{decompo}
F=HG^{-1}
\eeq
with $G-I\in\hat  {\cL G}$ and $H-I\in \hat  {\cL H}$.
  If $F$ is convergent, then $G$ and $H$ are also convergent.
\end{lemma}
\begin{proof} If $f$ is a formal mapping, we define the  $k$-jet: 
 $$
 J^kf(z)=\sum_{|Q|\leq k}f_Qz^Q.
 $$
 Write $F=I +f$,  $G=I +g$ and $H=I+h$. We need to solve $FG=H$, i.e to solve
$$
h-g=f(I+g).
$$
Since $f'(0)=0$, then for any $k\geq 2$, the $k$-jet of $f(I +g)$ depends only on the $(k-1)$-jet of $g$. Since $\pi$ is linear and preserves degrees,  \re{pifp} implies that $J^k$ commutes with $\pi$.
 Hence we can define, for all $k\geq 2$,
$$
-J^{k}(g):=\pi\left(J^{k}(f(I +g)\right),\quad J^{k}(h):=(I-\pi)\left(J^{k}(f(I +g)\right).
$$
This solves the formal decomposition uniquely. Assume that $F$ is a germ of holomorphic mapping.
Hence, we have
$$
 g\prec C(f(I +g))_{sym}\prec C f_{sym}(I_{sym}+ g_{sym}).
$$
Since $g_{sym}$ is the smallest left and right $\cL A$ invariant power series that dominates $g$, we have
$$
g_{sym} \prec C f_{sym}(I_{sym}+ g_{sym}).
$$
Therefore, $g_{sym}$ is dominated by the solution $u$ to
$$
u=Cf_{sym}(I_{sym}+ u),\quad u(0)=0.
$$
Notice that $u$  is real analytic near the origin by the implicit function theorem. So, $g_{sym}$ is convergent,  and   $g$, $h
=g+f(I +g)$ are convergent. 
\end{proof}

\begin{rem}Let $\cL A,\cL B$ be two subgroups of permutations. Instead of using the full permutations group, we could have used $G_{sym}:=\cL AG\cL B$. We have
$$
G\prec \cL AG\cL B\prec C\cL A(F\circ (I+G))\cL B\prec (\cL AF\cL A)(\cL AI\cL B+\cL AG\cL B).
$$
\end{rem}


\subsection{Abelian family of biholomorphisms}
Let ${\mathbf D}_i:=\text{diag}(\mu_{i1},\ldots,\mu_{in})$ with $1\leq i\leq\ell$
 be diagonal invertible matrices of $\cc^n$. Let $D_i\colon x\to {\mathbf D}_ix$ be the linear mappings.
Let  $D$ denote the family $\{  D_i 
\}_{i=1,\ldots \ell}$ of linear mappings.
\begin{defn}\label{ccst}
We say that $F=(f_1,\dots, f_n)$ is {\it normalized} with respect to $D$ if it  is tangent to the identity and it
does not have components along the centralizer of $D$, i.e.   for each
$Q$ with $|Q|\geq2$,
\eq{norm-D}
\nonumber
f_{ j,Q}=0,\quad\text{if}\; \mu_i^Q=\mu_{ij}\; \text{for all}\; i. 
\eeq
Let   ${\cL C}^{\mathsf{ c}}(D)$ denote the set of formal mappings  normalized with respect to  $D$. Let
${\cL C}_2^{\mathsf{ c}}(D)$ be the set of all $H\in   (\widehat{\mathfrak M}^2_n)^n$ satisfying $I+H\in
{\cL C}^{\mathsf{ c}}(D)$.
\end{defn}
%

  Let us consider a family $F:=\{F_i\}_{i=1}^{\ell}$ of  germs of holomorphic diffeomorphisms  of $(\cc^n,0)$ of which the linear of $F_i(x)$ at the origin is
$D_i.
$
Thus
$$
F_i(x)={\mathbf D}_ix+f_i(x), \quad   f_i(0)=0,\quad Df_i(0)=0.
$$
The group of germs of (resp. formal) biholomorphisms tangent to identity acts on
the family $F$ by  
$\Phi_*F:=\{\Phi^{-1}\circ F_i\circ\Phi\colon 1\leq i\leq\ell\}$.

Let  $\{F_i\}_{i=1,\ldots \ell}$ be a family of {\bf commuting germs of biholomorphisms}
with $\ell<\infty$.
Let us recall a result by M. Chaperon (see theorem 4 in \cite{Ch86}, page 132):
\begin{prop} 
If the family of diffeomorphisms is abelian then there exists a formal diffeomorphism $  \Phi$, which is tangent to the identity,
 such that
$$
{\widehat F_i}({\mathbf D}_jx)= {\mathbf D}_j{\widehat F_i}(x),\quad 1\leq i,j\leq \ell
$$
where ${\widehat F_i} :=   \Phi_*F_i$, for $1\leq i\leq \ell$.
We call the family  $\{\widehat F_i\}$ a {formal normal form} of the family $F$ (or a normalized family) with respect to the family $D$ of linear maps.
\end{prop}

For convenience, we restrict   changes of holomorphic coordinates to the ones that are tangent to the identity.
Also $\Phi_*\{F_i\}_{i=1}^\ell=\{\tilde F_i\}_{i=1}^\ell$ means that
 $$
 \Phi_*F_i=\tilde F_i, \quad 1\leq i\leq \ell.
 $$
These  restrictions will be removed by mild changes.  For instance, if $\Phi$ transforms a family $F$ into a family $\hat F$ that commutes with $LF$, the family of the linear part of the $F$, then $(L\Phi)^{-1}(LF_i)L\Phi=L\hat F_i$. Therefore, $\Phi(L\Phi)^{-1}$ is tangent to the identity
and  transforms $F$ into $ (L\Phi) \hat F (L\Phi)^{-1}$ which commutes with $LF$.

Let $\widehat{\mathcal O}_n^D$ be the ring of formal invariants of the family $D$, that is
$$
\widehat{\mathcal O}_n^D:=\{f\in \widehat{\mathcal O}_n\,|\;f({\mathbf D}_ix)=f(x),\; i=1,\ldots, \ell  \}.
$$
 If $Q\in\nn^n$, $Q\neq 0$, then
    $x^Q\in \widehat{\mathcal O}_n^D$ if and only if
\eq{muiq-}
\nonumber
\mu_i^Q:=\mu_{i1}^{q_1}\cdots\mu_{in}^{q_n}=1,\quad \forall\; 1\leq i\leq \ell.
\eeq
If $|Q|>1$, then   $x^Qe_j\in {\mathcal C}_2(D)$ if and only if
$ 
\mu_i^Q=\mu_{ij}, \forall\; 1\leq i\leq \ell. 
$
With notation of \rd{def-centr}, we have
\begin{lemma}\label{pcn0}
Any formal diffeomorphism $\Phi$ of $(\cc^n,0)$, tangent to identity, can be written uniquely
as $\Phi=\Phi_1\circ\Phi_0^{-1}$  with $\Phi_1\in\cL C^{\mathsf c}(D)$ and $\Phi_0\in \cL C(D)$. Furthermore, $\Phi_0,\Phi_1$
are convergent when $\Phi$ is convergent.
\end{lemma}
\begin{proof} This follows from \rl{fhg-}, where $\hat{\cL H}$ is replaced by $\cL C_2(D)$ and $\pi$
is defined by
$$
\pi\left(\sum f_{j,Q}x^Qe_j\right)=\sum_j\sum_{x^Qe_j \not\in\cL C_2(D)}f_{j,Q}x^Qe_j.
\qquad \qedhere$$
\end{proof}
\begin{lemma}\label{lem-nf-nf}
Let $\hat F:=\{\hat F_i\}$ and $\tilde F:=\{\tilde F_i\}$ be two formal normal forms of the abelian family of diffeomorphisms $F:=\{F_i\}$.    There exists a
formal diffeomorphism $\Phi$, tangent to identity at the origin, such that $ \Phi\in  {\mathcal C}(D)$ and $\Phi\circ\tilde F_i=\hat F_i\circ\Phi$.
Furthermore, there is a unique $\Phi\in\cL C^c(D)$ that transforms the family $F$ into a normal form.
\end{lemma}
\begin{proof}
Since both $\hat F$ and $\tilde F$ are normal forms of $F$,  there exists a formal diffeomorphism $\Phi$, tangent to identity at the origin, such that $\tilde F_i\circ\Phi=\Phi\circ \hat F_i$. According to  \rl{pcn0}, we can decompose  $\Phi=\Phi_1\circ\Phi_0^{-1}$ where $\Phi_0\in \cL C(D)$ and $\Phi_1\in \cL C^{\mathsf c}(D)$. Hence, we have $\Phi_1^{-1}\circ\tilde F_i\circ\Phi_1=\Phi_0^{-1}\circ \hat F_i\circ\Phi_0$.  Let us set $G_i:=\Phi_0^{-1}\circ\hat F_i\circ\Phi_0$. Then $G_i$ is a formal diffeomorphism satisfying $G_i(x)-{\mathbf D}_ix\in {\mathcal C}(D)$. 
Let us show by induction on $N\geq 2$ that if $\Phi_1=I+\Phi_1^N+O(N+1)$ with $\Phi_1^N$ being
homogeneous of degree $N$, then $\Phi_1^N=0$. Indeed, a computation  shows that
$$
\{G_i\}_N=\{\hat F_i\}_N+ D_i\circ \Phi_1^N- \Phi_1^N\circ D_i.
$$
Express   $\Phi_1^N$ as sum of monomial mappings. The monomial mappings are not
in $\cL C(D)$,  while those of $F_i$ and $G_i$ are. We obtain $\Phi_1^N=0$.

To verify the last assertion,  assume that $\Psi_*F=\hat F$ and $\tilde\Psi_*F
=\tilde F$
are in the normal form.  Suppose that $\Psi,\tilde\Psi$ are normalized. Then $(\Psi^{-1}\tilde\Psi)_*\hat F
=\tilde\Psi_*(\Psi^{-1})_*\hat F$ is in the normal form.  Write $\Psi^{-1}\tilde\Psi=\psi_1\psi_0^{-1}$ with $\psi_1\in\cL C^{\mathsf c}(D)$
and $\psi_0\in\cL C (D)$. Then $(\psi_1)_*\hat F$ is in a normal
form. From the above proof, we know that $\psi_1=I$. Now $\Psi=\tilde\Psi\psi_0$, which implies that $\Psi=\tilde\Psi$.
\end{proof}
%
\begin{lemma}\label{complete-lemma}
If a formal normal form of $F$ is completely integrable so are all other normal forms of $F$; in particular, the unique $\Phi$ in \rla{lem-nf-nf} transforms $F$ into a completely integrable normal form.
\end{lemma}
\begin{proof}
By Lemma \ref{lem-nf-nf}, we transform a normal form $\{\hat F_i\}$ into another one $\{\tilde F_i\}$ by applying a transformation $\Phi$ that commutes with each $D_j$. Hence, we have $\hat F_i:=\Phi^{-1}\circ \tilde F_i\circ \Phi$, for all $i=1,\ldots, \ell$. Let us write $\Phi(x)=\sum_{Q\in\nn^n,\; 1\leq j\leq n}\phi_{j,Q}x^Q e_j$. Suppose that $\{\hat F_i\}$ is completely integrable, then
$$
\Phi\circ \hat F_i(x)=\sum_{Q\in\nn^n}\phi_{j,Q}\mu_{i}(x)^Qx^Q e_j.
$$
Since $\Phi$ commutes with each $D_j$, then
$$
\Phi\circ \hat F_i(x)= \text{diag}(\mu_{i1}(x),\ldots,\mu_{in}(x))\cdot \Phi(x).
$$
The conjugacy equation leads to
$$
\mu_{ij}(x) \cdot \Phi_j(x)= \tilde {F}_{ij}(\Phi(x)), \quad 1\leq j\leq n.
$$
As a consequence, we have $\tilde F_i(x)=\text{diag}((\tilde \mu_{i1}(x),\ldots,\tilde \mu_{in}(x))\cdot x$ with $(\tilde\mu_{ij}\circ \Phi(x))\cdot\Phi_  j (x)= \mu_{ij}(x)\cdot\Phi_j(x)$, i.e. $\tilde \mu_{ij}=\mu_{ij}\circ \Phi^{-1}$.

Each function $\tilde\mu_{ij}$ is  an invariant function of $D$ since
$$
\tilde\mu_{ij}({\mathbf D}_k x)= \mu_{ij}\circ \Phi^{-1}({\mathbf D}_kx)= \mu_{ij}\circ D_k (\Phi^{-1}(x))= \mu_{ij}\circ \Phi^{-1}(x).
$$
The second and third conditions of the definition of the complete integrability is obviously satisfied by $\{\tilde{F}_i\}$ since $\tilde \mu_{ij}=\mu_{ij}\circ \Phi^{-1}$.
\end{proof}
\begin{lemma}
If a formal normal form of $F$ is linear so are all other normal forms of $F$.
\end{lemma}
\begin{proof}
According to Lemma \ref{lem-nf-nf}, we transform a    linear normal form $\{\hat F_i\}$ into another one $\{\tilde F_i\}$ by applying a transformation $\Phi$ that commutes with each $D_j$. Since $ \hat F_i(x)={\mathbf D}_ix$, we have $\tilde F_i=\Phi^{-1}(D_i\Phi(x))$, for all $i=1,\ldots, \ell$. Since $\Phi$ commutes with each map $x\mapsto {\mathbf D}_ix$, then
$$
\tilde F_i=\Phi^{-1}(D_i\Phi(x))=\Phi^{-1}(\Phi({\mathbf D}_ix))={\mathbf D}_ix.
\qquad\qedhere
$$
\end{proof}

\begin{defn}\label{small divisors}
We  say that 
 {\it the family $D$ is of Poincar\'e type}
 if there exist  constants $d>1$ and $c>0$ such that, for each $(j,Q)\in\{1,\ldots, n\} \times\nn^n$
 that satisfies
  $\mu_{m}^{Q}-\mu_{m j}\neq 0$ for some $m$,
  there exists $(i,Q')\in \{1,\ldots, \ell\}\times \nn^n$ such that
$\mu_{k}^{Q'}=\mu_k^Q$ for all $1\leq k\leq \ell$,
$\mu_{i}^{Q'}-\mu_{ij}\neq0$,
 and
\gan
\max\left(|\mu_{i}^{Q'}|,|\mu_{i}^{-Q'}|\right)>c^{-1}d^{|Q'|}, \quad
 \text{ $Q'-Q\in\nn^n\cup(-\nn^n$)}.
\end{gather*}
 \end{defn}
  Such a condition has appeared in the definition of the good set in~\cite{BHV10}.
\begin{defn}
Let $f=\sum_{Q\in\nn^n}f_Qx^Q$ and $g=\sum_{Q\in\nn^n}g_Qx^Q$ be two formal power series. We   say that $g$  {\it majorizes} $f$,
 written as $f\prec g$, if $g_Q\geq 0$ and $|f_Q|\leq g_Q$ for all $Q\in\nn^n$.  Set
$$
\bar f := \sum_{Q\in\nn^n}|f_Q|x^Q.
$$
\end{defn}
\begin{thm}\label{thm-conv-nf}
Let $F$ be an abelian family of germs of holomorphic diffeomorphisms at the origin of $\cc^n$. Assume that it is formally completely integrable and   that its linear part at the origin is of Poincar\'e type.
Then $F$ is holomorphically conjugated to a normal form $\hat F=\{\hat F_1,\ldots, \hat F_\ell\}$
 so that 
 each
 $\hat F_i$ is defined by
$$
x_j' = \mu_{ij}(x)x_j,\quad j=1,\ldots,n
$$
where  $\mu_{ij}(x)$ are   germs of holomorphic functions  invariant under $D$
and $\mu_{ij}(0)=\mu_{ij}$.
 In fact, the unique normalized mapping $\Phi$ in \rla{lem-nf-nf}
is convergent.
\end{thm}
The primary example is the  following Moser-Webster normal form of reversible mappings:
$$
\hat\sigma: \xi'=  M_1(\xi\eta, \zeta)\xi\quad \eta'= M_1^{-1}(\xi\eta,\zeta)\eta, \quad \zeta'= \zeta.
$$
where $(\xi,\eta)\in\cc^2$, $\zeta\in\cc^{n-2}$, and $|M_1(0)|>1$.  Our convergence proof is inspired by the proof in Moser-Webster~\cite{MW83}.
\begin{proof}
The last assertion follows  from \rl{pcn0}
and \rl{complete-lemma}.
Let us conjugate, simultaneously, each $F_i= {\mathbf D}_ix+f_i$ to $\hat F_i:=\hat{\mathbf D_i}(x)x$ by the action of $\Phi(x)=x+\phi(x)$ where $\phi(0)=0$ and $\phi'(0)=0$.  Here, $\hat{\mathbf D}_i(x)$ denotes the matrix $\text{diag}(\hat\mu_{i1}(x),\ldots,\hat\mu_{in}(x))$ and each $\hat\mu_{ij}(x)$ is a formal power series invariant under $D$, i.e. $\hat\mu_{ij}(x)\in\widehat{\mathcal O}_n^D$. We can assume that $\Phi$ does not have a non-zero component along the centralizer of $D$; indeed,  by \rl{complete-lemma}, we can assume that $\Phi$ is normalized 
 w.r.t $D$.
 Then, for each $i=1,\ldots, \ell$, we have
$$
F_i\circ\Phi(x)= {\mathbf D}_ix+f_i(\Phi)(x)+{\mathbf D}_i\phi(x), \quad \Phi\circ \hat F_i(x)= \hat{\mathbf {\mathbf D}_i}(x)x+ \phi(\hat F_i)(x).
$$
Equation $F_i\circ\Phi=\Phi\circ \hat F_i$  reads
\begin{equation}\label{conjugacy}
\left(\phi(\hat {\mathbf D}_i(x)x)-{\mathbf D}_i\phi(x)\right) +\left(\hat {\mathbf D}_i(x)-{\mathbf D}_i\right)x = f_i(\Phi)(x)\quad i=1,\ldots, \ell.
\end{equation}
Our convergence proof is  based  on two conditions: the existence of a formal $\phi\in\cL C^{\mathsf c}(D)$ that satisfies
the above equation, and the Poincar\'e type condition on the linear part $D$. We already know that
$\phi$ is unique. We shall project equation \re{conjugacy} along the ``non-resonant'' space (i.e. the space ${\mathcal C}^{\mathsf c}(D)$ of normalized mappings w.r.t. ${\mathbf D}$). The mapping $\phi$ also solves this last equation and we shall majorize it using that projected equation.

Let us first decompose these equations along the ``resonant'' and ``non-resonant'' parts,   i.e.
  $\cL C_2(D)$ and $\cL C^{\mathsf c}_2(D)$.
 Since $\phi=\sum_{Q\in \nn^n, |Q|\geq 2} \phi_{j,Q}x^Qe_j$ is normalized then $\phi_{j,Q}=0$ for some $Q\in\nn^n$, $|Q|\geq 2$ and $1\leq j\leq n$, if we have $\mu_{m}^Q=\mu_{mj}$ for all $m$. We recall that, since each ${\mathbf D}_i$ is a diagonal matrix, then a map belongs to the centralizer of $D$ if and only if each monomial   map of its Taylor expansion at the origin belongs to this centralizer as well.
Since the $\hat \mu_{ij}$ is a  formal invariant function then
$$
\phi(\hat {\mathbf D}_i(x)x)=\sum_{Q\in \nn^n, |Q|\geq 2} \phi_{j,Q}\hat\mu_i^Q(x)x^Qe_j=:\sum_{Q'\in \nn^n, |Q'|\geq 2} \psi_{j,Q'}x^{Q'}e_j.
$$
The latter
contains only non-resonant terms, that is that if   $\mu_{i}^{Q'}=\mu_{ij}$ for all $i$, then $\psi_{j,Q'}=0$.
  Indeed,  $\hat\mu_i^Q(x)$ contains monomials of the form $x^P$ with $\mu_i^P=1$ for all $1\leq i\leq \ell$.   Hence, $\psi_{j,Q'}$ is a linear combination of $\phi_{j,Q}$ such that $Q'=Q+P$ with $\mu_i^P=1$ for all $i$. Therefore, if $\mu_{i}^{Q'}=\mu_{ij}$ for all $i$, then for all these $Q$'s, we have $\mu_{i}^{Q}=\mu_{i}^{Q'}=\mu_{ij}$ for all $i$ so that
  $\phi_{j,Q}=0$; that is $\psi_{j,Q'}=0$.

Hence, the projection on  the resonant mappings in $\cL C_2(D)$ 
 leads to
\beq\label{res}
\left(\hat {\mathbf D}_i(x)-{\mathbf D}_i\right)x = \{f_i(\Phi)(x)\}_{res},\quad i=1,\ldots, \ell.
\eeq
Here  for any formal mapping $g(x)=O(|x|^2)$ on $\cc^n$,  we define the projection on $\cL C_2(D)$
by
\eq{gres}
  (g(x))_{res}=\sum_j\sum_{\forall i,\mu_i^Q= \mu_{ij}} g_{j,Q}x^Qe_j.
\eeq
  The projection $g$ onto $\cL C_2^{\mathsf c}(D)$ is defined as
$ 
g(x)-(g(x))_{res},
$ 
i.e. it is the projection of $g$ on the non-resonant mappings.

Let us consider the projection on the non-resonant mappings. We first need to decompose power series according to a non-homogeneous equivalence relation on
their coefficients.
Let us define the equivalence relation on $ \{1,\ldots,n\}\times\nn^n$ by
 $$
 (j,Q)\sim (\tilde{\jmath},\widetilde Q), \  \text{if $\mu_{ij}-\hat\mu_i^Q(x)=
 \mu_{i\tilde{\jmath}}-\hat\mu_i^{\widetilde Q}(x)$ for all $1\leq i\leq \ell$}.
 $$
Here the identities hold as formal power series.
Let $\Delta$ be the set of the equivalence classes   on the non-resonant   multiindex set
$$
 \left\{(j,Q)\in\{1,\dots, n\}\times
\nn^n\colon (\mu_1^Q-\mu_{1j},\dots, \mu_{\ell}^Q-\mu_{\ell j})\neq 0,  |Q|>1\right\}.
$$

If  
 $\mu_k^Q-\mu_{kj}\neq 0$ for some $k$,  clearly
 $\hat\mu_k^Q-\mu_{kj}$ is not identically zero.
We can decompose any  formal power series map $f$ along these equivalent classes
and the resonant part  of the mapping. Let $\delta \in \Delta$ and $f=\sum_{Q=\in\nn^n, 1\leq j\leq n}  f_{j,Q}x^Qe_j$
  with $f=O(2)$. We can  write
\eq{fdelx}
 f_{\delta}(x):=\sum_{  (j,Q)\in \delta}f_{j,Q}x^Qe_j,   \quad \sum_{\del\in\Del}
 f_{\del}(x)\prec \ov f(x).
\eeq
We   denote by  $\widehat { \mathfrak M}_{n,\delta}^n$ the vector space of such maps.
To a given equivalence class $\delta$, we  associate a representative $(j_{\delta}, Q_{\delta})$, and
  we shall identify an equation among $n$ equations in \re{conjugacy}
for estimation.
  Since $\phi$ contains no resonant  mappings, 
  then
\eq{phdi}
\phi=\sum_{\delta\in\Delta}\phi_\delta.
\eeq
Let us decompose the projection onto   non-resonant mappings in $\cL C_2^{\mathsf c}(D)$
of equation
 $(\ref{conjugacy})$ along each equivalence class $\delta$ as follows. 
Using the definition of the equivalence class $\Delta$, we obtain
\begin{equation}\label{conjugacy-delta}
\left[\hat\mu_i^{Q_{\delta}}(x)-\mu_{ij_{\delta}}\right]\phi_{\delta}(x) = \left\{f_i(\Phi)\right\}_{\delta}(x),\quad \forall i=1,\ldots, \ell
\end{equation}
where $\{f\}_{\delta}$ denotes the projection of $f$ on $\widehat { \mathfrak M}_{n,\delta}^n$, defined by \re{fdelx}.

For each   $(j_\del,Q_\del) 
\in\delta$, we know that   $\mu_k^{Q_\del}-\mu_{kj_\del}\neq 0$  
for some $k$.
By the Poincar\'e type condition, there exist $i$
and   $Q_{\delta}'\in\nn^n$ such that
\eq{mulq}\mu_{i}^{Q_{\delta}'}-\mu_{ij_{\delta}}\neq 0; \quad
\mu_{m}^{Q_{\delta}'}=\mu_{m}^{Q_{\delta}},\quad \forall  1\leq m\leq   \ell;\quad Q_{\delta}'-Q_\del\in\nn^n\cup(-\nn^n)
\eeq
and, furthermore,
one of the following holds:
\ga\label{miqp}
|\mu_{i}^{Q'_{\delta}}|\leq cd^{-|Q'_{\delta}|}, \quad\text{or}\quad
|\mu_{i}^{-Q'_{\delta}}|\leq cd^{-|Q'_{\delta}|}.
\end{gather}
  Here, $d>1$ does not depend on $Q_{\delta}$.
So, let us use the $i$th equation of $(\ref{conjugacy-delta})$ to estimate $\phi_{\delta}$. We have, for that $i$,
\eq{pdhm-}
\phi_{\delta} = \left[\hat\mu_{i}^{Q_{\delta}}-\mu_{ij_{\delta}}\right]^{-1}\left\{f_{i}(\Phi)\right\}_{\delta}.
\eeq
Therefore,  we have established the uniqueness of $\phi$ under \re{phdi} and \re{pdhm-}, and under the
condition that $\phi$ satisfies the equation when \re{conjugacy} is projected onto $\cL C^{\mathsf c}(D)$.
The existence of $\phi$
is ensured by assumption.   We now consider the convergence of $\phi$.
By \re{mulq} 
  we obtain $\hat\mu_{i}^{Q_{\delta}'-Q_{\delta}}=1$.
This allows us to rewrite \re{pdhm-} as
\eq{pdhm}
\phi_{\delta} = \left[\hat\mu_{i}^{Q_{\delta}'}-\mu_{ij_{\delta}}\right]^{-1}\left\{f_{i}(\Phi)\right\}_{\delta}.
\eeq
We majorize this power series.

  Recall that $\hat\mu_{ij}(0)=\mu_{ij}$. Let us set
$
  M_{ij}(x):=\mu_{ij}^{-1} \hat \mu_{ij}(x). 
$
We have $M_{ij}(0)=1$ and we decompose
$
 M_{ij}(x)=\sum_{Q\in\nn^n} M_{ij,Q}x^Q.
$
Let us set $\mu^*:=\max_{ij}\{|\mu_{ij}|,|\mu_{ij}^{-1}|\}$, and
$$
m_{i}=\sum_{Q\in\nn^n} \max_{1\leq j\leq n}|M_{ij,Q}|x^Q,\quad m=\sum_{Q\in\nn^n} \max_{1\leq i\leq \ell,\;1\leq j\leq n}|M_{ij,Q}|x^Q.
$$
Note that $m(0)=1$. Then $M_{ij}\prec m$ and
$$
M_{ij}^{-1}=\frac{1}{1+(M_{ij}-1)}\prec\f{1}{1-(m-1)} =\f{1}{2-m}.
$$
  Here and in what follows, if $f(x)$ is a formal power series with $f(0)=0$, then for any number
$a\neq0$, $\frac{1}{a-f(x)}$ stands for the  formal power series in $x$ for
$$
\frac{1}{a}\left \{1+\sum_{n=1}^\infty(a^{-1}f(x))^n\right\}.
$$

To simplify notation in  \re{pdhm}, let us write $Q$
for $Q_{\delta}'$ and $j$ for $j_{\delta}$. We want to show that, 
\eq{hmiq}
(\hat\mu_i^{Q}-\mu_{ij})^{-1}\prec S(m-1).
\eeq
Here $S(t)$ is a convergent power series   independent of  all $(j,Q)'s$ of the form $(j_\delta,Q_\delta')'s$. 
Fix $d_1$ with $1<d_1<d$.
We  consider the first case that $\mu^*cd^{-|Q|}>d_1^{-|Q|}$. Since $d>d_1$,  we have only finitely many such  $Q's$   (recall
that each $Q$ has the form $Q_\del'$).
  The function  $M_i\mapsto \mu_{ij}-\mu_i^QM_i^Q$ is holomorphic in $M_i\in\cc^n$ at $M_i= 
  (1,\ldots, 1)$ and does not vanish at this point.
  Hence, the function
$$
(\mu_{ij}-\hat\mu_i^Q)^{-1}=(\mu_{ij}-\mu_i^QM_i^Q)^{-1}
$$
is also holomorphic at $M_i=(1,\ldots, 1)$. For all  $Q's$  in the first case, 
we
have
\aln
  (\mu_{ij}-\hat\mu_i^Q)^{-1}&\prec
\frac{C}{1-C(\ov M_{i1}-1+\cdots+\ov M_{in}-1)}\prec\frac{C}{1-nC(m-1)}.
\end{align*}
Consider   the second case that  $\mu^*c d^{-|Q|}\leq d_1^{-|Q|}$.  For the first case in \re{miqp}, we obtain
\begin{eqnarray*}
(\hat\mu_i^{Q}-\mu_{ij})^{-1} & = & -\mu_{ij}^{-1}(1 - \mu_{ij}^{-1}\mu_i^QM_i^Q)^{-1}
\prec
\mu^*\left[1- \mu^*cd^{-|Q|}m^{|Q|}\right]^{-1}
\\
&\prec &\mu^*\left[1-d_1^{-|Q|}m^{|Q|}\right]^{-1}
\prec
\mu^*\left[1-d_1^{-1}m \right]^{-1}.
\end{eqnarray*}
For the second case in \re{miqp}, we have
\begin{eqnarray*}
(\hat\mu_i^{Q}-\mu_{ij})^{-1} & = & - \mu_{i}^{-Q}M_i^{-Q}\left[1 - \mu_{ij}\hat\mu_i^{-Q}M_i^{-Q}\right]^{-1}
\\  &
\prec &
cd^{-|Q|}(2-m)^{-|Q|}\left[1-\mu^*cd^{-|Q|}(2-m)^{-|Q|}\right]^{-1}\\
&\prec &(\mu^*)^{-1}d_1^{-|Q|}(2-m)^{-|Q|}\left[1-d_1^{-|Q|}(2-m)^{-|Q|}\right]^{-1}\\
&\prec &(\mu^*)^{-1} \left[1-d_1^{-1}(2-m)^{-1}\right]^{-1}.
\end{eqnarray*}
  We have obtained the estimates for the second case. 
Therefore, we have shown that for any $Q=Q_\del'$ and $j=j_{\delta}$
\eq{hmiq}
(\hat\mu_i^{Q}-\mu_{ij})^{-1}\prec S(m-1).
\eeq
Here $S(t)$ is a convergent power series   independent of  all $(j,Q)'s$ of the form $(j_\delta,Q_\delta')'s$. 

Let us set
$$
f^*:=\sum_{Q\in \nn^n}\max_{1\leq i\leq \ell,\;1\leq j\leq n}|f_{ij,Q}|x^Q e_j.
$$
  By the definition of the equivalence relation on multiindices, we have
\eq{phdi+}
\sum_{\delta\in\Delta}f_\delta^*\prec f^*.
\eeq
According to \re{pdhm} and \re{hmiq}, we have $\phi_{\delta} \prec S(m-1)\left\{f^*(\bar \Phi)\right\}_{\delta}$. Now \re{fdelx} and \re{phdi+} imply
\beq\label{maj-nonres}
\phi \prec  S(m-1)f^*(\bar \Phi).
\eeq

Let us  project $(\ref{res})$ onto   the $k$th components of $\cL C_2(D)$ as follows.
For a power series   map $g$, we define
$$
g_{res,k}(x)=\sum_{\mu^Q=\mu_k}\quad g_{k,Q}x^Q.
$$
  By the definition of $g_{res}$ in \re{gres}, $g_{res}=(g_{res,1},\ldots, g_{res,n})$.
 We have
$$
\mu_{ik}\left(M_{i,k}(x)-1\right)x_k =\left(\hat \mu_{ik}(x)-\mu_{ik}\right)x_k = \{f_{ik}(\Phi)\}_{res,k}(x).
$$
Therefore, for all $1\leq k\leq n$,
\beq\label{maj-res}
(m-1)x_k\prec \frac{1}{\min_{i,j}|\mu_{i,j}|}f^*(\bar \Phi).
\eeq
Let us set $\mu_*:=\frac{1}{\min_{i,j}|\mu_{ij}|}$. We
 set $x_1= t, \ldots, x_n= t$ 
 in $\ov\Phi(x)$ and $m(x)$.
  Let $\phi( t)$,   $\ov\Phi( t)$,  and $ m( t)$
still denote $\phi( t,\ldots,  t)$, $\ov\Phi( t,\ldots, t)$, and $m( t,\ldots, t)$, respectively. Let
$$
 t W( t) := \phi( t) + (m( t)-1) t.
$$
We have $W(0)=0$, $\phi( t)\prec  t W( t)$, and $(m( t)-1)\prec W( t)$.
 From estimates $(\ref{maj-nonres})$ and $(\ref{maj-res})$, we obtain
\beq \label{maj-W}
 t W( t)\prec \mu_*f^*(\bar \Phi( t))+S(m( t)-1)f^*(\bar \Phi( t)).
\eeq
Since  ${f_{ij}}(x)=O(|x|^2)$, 
 there exists a constant $c_1$ such that
$$
f^*(x)\prec \frac{c_1(\sum_j x_j)^2}{1-c_1(\sum_j x_j)}.
$$
Hence, estimate $(\ref{maj-W})$ reads
\begin{eqnarray} \label{rel-dom}
 t W( t)&\prec &\left(\mu_*+S(m( t)-1)\right)\frac{c_1(n( t+\phi))^2}{1-c_1n( t+\phi)} 
 \\
&\prec & \left(\mu_*+S( W( t))\right)\frac{c_1 t^2(n(1+W( t)))^2}{1-c_1n t(1+W( t))}.\nonumber
\end{eqnarray}
Let us consider the equation in the unknown $U$ with $U(0)=0$~:
\beq\label{equ-maj}
U( t)(1-c_1n t(1+U( t)))= \left(\mu_*+S( W( t))\right)c_1 t(n(1+U( t)))^2.
\eeq
According to the implicit function theorem, there exists a unique germ of holomorphic function $ U( t)$, solution to $(\ref{equ-maj})$ with $U(0)=0$.
According to inequality $(\ref{rel-dom})$, the function $W$ is dominated by $U$~: $W( t)\prec U( t)$. This can be seen by induction on the degree of the Taylor polynomials at the origin. Therefore, $W$ converges at the origin. The theorem is proved.
\end{proof}

\setcounter{thm}{0}\setcounter{equation}{0}

\section{Real manifolds with an abelian CR-singularity}\label{sectabel}

Let us consider a real analytic manifold $M$ with a CR-singularity at the origin, which is an higher order perturbation of a product quadric. 
We assume that for its complexification ${\mathcal M}$, $\pi_1$ has  $2^p$ deck transformations generated by $\{\tau_{11},\dots, \tau_{1p}\}$. 
Let
$ 
\tau_{2j}=\rho\circ\tau_{1j}\circ\rho.
$

Let us consider the following germs of holomorphic diffeomorphisms~:
\ga\label{sigma_i}
\sigma_i:=\tau_{1i}\circ\tau_{2i},\quad 1\leq i\leq e_*+h_*,\\
\label{sigma_i+}
\sigma_s:=\tau_{1s}\circ\tau_{2 (s_*+s)},\quad \sigma_{s+s_*}=\tau_{1 (s+s_*)}\circ\tau_{2s},\quad e_*+h_*<s\leq p-s_*.
\end{gather}
Notice that    the above  property holds for quadrics of  the complex case by
\rp{sigs}.
 We assume that the linear parts  $T_{ij},S_j,S$ of $\tau_{ij},\sigma_j,\sigma$ and $\rho$ are
given by \re{sxij-mv}-\re{rRho-mv}.
The family $\{\sigma_i\}$  is reversible with respect to $\rho$. More precisely, we have the following relations
$$
\sigma_i^{-1}=\rho\sigma_i\rho, \quad 1\leq i\leq e_*+h_*;\quad
\sigma_{s+s_*}^{-1}=\rho\sigma_s\rho, \quad e_*+h_*<s\leq p-s_*.
$$

\begin{defn}
We   say that the manifold $M$ has an {\bf abelian CR-singularity} at the origin if its complexification ${\mathcal M}$ has the maximum number of deck transformation and if the family $\{\sigma_1,\dots,\sigma_p\}$ of germs of biholomorphisms at the origin of $\cc^{2p}$ is abelian, i.e. $$\sigma_i\sigma_j=\sigma_j\sigma_i.$$
\end{defn}


The aim of this section is to show that such an analytic perturbation with an abelian CR-singularity and no hyperbolic component is holomorphically conjugate to a normal form. We shall give two proofs of this result. The first one rests on Moser-Webster result \cite{MW83}[theorem 4.1] applied successively to each $\sigma_i$. It is to be emphasized that it is fortunate that we can apply such a result to our situation including the new type
of CR singularity of {\it  complex} type. The other one is based on the fact that the family $\{\sigma_i\}$ is formally completely integrable and their linear part is of Poincar\'e type. We then apply \rt{thm-conv-nf}.

Before we apply the above theorem, let us first exhibit an  example of real manifolds with an abelian CR singularity.
We start with a Bishop surface $M_0\subset\cc^2$ defined by $z_2=R(z_1,\ov {z_1})$ with
\eq{flatm}
R(z_1,\ov{z_1})=z_1\ov z_1+\gaa_1(z_1^2+\ov{z_1}^2)
+O(3), \quad R(z_1,\ov z_1)=\ov R(\ov {z_1},z_1).
\eeq
Let $\tau_{1}^0,\tau_2^0$ be the Moser-Webster involutions of $M_0$.  On the complexification $\cL M_0$ of $M_0$,  the $z_1$ and $z_2:=R(z_1,w_1)$
are invariant by $\tau_1^0$. Condition \re{flatm} implies that $w_2:=\ov R(w_1,z_1)=z_2$ is also invariant
by $\tau_1^0$. Analogously, $w_1,w_2$ and $z_2$ are invariant by $\tau_2^0$.
We are ready to verify the following
\pr{} Let $R$ be given by \rea{flatm} and let
\eq{}
M\colon z_3=R(z_1,\ov{z_1}), \quad z_4=(z_2+2\gaa_2 \ov{z_2}+C(z_2,z_3,\ov z_2))^2,
\eeq
where $C(z_2,R(z_1,\ov {z_1}),\ov{z_2})=O(3)$.  Then $M$ has an   abelian CR singularity at $0$.
\epr
\begin{proof} Let $\tau_{11}$ be defined by $(z_1',w_1')=\tau_1^0(z_1,w_1)$ and $(z_2',w_2')=
(z_2,w_2)$. Let us verify that $\tau_{11}$ is a deck transformation of $\pi_1$, i.e. all $z_j$ are invariant
by it. Obviously, $z_2,w_2$ are invariant by $\tau_{11}$.
 We know that $z_1,z_3=R(z_1,w_1)$ are invariant by $\tau_1^0$ and hence by $\tau_{11}$.  On $\cL M$,      $z_4=(z_2+2\gaa_2w_2+C(z_2,z_3,w_2))^2$ is
 then invariant by $\tau_{11}$.
Therefore, $\tau_{11}$ is an involution and it fixes a hypersurface in $\cL M$. Let $\tau_{12}$ be defined by
$(z_1',z_2',w_1')=(z_1,z_2,w_1)$ and
$$
z_2+2\gaa_2w_2'+C(z_2,R(z_1,w_1),w_2')=-z_2-2\gaa_2w_2-C(z_2,R(z_1,w_1),w_2).
$$
By the implicit function theorem, $w_2'=-\gaa_2^{-1}z_2-w_2+f(z_2,R(z_1,w_1),w_2)$ with $f$ being
convergent.
It is clear that $\tau_{12}$ leaves $z_j$ invariant and it fixes a hypersurface.

For the abelian property,  we note that $z_2,w_2$ and $R(z_1,w_1)$ are invariant by  $\tau_{11}, \tau_{21}$.
By a straightforward computation, we verity that $\tau_{11}$ and $\tau_{21}$ commute with $\tau_{12}$.
Now $\tau_{21}$ and $\tau_{11}$ commute with $\tau_{22}=\rho\tau_{12}\rho$. This shows that $\tau_{11}\tau_{21}$
commutes with $\tau_{12}\tau_{22}$.
\end{proof}

\subsection{Normal forms  for 
 abelian CR singularity
}
\begin{thm}\label{abelinv}
Let $M$ be a germ of real analytic submanifold in $\cc^n$ at an
 abelian CR-singularity at the origin. Suppose that $M$  is a higher order perturbation of a product quadric of which
 $\gaa_1, \ldots, \gaa_p$ satisfy \rea{0ge1}.  
Suppose that $M$ 
does not have a 
 hyperbolic component $($i.e. $e_*\geq 0, s_*\geq0, h_*=0)$ and $\RE\gaa_s<1/2$.
Then there exists a germ of biholomorphism $\psi$ that commutes with $\rho$ and such that, for $1\leq i\leq p$ and
$k=1,2$
\ga\label{convnfsi}
\psi^{-1}\circ\sigma_{i}\circ\psi:\begin{cases}\xi'_i = M_i(\xi\eta)\xi_i\\ \eta'_i= M_i^{-1}(\xi\eta)\eta_i\\ \xi'_j= \xi_j\\ \eta'_j= \eta_j, \quad j\neq i,\end{cases}\quad
 \psi^{-1}\circ\tau_{ki}\circ\psi:\begin{cases}\xi'_i  = \Lambda_{ki} (\xi\eta)\eta_i\\ \eta'_i = \Lambda_{ki}^{-1}(\xi\eta)\xi_i\\ \xi'_j = \xi_j\\
 \eta'_j = \eta_j,\quad j\neq i.\end{cases}
\end{gather}
 Moreover,  $\Lambda_{2j}=\Lambda_{1j}^{-1}$ and
\begin{gather} 
\label{lam1e}
\Lambda_{1e}=\overline{\Lambda_{1e}\circ{\rho_z}},\  1\leq e\leq e_*; \quad
\Lambda_{1s}^{-1}=\overline{\Lambda_{1(s+s_*)}\circ {\rho_z}},\ e_*<s\leq p-s_*,\\
\label{rhoz5}
{\rho_z}\colon\zeta_e\to\ov\zeta_e,
\quad\zeta_s\to\ov\zeta_{s+s_*}, \quad\zeta_{s+s_*}\to\ov\zeta_s.
\end{gather}
\end{thm}
\begin{proof}
  We will present two  convergence proofs: one is based on a convergent
  theorem of Moser and Webster and another is based on  \rt{thm-conv-nf}.
 We first use some formal results obtained by Moser and Webster~\cite{MW83} and  some results in section~\ref{nfcb}. The conditions of the theorem imply that, for all $i$,  $|\mu_i|\neq 1$.

Since $M$ is a higher order perturbation of a product quadric, there are linear coordinates such that, for $1\leq i\leq p$ and
$k=1,2$, $\tau_{ki}$ and $\sigma_i$ are higher order perturbations of
\gan
S_{i}:\begin{cases}\xi'_i = \mu_i\xi_i\\ \eta'_i= \mu_i^{-1}\eta_i\\ \xi'_j= \xi_j\\ \eta'_j= \eta_j, \quad j\neq i,\end{cases}\quad
 T_{ki}:\begin{cases}\xi'_i  = \lambda_{ki}\eta_i\\ \eta'_i = \lambda_{ki}^{-1}\xi_i\\ \xi'_j = \xi_j\\
 \eta'_j = \eta_j,\quad j\neq i.\end{cases}
\end{gather*}
For elliptic coordinates, this was computed in \cite{MW83} and recalled in \re{tau1e}. For complex coordinates, this is computed in \re{linears} and \re{taus}.
Recall that $\sigma_1,\dots, \sigma_p$ are defined by \re{sigma_i}-\re{sigma_i+}
Since $|\mu_1|\neq 1$, then by theorem 4.1 of Moser-Webster (\cite{MW83}), there is a  unique convergent transformation $\psi_1$
 normalized w.r.t. $S_1$ (see~\rd{ccst})
 such that  $\sigma_1^*:=\psi_1^{-1}\circ\sigma_{1}\circ\psi_1$ and $ \tau_{i1}^*:=\psi_1^{-1}\circ\tau_{i1}
 \circ\psi_1$
 are given by   \ga\label{sigma1norm}
\sigma_1^*\colon 
\begin{cases}x'_1 = M_1(\xi,\eta)\xi_1\\ \eta'_1= M_1^{-1}(\xi,\eta)\eta_1\\ \xi'_j= \xi_j\\ \eta'_j= \eta_j, \quad j\neq 1,\end{cases}\qquad
 \tau_{k1}^*\colon 
 \begin{cases}\xi'_1  = \Lambda_{k1} (\xi,\eta)\eta_1\\ \eta'_1 = \Lambda_{k1}^{-1}(\xi,\eta)\xi_1\\ \xi'_j = \xi_j\\
 \eta'_j = \eta_j,\quad j\neq 1.\end{cases}
\end{gather}
Here $k=1,2$.
It is a  simple fact (e.g. see  \rl{complete-lemma}, $D=\{ S_1\})$ that there is a unique
$\phi_1\in \cL C^c(S_1)$ such that $\phi^{-1}_1\sigma_1\phi_1$ is in the centralizer of $S_1$.
Therefore,  $\phi_1=\psi_1$ is also convergent.

Furthermore, we have $M_1(\xi,\eta)=\Lambda_{11}(\xi,\eta)\Lambda_{21}^{-1}(\xi,\eta)$;
and $\Lambda_{11}, \Lambda_{21}, M_1$ are invariant by $ S_1$. In the new coordinates, let
us denote $\tau_{im},\sigma_m$ by the same symbols for $m>1$. However, $\sigma_1=\sigma_1^*$
and $\tau_{k1}=\tau_{k1}^*$. Since each $\sigma_m$ commutes with $\sigma_1$, then
$\sigma_m$ is in the centralizer of $S_1$. Indeed, according to \cite{MW83}[Lemma 3.1](or \rl{pcn0} with $D=\{ S_1\}$), we can decompose $\sigma_m=\sigma_m^1\sigma_m^0$ where $\sigma_m^1$ is normalized w.r.t $S_1$ and $\sigma_m^0$ is in the centralizer of $ S_1$.  Write $\sigma_1\sigma_m=\sigma_m\sigma_1$ as
 $$(\sigma_m^1)^{-1}\sigma_{1}\sigma_m^1=\sigma_m^0\sigma_{1}(\sigma_m^0)^{-1}.$$
Since $\sigma_m^0\sigma_{1}(\sigma_m^0)^{-1}$ belongs to ${\mathcal C}(S_1)$, so does $(\sigma_m^1)^{-1}\sigma_{1}\sigma_m^1$.
Then applying the uniqueness of $\psi_1$ stated earlier to $\sigma_m^1$, we conclude that $\sigma_m^1=I$
and $\sigma_m=\sigma_m^0$ is in the centralizer of $S_1$.

Let us verify that $\sigma_m^0$ or in general  each (formal) transformation $\var$
  in $\cL C(S_1)$ preserves the form of $\sigma_{1}^*$ and $\tau_{i1}^*$.
Indeed, $\varphi^{-1}$ commutes with $S_1$ too. Thus $\varphi^{-1}\sigma_1^*\varphi$ commutes with $S_1$ and its linear part is $S_1$.  The linear part of $\varphi$ must preserve the eigenspaces of $S_1$ and hence it is given by
$$
\xi_1\to a\xi_1,\quad \eta_1\to b\eta_1,
\quad (\xi_*,\eta_*) \to \phi(\xi_*,\eta_*)
$$
for $\xi_*=(\xi_2,\dots, \xi_n)$ and $\eta_*=(\eta_2,\dots, \eta_n)$.
The linear part of $\tau_{k1}^*$ is given by
$$
\xi_1\to\lambda_{k1}\eta_1, \quad\eta_1\to\lambda_{k1}^{-1}\xi_1,\quad(\xi_*,\eta_*)\to(\xi_*,\eta_*).
$$
By a simple computation, the linear part of $\varphi^{-1}\tau_{k1}^*\varphi$ still has this form with $\lambda_{k1}\frac{b}{a}$ instead of $\lambda_{k1}$.   According to \cite{MW83}[lemma 3.2], there a unique normalized mapping $\Psi$ that normalizes $\varphi^{-1}\sigma_1^*\varphi$  and  the $\varphi^{-1}\tau_{k1}^*\varphi$'s. According to the uniqueness property of \rl{lem-nf-nf}, $\Psi=Id$. Therefore, $\var$ preserves the forms of $\tau_{i1}^*$ and $\sigma_1^*$.

Let $\psi_2$ be the unique biholomorphic map  normalized w.r.t. $S_2$  such that $\psi_2^{-1}\sigma_2\psi_2=\sigma_2^*$ and $\psi_2^{-1}\tau_{k2}\psi_2=\tau_{k2}^*$ are in the  normal form~:
\ga\label{sigma2norm}
\sigma_{2}^*:
\begin{cases}\xi'_2 = M_2(\xi,\eta)\xi_2\\ \eta'_2= M_2^{-1}(\xi,\eta)\eta_2\\ \xi'_j= \xi_j\\ \eta'_j= \eta_j, \quad j\neq 2,\end{cases}\quad
\tau_{k2}^*:\begin{cases}\xi'_2  = \Lambda_{k2} (\xi,\eta)\eta_2\\ \eta'_2 = \Lambda_{k2}^{-1}(\xi,\eta)\xi_2\\ \xi'_j = \xi_j\\
 \eta'_j = \eta_j,\quad j\neq 2.\end{cases}
\end{gather}
Here $k=1,2$, and  $M_2$ and $\Lambda_{k2}$ are invariant by $S_2$.
Since $\sigma_2$ commutes with $S_1$, we have
$$
(S_1^{-1}\psi_2S_1)^{-1}\circ\sigma_2\circ(S_1^{-1}\psi_2S_1)=S_1^{-1}\sigma_2^*S_1.
$$
Note that  $S_1^{-1}\sigma_{2}^*S_1$ (resp. $S_1^{-1}\tau_{k2}^*S_1$) has the form \re{sigma2norm}  
 in which $M_{2}$ (resp. $\Lambda_{k2}$) is replaced by $M_{2}\circ S_1$ (resp. $\Lambda_{k2}\circ S_1$). In other words $S_1^{-1}\sigma_{2}^*S_1$ and $S_1^{-1}\tau_{k2}^*S_1$ are still of the form  \re{sigma2norm}. Since $S_1$ is diagonal, then
$S_1^{-1}\psi_2S_1$ remains normalized w.r.t. $S_2$. Applying the above uniqueness on $\psi_2$ for
$\sigma_{2}$, we conclude that $\psi_2=S_1^{-1}\psi_2S_1$. This shows that $\psi_2$
preserves the forms of $\tau_{k1}^*$ and $\sigma_1^*$. By the same argument as above, we have $\sigma_m^*\in {\mathcal C}(S_1,S_2)$.

In summary, we have found holomorphic coordinates so that $\tau_{ij}=\tau_{im}^*$ and $\sigma_m=\sigma_m^*$ for $m=1,2$.   As mentioned
previously, we know that $\sigma_1^*, \sigma_2^*, \sigma_3, \ldots, \sigma_m$
commute with $S_1$ and $S_2$. In particular, $M_1, M_2$ are invariant by $S_1, S_2$.
Repeating  this procedure,  we find  a holomorphic map $\phi$   so that all $\phi^{-1}\sigma_j\phi=\sigma_j^*$
and $\phi^{-1}\tau_{kj}\phi=\tau_{kj}^*$ are in the normal forms. Furthermore, $M_i$ and $\Lambda_{k,i}$ are invariant by $\cL S=\{S_1, \ldots, S_p\}$.

By \rl{pcn0}, we decompose $\phi=\phi_1\phi_0^{-1}$ where $\phi_1$ is normalized w.r.t.
$\cL S$ 
and $\phi_0$ is in the centralizer of $\cL S$.
 Then
$\phi_1^{-1}\sigma_j\phi_1=\sigma_j^*$
and $\phi^{-1}_1\tau_{ij}\phi_1=\tau_{ij}^*$ are in the normal forms,  since $\phi_0$ commutes with $S_j$.
We want to show that $\phi_1$ commutes with $\rho$.

Note that  $\sigma_e^{-1}=\rho\sigma_e\rho$ and $\sigma_{s+s_*}^{-1}=\rho\sigma_s\rho$.
Thus $(\rho\phi_1\rho)^{-1}\sigma_j(\rho\phi_1\rho)=\tilde\sigma_j^*$ where $\tilde\sigma_e^*:= \rho(\sigma_e^*)^{-1}\rho$ and $\tilde\sigma_s^*:= \rho(\sigma_{s+s*}^*)^{-1}\rho$.
It is easy to see  that $\rho\phi_1\rho$
is still normalized w.r.t. $\cL S$ (see also \rd{dfnorm}).
By \rla{lem-nf-nf},  we  know that there is a unique normalized formal  mapping $\phi_1$ such that $\phi_1^{-1}\sigma_j\phi_1$ are in the centralizer of $\cL S$. Since $\tilde\sigma_j^*$ belongs to the centralizer of $\cL S$, then we have $\rho\phi_1\rho=\phi_1$.

 Now, $\tau_{2j}^*=\rho\tau_{1j}^*\rho$ follows from $\tau_{2j}=\rho\tau_{1j}\rho$.
This shows that   
$$
\Lambda_{2e}=\overline{\Lambda_{1e}^{-1}\circ\rho_z},\quad 
\Lambda_{2s}=\overline{\Lambda_{1 (s+s_*)} \circ \rho_z},\quad
\Lambda_{2(s+s_*)}=\overline{\Lambda_{1 s} \circ \rho_z},
$$
where $1\leq e\leq e_*$ and $e_* <s\leq p-s_*$. Let $\phi_2$ be defined by $$
\xi_j'=(\Lambda_{1j}^{1/2}M_j^{1/4})(\xi\eta)\xi_j, \quad
\eta_j'=(\Lambda_{1j}^{-1/2}M_j^{-1/4})(\xi\eta)\eta_j, \quad
1\leq j\leq p.$$ For  a suitable choice of the roots, we have $\phi_2\rho=\rho\phi_2$. Furthermore, $\phi_2$ preserves all invariant functions of $\cL S$. Hence, each $\phi_2^{-1} \circ\phi_1^{-1}\circ\tau_{ki}\circ\phi_1\circ\phi_2$ has the form $\tau_{kj}^*$ stated in
 \rt {abelinv}.

 \medskip

 We now present another proof by using the more general \rt{thm-conv-nf}.

Note that the above proof is valid at the formal level without using the convergence result of Moser and Webster.
More specifically,
 if $\tau_{ij}$ are given by formal power series with
$\sigma_1, \ldots, \sigma_p$ commuting pairwise, there exists a formal map $\psi$ that is tangent to the identity and commutes
with $\rho$ such that \re{convnfsi} holds. Since each $\mu_j$ is not a root of unity,  then
\re{convnfsi} implies that the   conjugate family $\{\sigma_m^*\}$ is a completely integrable normal form.

%

Let $\sigma_i$ be defined as above. Let $S_i$ be its linear part at the origin of $\cc^{n}$. The eigenvalues $\{\mu_{ij}\}_{1\leq j\leq n}$ of $S_i$ are either $\mu_i$, $\mu_i^{-1}$ or $1$. More precisely, if $Q\in\nn^n$, $|Q|\geq 2$ then
\eq{muiij}
\mu_m^Q-\mu_{mj}= \mu_m^{q_m-q_{m+p}} - \begin{cases}\mu_m\quad\text{if }j=m \\\mu_m^{-1}\quad\text{if }j=m+p \\1\quad\text{otherwise. }\end{cases}
\eeq
We need to verify the condition that the family of linear part  $\{ S_1, \ldots,  S_p\}$ is
of  the Poincar\'e type. So we can apply \rt{thm-conv-nf}.

Suppose that $(j,Q)\in\{1,\ldots, 2p\}\times\nn^{2p}$ satisfies
$
\mu_l^Q-\mu_{l j}\neq0
$
for some $1\leq l\leq 2p$. Set $d=\{\min_i\max(|\mu_i|,|\mu_{i}^{-1}|)\}^{1/{(2p)}}$.
 We define
$$
Q'=Q-\sum_{i=1}^p\min(q_i,q_{i+p})(e_i+e_{i+p}):=(q_1',\ldots, q_{2p}').
$$
Then $\mu_i^Q=\mu_i^{Q'}$ for all $i$. Take $i=l$ if $|Q'|\leq 2p$. In this case, we easily get
\eq{miqp9}
\mu_i^{Q'}-\mu_{ij}\neq0, \quad |\mu_i^{Q'}|>c^{-1}d^{|Q'|}
\eeq
by choosing a sufficiently large $c$. Assume that $|Q'|>2p$.
Take $i$ such that
$$
q'_{i}+q'_{i+p}=\max_k (q'_{k}+q'_{k+p}).
$$
Then $q'_i+q'_{i+p}\geq |Q'|/p>2$. By \re{muiij}, we get the first equality in \re{miqp9}.
We note that $(q'_i,q'_{i+p})=(q_i,0)$ or $(0,q_{i+p})$. Thus
$$
\max(|\mu_i^{Q'}|,|\mu_i^{-Q'}|)=(\max(|\mu_i|,|\mu_i|^{-1}))^{q'_i+q'_{i+p}}\geq d^{|Q'|}.
$$
This shows that $\{D\sigma_1(0),\ldots, D\sigma_p(0)\}$ is of the Poincar\'e type.

We now apply Theorem~\ref{thm-conv-nf} as follows.
We decompose $\psi=\psi_1\psi_0^{-1}$ such that $\psi_1\in\cL C^{\mathsf c}( S_1, \ldots,
 S_p)$ and $\psi_0\in\cL C( S_1, \ldots,  S_p)$. Then each $\sigma_i^*=\psi_1^{-1}\sigma_i\psi_1$
still has the form in \re{convnfsi}; in particular, $\{\sigma_1^*, \ldots, \sigma_p^*\}$ is a completely
 integrable formal normal form. By \rt{thm-conv-nf}, $\psi_1$ is convergent.  Now, $\psi_1^{-1}\tau_{kj}\psi_1
 =\psi_0^{-1}(\psi^{-1}\tau_{kj}\psi)\psi_0$ are still of the form \re{convnfsi}; however \re{lam1e}
 and $\Lambda_{2j}\Lambda_{1j}=1$ 
might not hold. 
As in the first proof, we can verify
 that $\psi_1\rho=\rho\psi_1$. Applying another change of coordinates that commutes with $\rho$ and
 each $ S_j$ as before, we achieve \re{convnfsi}-\re{lam1e} and $\Lambda_{2j}\Lambda_{1j}=1$.  The proof of the theorem is complete.
\end{proof}


As a corollary of \rt{abelinv}, we have the following normal form for real submanifolds. In order to study
the holomorphic flatness and hull of holomorphy,
 we choose a realization similar to the case
of Moser-Webster for $p=1$.
\begin{thm}\label{abelm}
Let $M$ be as in \rta{abelinv}.
Then $M$ is holomorphically equivalent to
\begin{gather}\label{Hmpp}
\widehat M\colon
z_{p+j}=\Lambda_{1j}(\zeta)\zeta_j,\quad 1\leq j\leq
p,
\end{gather}
where $\zeta=(\zeta_1,\ldots, \zeta_p)$ are the convergent solutions to
\begin{align}
\zeta_e&=A_e(\zeta)|z_e|^2-B_e(\zeta)(z_e^2+\ov z_e^2),\label{ze+L}\\
\zeta_s&=A_s(\zeta)z_s\ov z_{s+s_*}-B_s(\zeta)
(z_s^2+\Lambda_{1s}^2(\zeta)\ov z_{s+s_*}^2),\label{zs+L}\\
\label{zsss9}\zeta_{s+s_*}&=A_{s+s_*} \ov z_s z_{s+s_*}
- B_{s+s_*}(\zeta)
(z_{s+s_*}^2+\Lambda_{1(s+s_*)}^2(\zeta)\ov z_{s}^2).
\end{align}
Here $\Lambda_{1j}(\zeta)=\la_j+O(\zeta)$ $(1\leq j\leq p)$ satisfy \rea{lam1e} and
\begin{gather} \label{AeAj}
A_e=\frac{1+\Lambda_{1e}^{2}}{(1-\Lambda_{1e}^2)^{2}},\quad  A_j=\f{\Lambda_{1j}+\Lambda_{1j}^3}{(1-\Lambda_{1j}^2)^2}, \ j=s,s+s_*, \\
B_j=\frac{\Lambda_{1j}}{(1-\Lambda_{1j}^2)^{2}}, \quad j=e,s,s+s_*. \label{AeAj+}
\end{gather} 
In particular, $\widehat M$ is contained in $z_{p+e}=\ov z_{p+e}$ and  $z_{p+s}\tilde\Lambda_{1s}^{2}(z'')=\ov z_{p+s+s_*}$, where $\zeta_j=z_{p+j}\tilde\Lambda_{1j}(z''), 1\leq j\leq p,$ is the inverse mapping of $z_{p+j}=\zeta_j\Lambda_{1j}(\zeta),1\leq j\leq p$.
\end{thm}
\begin{proof}   We use a realization which is different from \re{realM2}.
We assume
that $M$ already has the normal form as in \rt{abelinv}. Thus for $j=1,\ldots, p$, we have
\eq{ta1j}
\tau_{1j}\colon\xi_j'
=\Lambda_{1j}(\xi\eta)\eta_j,\quad \eta_j'=\Lambda_{1j}^{-1}(\xi\eta)\xi_j, \quad (\xi_k',\eta_k')=(\xi_k,\eta_k), \quad k\neq j.
\eeq
 Let us define
$$
f_j(\xi,\eta)=\xi_j+\xi_j\circ\tau_{1j}, \quad
g_j=\ov{f_j\circ\rho}, \quad 1\leq j\leq p.
$$
The latter implies that the biholomorphic mapping $\varphi (\xi,\eta)= (f(\xi,\eta),
g(\xi,\eta))$ transforms $\rho$ into the standard
complex conjugation $(z',w')\to(\ov w',\ov z')$.
Define $$
F_{j}(\xi,\eta)=\xi_j\circ\tau_{1j}(\xi,\eta) \xi_j, \quad
1\leq j\leq p.
$$
 Using the expressions of $\tau_{1j}$ given by \re{ta1j}, we verify that  $f_j$ and $F_{j}$
are invariant by $\tau_{1k}$.
Note that the linear part of $f_j(\xi,\eta)$ is $\xi_j+\la_j\eta_j$
for $1\leq j\leq p$, and the
quadratic part of $F_{j}(\xi,\eta)$ is $\la_j\xi_j^2$.
By \rl{twosetin},  $f_1,\dots, f_p$ and $F_{1},\dots, F_{p}$
 generate all invariant functions
of $\{\tau_{11},\ldots,\tau_{1p}\}$.

Using $\ov{\Lambda_{1 e}\circ{\rho_z}}=\Lambda_{1e}$ and $\ov{\Lambda_{1s}\circ{\rho_z}}=\Lambda_{1(s+s_*
)}^{-1}$,
  rewrite $z_j=f_j(\xi,\eta), w_j=g_j(\xi,\eta)$ as
\begin{alignat*}{5}
 &\xi_e&&=\frac{z_e-\Lambda_{1e}(\xi\eta)w_e}{1-\Lambda_{1e}^2},
\quad &\eta_e&=\frac{w_e-\Lambda_{1e}(\xi\eta)z_e}{1-\Lambda_{1e}^2},\\
& \xi_s&&=\frac{z_s-\Lambda_{1s}^2(\xi\eta)w_{s+s_*}}{1-\Lambda_{1s}^2(\xi\eta)}, \quad
&\eta_s&=\frac{\Lambda_{1s}(\xi\eta)(w_{s+s_*}-z_s)}{1-\Lambda_{1s}^2(\xi\eta)},\\
 &\xi_{s+s_*}&&=\frac{z_{s+s_*}-\Lambda_{1(s+s_*)}^2(\xi\eta)w_{s}}{1-\Lambda_{1(s+s_*)}^2(\xi\eta)}, \quad
&\eta_{s+s_*}&=\frac{\Lambda_{1(s+s_*)}(\xi\eta)(w_{s}-z_{s+s_*}) }{1-\Lambda_{1(s+s_*)}^2(\xi\eta)}.
\end{alignat*}
Using the above formulas and  $w_j=\ov z_j$, we compute $\zeta_j=\xi_j\eta_j$ to
obtain  \re{ze+L}-\re{zsss9}.

Note that $F_j(\xi,\eta)=\zeta_j\Lambda_{1j}(\zeta)$. This shows
that $z_{p+j}=F_j\circ\varphi^{-1}(z',\ov z')$
have the form \re{Hmpp}. Again, we use the
formula of $\tau_{1k}$ to verify that $z=(z',z'')$
are invariant by all $ \var\tau_{1k}\var^{-1}$.   On the other hand, $z=(z',z'')$ generate invariant functions
of   the deck transformations
of $\pi_1$ for the complexification of $\hat M$ given by \re{Hmpp}.  This shows that $\{\var\tau_{11}\var^{-1}, \ldots, \var\tau_{1p}\var^{-1}\}$
and the deck transformations of $\pi_1$, of which each family consists of commuting involutions, have the same
invariant functions.  By \rl{twosetin}, we know that the two families must be identical.
This shows
that \re{Hmpp} is a realization for $\{\tau_{11},
\ldots, \tau_{1p},\rho\}$.

To verify the last assertion of the theorem, we first note that by \rea{lam1e}  the solutions $\zeta_1,\dots,\zeta_p$ to \re{ze+L}-\re{zsss9} satisfy $\zeta_e=\ov\zeta_e$ and $\zeta_{s+s_*}=\ov\zeta_s$.
By \rea{lam1e}, on $M$ we have $\ov z_{p+e}=\ov{\Lambda_{1e}(\zeta)}\ov\zeta_e=z_{p+e}$. Also
$\ov z_{p+s+s_*}=\ov{\Lambda_{1(s+s_*)}(\zeta)}\ov\zeta_{s+s_*}=\Lambda_{1s}^{-1}(\zeta)\zeta_{s}
=\Lambda_{1s}^{-2}(\zeta)z_{p+s}$. From \re{Hmpp}, solve
$\zeta_j=\tilde\Lambda_{1j}(z'')z_{p+j}$ $ (1\leq j\leq p)$ to obtain $\Lambda_{1j}^{-1}(\zeta)=\tilde\Lambda_{1j}(z'')$.
The proof is complete.
\end{proof}

\subsection{Hull of holomorphy for 
the  abelian CR singularity}
Let $X$ be a subset of $\cc^n$. We
define the hull of holomorphy of $X$, denoted by $\cL H(X)$,
to be the intersection of domains of holomorphy in $\cc^n$ that contain $X$.
Let $B^n_r$ be the ball in $\cc^n$ of radius $r$ and centered at the origin.

By \rt{abelm},  we assume that
 $M$ has pure elliptic type and it is equivalent to
\begin{gather*}
M\colon
z_{p+j}=\Lambda_{1 j}(\zeta)\zeta_j,\quad 1\leq j\leq
p,
\end{gather*}
where $\zeta_j=\zeta_j(z')$ $(j=1,\ldots, p)$ are the convergent real-valued solutions to \re{ze+L}.
For  $\zeta\in\rr^p$ with small $|\zeta|$, we know that  $\Lambda_{1 j}(\zeta)>1$.

Near  the origin in $\rr^p$, we define a real analytic  diffeomorphism:
$$
R\colon \zeta\to 
\left(\Lambda_{1 1}(\zeta)\zeta_1,\dots,\Lambda_{1 p}(\zeta)\zeta_p \right).
$$
 If $\e$ is small enough, for each $x''\in [0,\e]^{p}$, we can define $\zeta=R^{-1}(x'')$.
  Note that $R$ sends $\zeta_j=0$ into $x_{p+j}=0$ for each $j$. We can write
 $$
 R^{-1}(x'')=(x_{p+1}S_1(x''), \ldots, x_{2p}S_p(x''))$$
  with $S_j(0)>0.$
 Then
$M\cap\{z''=x''\}$ is given by \re{Hmpp}-\re{zsss9}.
For $x''\in[0,\e]^p$ let $D_j(x'')$ be the compact   set in  the complex plane whose boundary is defined
by the $j$th equation in \re{Hmpp}-\re{zsss9} where $\zeta=R^{-1}(x'')$.  When $x_{p+j}>0$,   the boundary
of $D_j(x'')$ is an ellipse with
\eq{djxpp}
D_j(x'')\subset B^1_{C_1\sqrt{x_{p+j}}}.
\eeq
Here and in what follows constants  will depend only on $\la_1,\ldots, \la_p$.
 Thus
$$
D(x''):=D_1(x'')\times\cdots\times D_p(x'')\times\{x''\}\subset\cc^p\times\rr^p
$$ is a product
of ellipses and its dimension equals the number of positive numbers among $x_{p+1},\ldots, x_{2p}$.
 We will call
$D(x'')$ an {\it analytic polydisc} and 
 $$\pd^*D(x''):=\pd D_1(x'')\times
\cdots\times\pd D_p(x'')\times\{x''\}$$ its distinguished boundary which
 is contained in $M$.  Set $D(0'')=\pd^*D(0'')=\{0\}$.
Thus,  $M$ is foliated by $\pd^*D(x'')$ as $x''$ vary in $[0,\e]^p$ and $\e$ is sufficiently small. Specifying the  $\e$ later, we will use this foliation
and Hartogs' figures in   analytic polydiscs  to find the local hull of holomorphy of $M$  at the origin.

As $x''$ vary in $[0,\e]^p$, let
$
M_\epsilon
$
be the union of $\pd^*D(x'')$, and  $\cL H_\e$  the union of  $D(x'')$. Both $\cL H_\e$ and $M_\e$ are
  compact
subsets in $\cc^{2p}$.  Note that
$$
B^{2p}_{\e_*}+M_\e:=\{a+b\colon a\in B^{2p}_{\e_*}, b\in M_\e\}
$$
is contained in a given neighborhood of $M_\e$, if $\e_*$ is sufficiently small.  Analogously,  $B^{2p}_{\e_*}+\cL H_\e$
is a   connected open neighborhood of $\cL H_\e$.
Let us first verify that  a function that
is holomorphic in a connected neighborhood of $M_\e$ in $\cc^{2p}$  extends holomorphically to a neighborhood of $\cL H_\e$ such that the extension agrees with the original function on a possibly smaller neighborhood of $M_\epsilon$.
 Assume that $f$ is holomorphic in a neighborhood $\cL U$ of $\pd^*_\e D:=\cup_{x''\in[0,\e]^p}\pd^*D(x'')$.
 Note that  $\cL H_\e$ is defined by
 \begin{align}\label{xpj}
A_j(x'')|z_j|^2  -B_j(x'')(z_j^2+\ov z_j^2)&\leq x_{p+j}, \quad 1\leq j\leq p;\\
y''=0, \quad x''&\in[0,\e]^p\label{y''x''}
\end{align}
with
\begin{gather}\nonumber
A_j(x'')=\frac{1+\Lambda_{1 j}^2(R^{-1}(x''))}{S_j(x'')(1-\Lambda_{1 j}^2(R^{-1}(x''))^2},
\quad
B_j(x'')=\frac{\Lambda_{1 j}(R^{-1}(x''))}{S_j(x'')(1-\Lambda_{1 j}^2(R^{-1}(x''))^2}.
\end{gather}
Let $\del$ be a small positive number.  For $x''\in[-\del,\e]^p$, let $D_j^\del(x'')\subset\cc$ be defined by
 \begin{align}\nonumber
A_j(x'')|z_j|^2  -B_j(x'')(z_j^2+\ov z_j^2)&\leq x_{p+j}+\del.
\end{align}
Let
$
P_{\e}^{\del}
$
(resp.~$\pd^*P_{\e}^{\del}$)
be the set of $z=(z',z'')$ such that
$y''\in[-\del,\del]^p$,  $x''\in[-\del,\e]^p$,  and $z_j\in D_j^{\del}(x'')$ (resp.~$z_j\in\pd D_j^{\del}(x'')$) for $1\leq j\leq p$.
Let $\cL U_{\e}^{\del}$ (resp. $\cL U_{\e}^{\del_1}$) be a small neighborhood of $P_{\e}^{\del}$ (resp. $P_{\e}^{\del_1}$).
Assume that $0<\del_1<\del$ and $\del_1$ is sufficiently small.  We may also assume that $\cL U_{\e}^{\del_1}$
is contained in $\cL U_{\e}^{\del}$ and  $\pd^*P_{\e}^{\del}\subset \cL U.$
  Thus, for $(z',z'')\in \cL U_{\e}^{\del_1}$, we can define \eq{Fzpzpp}
F(z',z'')= \f{1}{(2\pi i)^p}\int_{\zeta_1\in \pd D_1^{\del}(x'')}
\cdots\int_{\zeta_p\in \pd D^{\del}_p(x'')}\f{f(\zeta,z'')\, d\zeta_1\cdots d\zeta_p}{(\zeta_1-z_1)\cdots(\zeta_p-z_p)}.
\eeq
When $z$ is sufficiently small, $F(z)=f(z)$ as $f$ is holomorphic near the origin.  Fix $z_0\in\cL U_{\e}^{\del_1}$. We want to show that
  $F$ is holomorphic at $z_0$.  So $F$ is a desired extension of $f$.
By  continuity, when $z=(z_1,\dots, z_{2p})$ tends  to $z_0$, $x''$ tends to $x_0''$
and  $\pd D_j^\del(x'')$ tends to $\pd D_j^{\del}(x_0'')$,
while $z_j\in D_j^{\del}(x_0'')$ when $z$ is sufficiently close to $z_0$.
  By Cauchy theorem, for $z$ sufficiently close to $z_0$ we change 
   the repeated integral for
$  \zeta_j\in \pd D_j^{\del}(x_0'')$, $1\leq j\leq p$.
The domain of integration is thus fixed. The integrand is holomorphic in $z$. Hence   $F$ is holomorphic at $z=z_0$.

Next we want to show that $\cL H_\e$ is the hull of holomorphy of $M_\e$ in $B^{2p}_{\e_0}$ for suitable
$\e,\e_0$ that can be arbitrarily small.

Let us first show that
 $\cL H_\e$ is the intersection  of domains of holomorphy in $\cc^{2p}$.
  Recall that  $ \cL H_\e$ is defined by \re{xpj}-\re{y''x''}.
Define   for $\delta':=(\delta_1,\ldots,\delta_p)$ with $\del_j>0$ 
\begin{align*}
\rho_j^{\delta'}&=A_j(x'')|z_j|^2-B_j(x'')(z_j^2+\ov z_j^2)-x_{p+j}+(  \delta_1^{-1}+\cdots+ \delta_p^{-1})
\sum_{i=1}^py_{p+i}^2\\
&\quad +\sum_{i\neq j} \delta_i^{-1}
\left\{A_i(x'')|z_i|^2-B_i(x'')(z_i^2+\ov z_i^2)-x_{p+i}\right\}.
\end{align*}
When $p=1$, the last summation is $0$.
 The complex Hessian of $\rho_j^{\delta'}$ is
\begin{align*}
&\sum_{\all,\beta=1}^{2p}\frac{\pd^2\rho^{\delta'}_j}{\pd z_\alpha\ov z_\beta}t_\alpha\ov t_\beta=A_j(x'')|t_j|^2
+\frac{ \delta_1^{-1}+\cdots+ \delta_p^{-1}}{2}\sum_i |t_{p+i}|^2+\sum_{i\neq j}\frac{1}{ \delta_i}A _i(x'')|t_i|^2\\
&\qquad +\RE\sum_k a_{jk}(x''; z_j)t_j \ov t_{p+k} +
\sum_{k,\ell} b_{j,k\ell}(x''; z_j)t_{p+k} \ov t_{p+\ell} \\
&\qquad+\RE\sum_{i\neq j} \sum_k\frac{1}{ \delta_i} c_{j,ik}(x''; z_i)t_i \ov t_{p+k}
+\sum_{i\neq j} \sum_{k,\ell}\frac{1}{ \delta_i} d_{j,k\ell}(x''; z_i)t_{p+k} \ov t_{p+\ell}.\end{align*}
Here $a_{jk}(x'';0)=b_{j,kl}(x'';0)= c_{j,ik}(x'';0)=d_{j,kl}(x'';0)=0$,
and  $i,j,k,\ell$ are in $\{1,\ldots, p\}$.  From the Cauchy-Schwarz inequality,
it follows  that for $ z\in B^{2p}_{e_0}$ with $\e_0>0$ sufficiently
small and  $0<\delta_j<1$,
\begin{align*}
2\sum_{\all,\beta=1}^{2p}\frac{\pd^2\rho^{\delta'}_j}{\pd z_\alpha\ov z_\beta}t_\alpha\ov t_\beta&\geq A_j(x'')|t_j|^2
+\frac{ \delta_1^{-1}+\cdots+ \delta_p^{-1}}{2}\sum_j|t_{p+j}|^2+  \sum_{i\neq j} \delta_i^{-1}A_i(x'')|t_i|^2.
\end{align*}
Therefore, each $\rho_j^{\delta'}$ is  strictly plurisubharmonic on $|z|<\e_0$ for all $0< \delta_i<1$.
 Hence for $ \del_*=(\del_0,\ldots, \del_p)=(\del_0,\del')\in(0,1)^{p+1}$,
$$
\rho_\e^{\delta_*}(z)=\max_j\{\rho_j^{\delta'},|y''|^2- \delta_0^2, x_{p+j}^2-\e^2\}
$$
is plurisubharmonic on $B^{2p}_{\e_0}$.   By \re{djxpp},
$ D(x'')$ is contained in $B^{2p}_{C_2\e^{1/2}}$ for $x''\in[0,\e]^p$.   We now fix $\e<(\e_0/{C_2})^2$ to ensure
\eq{dxbe}
D(x'')\subset B_{\e_0}^{2p}, \quad \forall x''\in[0,\e]^p.
\eeq
This shows that
$
 \cL H^{\del_*}_\e:=\{z\in  B_{\e_0}^{2p}\;|\;\colon \rho_\e^{\del_*}(z)<0\}
$
is a domain of holomorphy.

Let us verify that
$ 
\cL H_{\e}=\bigcap_{\e'>\e, \delta_0>0,\ldots,\del_p>0} \cL H^{\del_*}_{\e'}.
$ 
Fix $z\in \cL H_\e$. From \re{dxbe} we get $z\in B_{\e_0}^{2p}$.
We have $y''=0$.  Hence \re{xpj} hold and $x_{p+j}^2\leq\e^2$. Clearly, $\rho_j^ \delta(z)<0$ for
each $j$ and $ \delta\in(0,1)^p$. This shows that $z\in \cL H_\e$ is in the intersection.  For the other inclusion,
let us assume that $z$ is in the intersection. Then $y''=0$.  
 With $\rho_j^{\del_*}(z)<0$, we let $ \delta_i$ tend to $0$ for $i\neq j$.
We conclude
\begin{gather}\nonumber
A_i(x'')|z_i|^2-B_i(x'')(z_i^2+\ov z_i^2)\leq x_{p+i}
\end{gather}
for all $i\neq j$, and hence for all $i$ as $p>1$.
 When $p=1$ the above inequality can be obtained directly from
$\rho_1^{\del_*}$.  We also see  that $0\leq x_{p+j}\leq\e$. We have verified \re{xpj} and \re{y''x''}. This shows that $z\in\cL H_\e$.

In view of \re{xpj}-\re{y''x''}, the boundary of $\cL H_\e$ is the union $\cup_{j=1}^p\cL H_j^\e$ with  $\cL H_j^\e$
being defined by
\begin{alignat*}{2}
A_j(x'')|z_j|^2 -B_j(x'')(z_j^2+\ov z_j^2)&= x_{p+j},\\
A_i(x'')|z_i|^2 -B_i(x'')(z_i^2+\ov z_i^2)&\leq x_{p+i}, \quad 1\leq i\leq p,\  i
\neq j;\\
y''=0, \quad\qquad x_{p+ i}&\leq\e, \qquad\ \  1\leq  i\leq p.
\end{alignat*}
Therefore, we have proved the following theorem.
\begin{thm}\label{hullj}
 Let $M$ be a germ of real analytic submanifold at an abelian CR singularity. Assume that the complex tangent of $M$
is purely elliptic 
 at the origin. There is a base of neighborhoods $\{U_j\}$ of the origin in $\cc^n$ which satisfies the following: For each $U_j$,  the hull of holomorphy $H(M\cap U_j)$ of $M\cap U_j$ is foliated by embedded complex submanifolds with boundaries. Furthermore, near the origin $H(M\cap U_j)$
is the transversal intersection of $p$ real analytic submanifolds of dimension $3p$ with boundary.
 The boundary of $H(M\cap U_j)$ contains $M\cap U_j$; and two sets are the same if and only if $p=1$.
\end{thm}

 \begin{rem} The proof shows that the hull  $H(M\cap U_j)$ is foliated by analytic polydiscs, i.e. holomorphic embeddings
of closed unit polydisc in   $\cc^p$.
\end{rem}

\setcounter{thm}{0}\setcounter{equation}{0}

\section{Rigidity of product quadrics}\label{rigidquad}

The aim of this section is to prove the following rigidity theorem: Let us consider
a higher order analytic perturbation of a product quadric. If this manifold is formally equivalent to the product quadric, then under a small divisors  condition, it is also holomorphically equivalent to it.
Notice that when $p>1$, there are real submanifold $M$ with a linearizable $\sigma$ such
that $M$ is not formally equivalent to the quadric, or equivalently, the $\{\tau_{1j},\rho\}$ is not formally linearizable; see~\cite{part2}.

The proof goes   as follows~: Since the manifold is formally equivalent to the quadric,   the associated   involutions $\{\tau_{1 i}\}$ and $\{\tau_{2 i}\}$ are simultaneously linearizable by a formal biholomorphism that commutes with $\rho$. In particular,  $\sigma_1,\ldots, \sigma_p$,  as defined by \re{sigma_i} and \re{sigma_i+}, are formally linearizable and  they   commute pairwise. These are germs of biholomorphisms with a diagonal linear part. According to \cite{stolo-bsmf}[theorem 2.1], this abelian family can be holomorphically linearized under a collective Brjuno type condition \re{bnI}. Furthermore, the transformation commutes with $\rho$.
Then, we linearize simultaneously and holomorphically both $\tau_1:=\tau_{1 1}\circ\cdots\circ \tau_{1 p}$ and $\tau_2:=\tau_{2 1}\circ\cdots\circ \tau_{2 p}$ by a transformation that commutes with both $\rho$ and $\cL S$, the family of linear parts of the $\sigma_1,\ldots,\sigma_p$.
Finally, we linearize simultaneously and holomorphically both families $\{\tau_{1 i}\}$ and $\{\tau_{2 i}\}$ by a transformation that commutes with $\rho$, $\cL S$, $T_1$ and $T_2$.

These last two steps will be obtained through a majorant method and the application of a holomorphic implicit function theorem. This is obtained in
Proposition~\ref{linearS}. They first require a complete description of the various centralizers and their   associated normalized  mappings, i.e. suitable complements. This is a goal of Proposition~\ref{STiR}.

Throughout this section, we do not assume that  $\mu_1,\ldots, \mu_p$ are non resonant in the
sense that $\mu^Q\neq1$ if $Q\in\zz^p$ and $Q\neq0$.  In fact, we will apply our results to $M$
which might be resonant.
However, we will retain the assumption that $\sigma$ has distinct
eigenvalues when we apply the results to the manifolds.

\subsection{Centralizers}
We recall from \re{sigma_i} and \re{sigma_i+}, the definition and property of germs of holomorphic diffeomorphisms~:
$\sigma_i:=\tau_{1i}\circ\tau_{2i}$, $\sigma_i^{-1}=\rho\sigma_i\rho$, $1\leq i\leq e_*+h_*$; $\sigma_s:=\tau_{1s}\circ\tau_{2 (s_*+s)}$, $\sigma_{s+s_*}:=\tau_{1(s+s_*)}\circ\tau_{2s}$, $\sigma_{s+s_*}^{-1}=\rho\sigma_s\rho$, $e_*+h_*<s\leq p-s_*$.
Recall the  linear maps
\begin{align}
&S\colon\xi_j'=\mu_j\xi_j,\quad \eta_j'=\mu_j^{-1}\eta_j;\nonumber\\
\label{rSxi-}
&S_j\colon\xi_j'=\mu_j\xi_j,\quad \eta_j'=\mu_j^{-1}\eta_j,\quad \xi_k'=\xi_k,\quad \eta_k'=\eta_k,
\quad k\neq j;\\
\label{rTij-}
&T_{ij}\colon\xi_j'=\la_{i j}\eta_j,\quad\eta_j'=\la_{i j}^{-1}\xi_j,\quad\xi_k'=\xi_k,\quad
\eta_k'=\eta_k, \quad k\neq j; 
 \intertext{while $\rho$ is given by}
\label{rRho-}
& \rho\colon \left\{\begin{array}{ll}
(\xi_e',\eta_e',\xi_h',\eta_h')=
(\ov\eta_e,\ov\xi_e,\ov\xi_h,\ov\eta_h),\vspace{.75ex}
\\
(\xi_{s}', \xi_{s+s_*}',\eta_{s}',\eta_{s+s_*}')=(\ov\xi_{s+s_*},
\ov\xi_{s},\ov\eta_{s+s_*}, \ov\eta_{s}).
\end{array}\right.
\end{align}
Recall that indices $e.h,s$ have the ranges $1\leq e\leq e_*, e_*<h\leq e_*+h_*$, and $e_*+h_*<s\leq p-s_*$.
The basic conditions on $\mu_j=\la_j^2$ are the following:
\eq{resonant-mu}
|\mu_h|=1, \quad \mu_{s+s_*}=\ov\mu_s^{-1},\quad \mu_e>1, \quad |\mu_s|\geq1, \quad \mu_s^k\neq 1,  k=1,2,\dots
.
\eeq
In particular, a $\mu_j$ may be repeated  and $\mu_h$ can be $1$.

We need to introduce   notation for the indices to describe various centralizers
regarding $T_{1j}, S_j$ and $\rho$.
We first introduce index sets for the centralizer of $\cL S:=\{S_1,\ldots, S_p\}, T_1:=T_{i 1}\circ\cdots \circ T_{i p},\rho$. We recall that 
$\cL T_i:=\{T_{i 1}, \ldots, T_{ip}\}$.

Let  $(P,Q)\in\nn^p\times\nn^p$ and $1\leq j\leq p$. By definition, $\xi^P\eta^Qe_j$ belongs to the centralizer of $\cL S$ if and only if it commutes with each $S_i$. In other words,  $\xi^P\eta^Qe_j\in\cL C(\cL S)$ if and only if
\eq{mukpk}
 \mu_k^{p_k-q_k}=1,\quad \forall k\neq j; \quad\mu_j^{p_j-q_j}=\mu_j.
\eeq
Note that the same condition holds for  $\xi^Q\eta^Pe_{p+j}$ to belong to $\cL C(\cL S)$. This leads us to define the set of multiindices
$$
\cL R_j:=\{(P,Q)\in\nn^{2p}\colon\mu_j^{p_j-q_j}=\mu_j,\   \mu_i^{p_i-q_i}=1, \forall i\neq j\}, \quad 1\leq j\leq p.
$$
We observe that if $(P,Q)\in\cL R_j$, then \re{resonant-mu} implies that
\begin{gather}
p_j=q_j+1, \quad j\neq h; \quad p_i=q_i, \quad \forall i\neq j,h;\nonumber\\
 \la_h^{p_h-q_h}=\pm1, \quad h\neq j;\quad \la_j^{p_j-q_j-1}=\pm1, \quad j=h.\label{rj}
\end{gather}
Here we have used the assumption that $\mu_s$ are not root of unity, which simplifies greatly the results and computation in this section.

For convenience, we define for $P=(p_e,p_h,p_s,p_{s+s_*})$ and $Q=(q_e,q_h,q_s,q_{s+s_*})$
\al
\rho( P Q) &:=
( q_e ,p_h, p_{s+s_*}, p_{s}, p_e, q_h, q_{s+s_*}, q_{s}),
\nonumber
\\
\label{rhoapq}\rho_a(PQ) &:= (q_e ,p_h, p_{s+s_*}, p_{s}),\quad
\rho_b(PQ) 
:=(p_e, q_h, q_{s+s_*}, q_{s}),\\
\ov f_{\rho( P Q)} &:=(\ov{f\circ\rho})_{ P Q}. 
\label{compos-rho}
\end{align}
Here $p_e=(p_{1},\ldots,p_{e_*})$ denotes the ``elliptic coordinates'' of $P$.
Hence,
\begin{equation}\label{compos-rho+}
 \rho( P Q)=(\rho_a( P Q),\rho_b( P Q))=(\rho_b( QP),\rho_a( QP)).
\end{equation}

According to \re{mukpk} and  equation $(\ref{rRho-})$ of $\rho$,
the restriction of $\rho$ to  $\cL R_h$ is an involution, which will be denoted by $\rho_h$. Moreover, $\rho$ is  a bijection  $\rho_s$
from $\cL R_s$ onto $\cL R_{s+s_*}$.   We define an involution on $\cL R_e$ by
\eq{rhoepq}
\rho_e(PQ):=(\rho_b(PQ),\rho_a(PQ))
=(p_e,q_h,q_{s+s_*},q_s,q_e,p_h,p_h,p_{s+s_*},p_s).
\eeq
Note that  $\rho_e$ is not a restriction of $\rho$, and $\rho_s$ is not an involution either.

Next, we introduce sets of indices to be used to compute the centralizers on $\cL T_1,\cL T_2,\rho$. Set
\gan
\cL N_j:=\cL R_j\cap\{(P,Q) \colon p_i\geq q_i, \quad \forall i\neq j\}, \quad 1\leq j\leq p.
\end{gather*}
 Since there is no restriction for $p=1$, we have $\cL N_j=\cL R_j$ for $j=e$ or $h$.
Let us set
\gan
A_{jk}( P, Q):=
\max\{ p_k, q_k\}, \quad k\neq j,\quad
A_{jj}( P, Q)=
 p_j;\\
B_{jk}( P, Q):=
\min\{ p_k, q_k\}, \quad k\neq j,\quad
B_{jj}( P, Q)=
 q_j.
\end{gather*}
We define a mapping
$$
(A_j,B_j)\colon \cL R_j\to\cL N_j
$$
with
$
A_j:=(A_{j1},\ldots, A_{jp})$ and $ B_j:=(B_{j1},\ldots, B_{jp}).
$
For $(P,Q)\in\cL N_j$   with  $j=e,h$, we have $A_j\circ \rho_j(P,Q)=(p_e,p_h,p_{s+s_*}, p_s)$ and $B_j\circ \rho_j(P,Q)=(q_e,q_h,q_{s+s_*}, q_s)$. In other words, on  $\cL N_j$
 for $j=e$ or $h$, $A_j\circ\rho_j$ just interchanges the $s$th   and  the $(s+s_*)$th coordinates for each $s$, so does $B_j\circ\rho_j$,   while $A_{s+s_*}\rho_s$ and $B_{s+s_*}\rho_s$ have the same property on $\cL N_s$.
 Furthermore,
 \ga
  \label{commut-arho}
 (A_{h},B_{h})\rho = \rho(A_{h},B_{h})\quad \text{on $\cL R_h$,}\\ 
(A_{s+s_*},B_{s+s_*})\rho = \rho(A_{ s},B_{s})\quad  \text{on $\cL R_s$}.   \label{commut-arho+}
 \end{gather}

Finally,   with the convention that the product over an empty set is $1$, 
we define, for $(P,Q)\in\cL R_j$~:
\al \label{nujab}
\nu_{ P Q}& :=\begin{cases}\prod_{h'}\la_{h'}^{ q_{h'}- p_{h'}}, & j\neq h, \\
  \la_h^{p_h-q_h-1}\prod_{h'\neq h }\la_{h'}^{ q_{h'}- p_{h'}},& j=h;
\end{cases}
\\
\nu_{ P Q}^+ &:=\begin{cases}
\prod_{h'| q_{h'}> p_{h'}}\la_{h'}^{ q_{h'}- p_{h'}},   & j\neq h;\\
\label{nujab+}\prod_{h'\neq h, q_{h'}> p_{h'}}\la_{h'}^{ q_{h'}- p_{h'}}, & j=h.
\end{cases}
\end{align}
Here $e_*<h',h\leq e_*+h_*$.
For convenience, we however define
$$
\nu_{QP}:=\nu_{PQ}, \quad (P,Q)\in\cL R_j.
$$
If $p=1$ we set $\nu_{P Q}^{ +}=1$.
 \begin{lemma}\label{nunu+} Let $(P,Q)\in\cL R_j$. Then  $\la_j^{-1}\la^{P-Q}=\nu_{PQ}$, and
\begin{alignat}{3} \label{nupm1}
 &\nu_{PQ}=\pm1;\quad  \nu_{PQ}^+=\pm1;
 \quad &&\nu_{PQ}^+=1, \quad
 (P,Q)\in\cL N_j;\\
 \label{nupm1+}
 &\nu_{\rho_e(PQ)}=\nu_{PQ}, \  j=e; \quad  \  &&\nu_{\rho(PQ)}=\nu_{PQ};
\\
 \label{nupm1++}
 & \nu_{\rho_e(PQ)}^+=\nu_{PQ}^+\nu_{PQ}, \  j=e;  \quad  &&\nu_{\rho(PQ)}^+=  \nu_{PQ}^+.
 \end{alignat}
\end{lemma}
\begin{proof}   The first identity follows from the definition of $\nu_{PQ}$ and $\cL R_j$.
 From the definition
of $\cL R_j$, we have $(\la_i^{p_{i}-q_{i}})^2=\mu_i^{p_{i}-q_{i}}=1$ for $i=h'$
in \re{nujab}-\re{nujab+}.
We also have $\mu_h^{p_{h}-q_{h}-1}=1$ for terms in \re{nujab}-\re{nujab+}. Thus
\eq{Lhph}
\nonumber
\la_{h'}^{p_{h'}-q_{h'}}=\pm1,\quad \la_h^{p_h-q_h-1}=\pm1.
\eeq
    Thus  we obtain \re{nupm1}; the rest identities
follow from the definition of $\rho_e$, $\rho$, and the above identities.    \end{proof}

\begin{lemma}\label{comput-rhoT1}
  For all multiindices  $(P,Q)$, $\xi^P\eta^Q\circ \rho = \ov\xi^{\rho_a(PQ)}\ov\eta^{\rho_b(PQ)}$.
For all multiindices  $(P,Q)\in  \cL R_e\cup\cL R_h$, we have
\ga\label{lambda}
\ov{\la^{\rho_a( P, Q)-\rho_b( P, Q)}}=
\la^{ Q- P},\quad \ov{\mu^{ \rho_b-\rho_a}}=\mu^{P-Q},\\
\label{T1rho}
\xi^P\eta^Q\circ \rho\circ T_1= \la^{ Q- P}\ov\xi^{\rho_b(PQ)}\ov\eta^{\rho_a(PQ)},\\
\label{rhoS-1}
\xi^P\eta^Q\circ \rho\circ S^{-1}= \mu^{ P-Q}\ov\xi^{\rho_a(PQ)}\ov\eta^{\rho_b(PQ)}.
\end{gather}
\end{lemma}
\begin{proof}
Identity \re{lambda} follows from \re{rhoapq} 
 and the fact that $\la_e$ and $\mu_e$ are reals, $\la_h^{-1}=\ov\la_h$,  $p_s=q_s$, and  $p_{s+s_*}=
q_{s+s_*}$. 
A direct computation shows that
$$
\xi^P\eta^Q\circ \rho\circ T_1= \ov\la^{ \rho_a- \rho_b}\ov\xi^{\rho_b(PQ)}\ov\eta^{\rho_a(PQ)},\quad \xi^P\eta^Q\circ \rho\circ S^{-1}= \ov\mu^{ \rho_b-\rho_a}\ov\xi^{\rho_a(PQ)}\ov\eta^{\rho_b(PQ)}.
$$
The result follows from \re{lambda}.
\end{proof}

Finally we note that
\eq{iotemove}
\iota_e\colon (P,Q)\to (A_e,B_e)\circ\rho_e(PQ)=(A_e,B_e)(\rho_b(P,Q),\rho_a(P,Q))
\eeq
defines an involution on $\cL N_e$.
 We now can describe the centralizers.
\begin{prop}\label{STiR} Let $\cL S=\{S_1,\ldots, S_p\}$,  $\cL T_i=\{T_{i1},\ldots, T_{ip}\}$
and $\rho$ be given by \rea{rSxi-}-\rea{rRho-}.
Let $\var=I+(U,V)$ be a formal biholomorphic map
that is tangent to the identity.
\bppp
\item $\var\in{\cL C}(\cL S)$ if and only if
\ga\label{muimuj}
U_{j,PQ}=0=V_{j,QP}, \quad\forall (P,Q)\not\in\cL R_j.
\end{gather}
Also, $\var\in{\cL C}(\cL S,\rho)$ if and only if \rea{muimuj} holds and
\ga
\label{uhus} U_{h,PQ}= \ov{U}_{h,\rho(PQ)}, \  (P,Q)\in\cL R_h;
\quad U_{s+s_*,PQ}=\ov{U}_{s,\rho(PQ)},\ (P,Q)\in\cL R_{s+s_*};\\
 V_{e,QP}=\ov U_{e, \rho_e(PQ)},\quad (P,Q)\in\cL R_e;
 \label{veqpo}\\
V_{h,QP}=\ov V_{h,\rho(QP)},\  (P,Q)\in\cL R_h; \quad V_{s+s_*,QP}=\ov{V}_{s,\rho(QP)}, \ (P,Q)\in\cL R_{s+s_*}.
\label{uhqp}
\end{gather}
\item $\var\in{\cL C}(\cL S, T_1)$ if and only if \rea{muimuj} holds and
\eq{vjlju}
V_{j,QP}= \nu_{PQ}
U_{j,PQ}, \quad \forall (P,Q)\in\cL R_j.
\eeq
\item $\var\in{\cL C}(\cL S,T_1,\rho)$ if and only if  \rea{muimuj}, \rea{uhus}
 and \rea{vjlju} hold, and
\begin{alignat}{4}\label{uenu}
 U_{e,PQ}&= \nu_{PQ}\ov{U}_{e, \rho_e(PQ)},\quad && (P,Q)\in\cL R_e.
\end{alignat}
\item
 Let $p>1$.
$\var\in{\cL C}(\cL T_1,\cL T_2)$ if and only if    \rea{muimuj} and \rea{vjlju} hold, and
\begin{eqnarray}
\label{ujrjnj}
U_{j, P Q}&=& \nu_{P Q}^+
U_{j,(A_j,B_j)( P, Q)},\quad (P,Q)\in\cL R_j\setminus\cL N_j. 
\end{eqnarray}
Also, $\var\in{\cL C}(\cL T_1,\cL T_2,\rho)$ if and only if
\rea{muimuj}, \rea{vjlju} and \rea{ujrjnj}  hold, and 
\begin{alignat}{4}
\label{sqnupq}
U_{e, P Q}&=  
\ov U_{e , (A_e,B_e)
 \rho_e( PQ)},\quad &&(P,Q)\in  \cL N_e,
\\   
\label{sqnupq1} U_{h,  P Q}&=\ov U_{h,\rho( PQ)},\quad  && (P,Q)\in \cL N_h,\\
 U_{s+s_*,  P Q}&=  
 \ov U_{s,
  \rho( P Q)}, \quad  &&(P,Q)\in  \cL N_{s+s_*}.
\label{sqnupq2}
\end{alignat}
\eppp
\end{prop}
 We remark that condition \re{ujrjnj}   holds trivially when $(P,Q)\in \cL N_j$, in which case it becomes
 $U_{j,PQ}=U_{j,PQ}$.
\begin{proof}
 To simplify  notation,
we abbreviate
$$
\rho_a=\rho_a(PQ), \quad\rho_b=\rho_b(PQ), \quad A_j=A_j(P,Q), \quad B_j=B_j(P,Q).
$$

Recall that $\la_e=\ov\la_e, \la_h=\ov\la_h^{-1}$ and $\la_{s+s_*}=\ov\la_s^{-1}$. By definition,
\ga\label{rSet-}\nonumber
S_e=T_{1e}T_{2e}, \quad S_h=T_{1h}T_{2h}, \quad S_{s}=T_{1s}T_{2(s+s_*)}, \quad S_{s+s_*}=T_{1(s+s_*)}T_{2s}.
\end{gather}
In the proof, we will   use the fact
  that $S_j$ is reversible by both involutions in the composition for $S_j$. In particular,
\eq{tijsj}
T_{1j}S_jT_{1j}=S_j^{-1},\quad \forall j.
\eeq
 However, we have $T_{2(s+s_*)}S_sT_{2(s+s_*)}=S_s^{-1}$
and $T_{2s}S_{s+s_*}T_{2s}=S_{s+s_*}^{-1}$.
For simplicity,  we will  derive identities by using \re{tijsj} and
\eq{se-1}
 S_e^{-1}=\rho S_e\rho, \quad S_h^{-1}=\rho S_h\rho, \quad S_{s+s_*}^{-1}=\rho S_s\rho.
\eeq
 Finally, we need one more identity.
Recall that
\begin{alignat*}{4}
&T_{1e}T_{2j}=T_{2j}T_{1e}, \quad j\neq e;\quad
&&T_{1h}T_{2j}=T_{2j}T_{1h}, \quad j\neq h;\\
&T_{1s}T_{2j}=T_{2j}T_{1s}, \quad j\neq s+s_*;\quad
&&T_{1(s+s_*)}T_{2j}=T_{2j}T_{1(s+s_*)}, \quad j\neq s.
\end{alignat*}
Therefore, for any $j$ we have the identity
\eq{sjt1}
T_1S_j T_1=S_j^{-1}.
\eeq
In what follows, we will derive all identities by using \re{tijsj}, \re{se-1} and \re{sjt1}, as well
as $S_iS_j=S_jS_i$,  $T_{1i}T_{1j}=T_{1j}T_{1i}$ and $T_2=\rho T_1\rho$.

(i)
 The centralizer of $\cL S$ is easy to describe.
Namely, $\var\in\cL C(\cL S)$  if and only if
\gan
\label{ujs=-}
U_j\circ S_j=\mu_jU_j, \quad U_j\circ S_k=U_j, \quad k\neq j,\\
\quad V_j\circ S_j=\mu_j^{-1}V_j,\quad V_j\circ S_k=V_j, \quad k\neq j.\label{vjsj}
\end{gather*}
 For $\var\rho=\rho\var$, we need
\ga\label{uhou}
  U_h=\ov {U_h\circ\rho},
\quad U_{s+s_*}=\ov{U_s\circ\rho},\\
V_{e}=\ov{U_e\circ\rho},
 \quad V_h=\ov {V_h\circ\rho}, \quad V_{s+s_*}=\ov{V_s\circ\rho}.
\label{ueve}
\end{gather}
Hence, using \re{compos-rho}-\re{rhoepq}, we have $\ov U_{e,PQ}
=V_{e,\rho(PQ)}=V_{e,\rho_e(QP)}$. The other equalities are obtained in the same way.

(ii) 
 If $\var\in\cL C(\cL S,T_1)\subset\cL C(S,T_1)$, then it satisfies
\begin{gather}\label{ujs=}
V_j=\lambda_{j}^{-1}U_{j}\circ T_{1}.
\end{gather}
This implies $(\ref{vjlju})$.

(iii) Assume furthermore that $\var\in \cL C(\cL S, T_1, \rho)$. Eliminating $V_e$ 
from \re{ujs=} and \re{ueve}, we obtain
 \eq{uele}
 \nonumber
 U_e=\la_e\ov{U_e\circ\rho\circ T_1}.
 \eeq
According to \re{T1rho}, we obtain   \re{uenu} by
$$
U_{e,\rho_b\rho_a}=\la_e\ov\la^{Q-P}\ov U_{e,PQ}=\nu_{PQ}\ov U_{e,PQ}.
$$

%
%
%

(iv)  Let $\var\in\cL C(\cL T_1,\cL T_2)$. Then, in particular, we have
\aln 
U_{j}=U_j(T_{1k}), \quad k\neq j; \quad
V_j=\la_j^{-1}U_j\circ T_1.
\end{align*}
Let $(P,Q)\in \cL R_j\setminus \cL N_j$.
We compose $U_j$ successively by each $T_{1 k}$ if $q_k>p_k$.
We emphasize that 
 such a $k$ is a hyperbolic index.  The previous identity yields
\eq{uL}
U_{j,PQ}=L_{j,PQ}U_{j,A_jB_j}, \quad L_{j, P Q}:=\prod_{k\neq j,  p_k< q_k}\la_{k}^{ q_k- p_k}.
\eeq
By the definition of $\nu_{PQ}^+$, we conclude 
\eq{Ljpqn}
L_{j,PQ}=\nu_{PQ}^+, \quad (P,Q)\in\cL R_j. 
\eeq
If $(P,Q)\in  \cL N_j$, then $(A_j,B_j)=(P,Q)$ and we   have $L_{j,PQ}=\nu_{PQ}^+=1$, so that the   relation \re{uL} just becomes the identity $U_{j,PQ}=U_{j,PQ}$.


Assume now that $\var\in\cL C(\cL T_1,\cL T_2,\rho)$. In addition to the previous conditions, we have \re{uhou} and \re{ueve}.
Hence, \re{uhus}, \re{uenu} and \re{uL} lead to:
\begin{alignat*}{4}
\nu_{PQ}\ov U_{e,\rho_e(PQ)}&=U_{e,PQ}=L_{e,PQ}U_{e,A_eB_e},\quad && (P,Q)\in \cL R_e;\\
\ov U_{h,\rho_h(PQ)}&=U_{h,PQ}=L_{h,PQ}U_{h,A_hB_h},\quad &&(P,Q)\in \cL R_h;\\
\ov U_{s+s_*,\rho(PQ)}&=U_{s,PQ}=L_{s,PQ}U_{s,A_sB_s},\quad &&(P,Q)\in \cL R_s.
\end{alignat*}
Since $\rho_e, \rho_h$ are involutions on  $\cL R_e$ and $\cL R_h$, respectively,
and since $\rho$ is a bijection from $\cL R_s$ onto $\cL R_{s+s_*}$, we obtain
\begin{alignat*}{4}
\nu_{\rho_e(PQ)}\ov U_{e,PQ}&=L_{e,\rho_e(PQ)}U_{e,(A_e,B_e)\circ\rho_e(PQ)},\quad && (P,Q)\in \cL R_e;\\
\ov U_{h,PQ}&=L_{h,\rho_h(AB)}U_{h,(A_h,B_h)\circ\rho_h(PQ)},\quad && (P,Q)\in \cL R_h;\\
\ov U_{s+s_*,PQ}&=L_{s,\rho(AB)}U_{s,(A_s,B_s)\circ\rho(PQ)},\quad && (P,Q)\in \cL R_{s+s_*}.
\end{alignat*}
By \re{Ljpqn}, we copy  the values $L_{j,\rho( P Q)}=\nu_{\rho(PQ)}^+$ from  \re{nupm1++}. We have
\begin{alignat*}{4}
\nu^+_{\rho_j(PQ)}&=\nu^+_{PQ}, \quad &&\text{if $j \neq e
$, and $ (P,Q)\in\cL R_j$;}\\
 \nu^+_{\rho_e(PQ)}&=\nu_{PQ}\nu_{PQ}^+,
 \quad && \text{if $(P,Q)\in\cL R_e$};\\
 \nu_{\rho_e(PQ)}&=\nu_{PQ}, \quad &&\text{if $(P,Q)\in\cL R_e$}.
\end{alignat*}
Finally, we obtain
\begin{eqnarray*}
U_{j,PQ}&=& \nu_{PQ}^+ 
\ov U_{j,(A_j,B_j)\circ\rho_j(PQ)},\quad (P,Q)\in \cL R_j, \quad j=e,h; \\
U_{s+s_*,PQ}&=&\nu_{PQ}^+\ov U_{s,(A_s,B_s)\circ\rho(PQ)},\quad (P,Q)\in \cL R_{s+s_*}.
\end{eqnarray*}
Therefore, we have derived necessary conditions for the centralizers. 
Let us verify that the conditions   are also sufficient. 
Of course, the verification for (i) is straightforward.
Furthermore,  that $\var=I+(U,V)$ commutes with $S_1, \ldots, S_p$
is equivalent to $U_{j,PQ}=V_{j,QP}=0$ for all $(P,Q)\in \cL R_j$, which
is also trivial in cases (ii) and (iii).

For (ii),  \re{muimuj} and \re{vjlju} imply that $\var$ commutes with $T_1$. We  verify that $\var$ commutes
with $\rho$. In other words, \re{veqpo} and \re{uhqp} hold.  The latter follows immediately from \re{uhus} and \re{vjlju}.
For the former, take $(P,Q)\in \cL R_e$. By \re{vjlju} and \re{uenu}, we get $V_{e,QP}=\nu_{PQ}U_{e,PQ}=\ov U_{e,\rho_e(PQ)}$, which is \re{veqpo}.


For (iii),  let us  verify that \re{ujrjnj}, 
 \re{muimuj}, and \re{vjlju} are sufficient conditions for $\var\in\cL C(\cL
T_1,\cL T_2)$. By \re{vjlju}, we get $\var T_1=T_1\var$.
 Also,  for $\var\in\cL C(\cL T_1)$ it remains to show that for $(P,Q)\in\cL R_j$
 \eq{uit1j}
 (U_j\circ T_{1k})_{PQ}=U_{j,PQ}, \quad k\neq j; \quad (U_{j}\circ T_{1j})_{QP}=\la_jV_{j,QP}.
 \eeq
   We introduce $(P_j,Q_j)$ via $\xi^P\eta^Q\circ T_{1j}=\la_j^{p_j-q_j}\xi^{P_j}\eta^{Q_j}$
 and also denote $(P_j,Q_j)$ by $(P,Q)_j$.
We
 remark that \re{ujrjnj} 
  also holds for $(P,Q)\in\cL N_j$.
Therefore, we will use \re{ujrjnj} 
for all $(P,Q)\in\cL R_j$.

  For $k\neq j, h$, we  have $(P_k,Q_k)=(P,Q)$. Thus in this case we immediately
 get the first identity in \re{uit1j}.
  Using \re{ujrjnj} twice, we obtain for $j\neq h$
 \begin{align*}
 (U_j\circ T_{1h})_{PQ}&=\la_h^{p_h-q_h}U_{j,(PQ)_h}=\la_h^{p_h-q_h}\nu_{(PQ)_h}^+U_{j,(A_j,B_j)(P,Q)}\\
 &=\la^{p_h-q_h}\nu_{(PQ)_h}^+ \ov \nu_{PQ}^+U_{j,PQ}= U_{j,PQ}.
 \end{align*}
 Combining with the identities which we have proved, we get    $(U_j\circ T_{1j})_{QP}=(U_j\circ T_1)_{QP}=(\la_jV_j)_{QP}$ for $j\neq h$.
 This gives us all the identities in \re{uit1j} for $(P,Q)\in\cL R_j$. These identities are trivial when $(P,Q)$ is not in $\cL R_j$. Therefore, we have shown that these conditions are sufficient for $\var\in\cL C(\cL T_1,\cL T_2)$.

 Finally, we need to verify that \re{muimuj}, \re{vjlju}, and \re{ujrjnj}-\re{sqnupq2} imply that $\var$ and $\rho$ commute. In other words, we need to verify \re{uhus} and \re{uenu}, by (iii).
 We have
\begin{equation}\label{non-commut}
 (A_eB_e)\circ\rho_e\circ (A_eB_e) = (A_eB_e)\circ\rho_e \quad  \text{on $\cL R_e$}.
 \end{equation}
Let $(P,Q)\in \cL R_e$.  By \re{ujrjnj}, \re{sqnupq}, \re{non-commut} and \re{ujrjnj}, we get
$$U_{e,PQ}=\nu_{PQ}^+U_{e,(A_e,B_e)(P,Q)}=  \nu_{PQ}^+\ov U_{e,(A_e,B_e)\rho_e(P,Q)}=
\nu_{PQ}^+\nu^+_{\rho_e(P,Q)}\ov U_{e,\rho_e(P,Q)}.$$  By \re{nupm1++}, \re{nupm1}, $\nu_{PQ}^+\nu^+_{\rho_e(PQ)}=\nu_{PQ}$. We obtain \re{uenu}.
%
%
%
Let us prove \re{uhus} with $PQ\in \cL R_{s+s_*}$. Using   \re{ujrjnj}, \re{sqnupq2} with $PQ=(A_{s+s_*}B_{s+s_*})(PQ)$,\re{commut-arho+},
\re{ujrjnj} with $PQ=\rho(PQ)$ successively, we get
\aln
U_{s+s_*,PQ}&=\nu^+_{PQ}U_{s+s_*,A_{s+s_*}B_{s+s_*}(PQ)}=\nu^+_{PQ}\ov U_{s,\rho(  A_{s+s_*}B_{s+s_*}(PQ))}
\\
&= \nu^+_{PQ}\ov U_{s,A_{  s}B_{  s}(\rho(PQ))}=\nu^+_{PQ}\nu^+_{\rho(PQ)}\ov U_{s,\rho(PQ)}.
\end{align*}
which gives us \re{uhus} by \re{nupm1++}.
To prove \re{uhus} for hyperbolic index, apply successively from left to right, \re{ujrjnj},  \re{sqnupq1} with $PQ=(A_hB_h)(PQ)$, \re{commut-arho} and \re{ujrjnj} with $PQ=\rho_h(PQ)$~:
\begin{align*}
  \nu_{PQ}^+U_{h,PQ}
&=U_{h,(A_hB_h)(PQ)}=\ov  U_{h,\rho_h(A_hB_h)(PQ)}\\
&=\ov U_{h,(A_hB_h)\rho_h(PQ)}=\nu_{\rho_h(PQ)}^+\bar U_{h,\rho_h(PQ)}.
\end{align*}
By \re{nupm1++} again, we obtain \re{uhus}.
 The proof is complete.
\end{proof}

 \subsection{Normalized mappings}
We have described the conditions on centralizers. We now determine  complements of these conditions
to define normalized mappings.
\begin{defn}\label{dfnorm} Let $\var=I+(U,V)$ be a formal mapping tangent to the identity.
\bppp
\item We say that $\var$ is {\it normalized} with respect to $S_1,\ldots, S_p$ if
\eq{ujpq=0}\nonumber
U_{j,PQ}=0=V_{j,QP}, \quad \text{if $(P,Q)\in\cL R_j, \quad \forall j$}.
\eeq
Furthermore, $\rho\var\rho$ is  normalized w.r.t. $S_1,\ldots, S_p$ if and only if $\var$ is.
\item  We say that $\var$ is {\it normalized} w.r.t.
$\{\cL S, T_1\}$, if
\eq{}\label{vjlju++}
V_{j,QP}=- \nu_{PQ} U_{j,PQ},
\quad (P,Q)\in\cL R_j.
\eeq
\item
 We say that $\var$ is {\it normalized} w.r.t.
$\{\cL S, T_1,\rho\}$ if
\begin{alignat}{4}
\label{uhus+1} U_{h,PQ}&=- \ov{U}_{h,\rho(PQ)}, \  &&\forall (P,Q)\in\cL R_h;
\\
\label{ussp}
 U_{s+s_*,PQ}&=-\ov{U}_{s,\rho(PQ)},\  &&\forall(P,Q)\in\cL R_{s+s_*};\\
\label{uhus+3} U_{e,PQ}&=-\nu_{PQ}\ov{U}_{e,\rho_e(PQ)},\quad  &&\forall(P,Q)\in\cL R_e.
\end{alignat}
\item Let $p>1$. We say that $\var$ is {\it normalized} w.r.t. $\{\cL T_1,\cL T_2\}$ if
\eq{ujrjnjC}
  U_{j, P Q}=0,\quad (P,Q)\in \cL N_j.
\eeq
 We say that $\var$ is {\it normalized} w.r.t. $\{\cL T_1,\cL T_2,\rho\}$ if
\begin{alignat}{4}
\label{UePQ-}
U_{e, P Q} &= -  
 \ov U_{e,   (A_e,B_e)\circ
 \rho_e( P, Q)},\quad &&\forall (P,Q)\in \cL N_e;\\
\label{UhPQ-}
U_{h, P Q}&=-\ov U_{h,\rho( P, Q)}, &&\forall (P,Q)\in\cL N_h;\\
\label{UssP} U_{s+s_*,PQ}&=-  
\ov U_{s, 
\rho( P, Q)}, \quad &&\forall (P,Q)\in  \cL N_{s+s_*}.
\end{alignat}
\eppp
  The set of normalized mapping w.r.t. to a family $\cL F$ is denoted ${\cL C}^{\mathsf{c}}(\cL F)$.
\end{defn}

\begin{lemma}\label{FHG-}
 Let $F$ be a formal map which is tangent to the identity.
There exists a unique formal decomposition $F=HG^{-1}$ with $G\in{\cL C}(\cL S,T_1,\rho)$
$($resp. ${\cL C}(\cL T_1,\cL T_2, \rho))$
and $H\in {\cL C}^\mathsf{c}(\cL S,T_1,\rho)$ $($resp. ${\cL C}^\mathsf{c}(\cL T_1,\cL T_2, \rho)))$.
If $F$ is convergent, then $G$ and $H$ are also convergent.
\end{lemma}
\begin{proof} We will apply \rl{fhg-} as follows. Let
$\hat H$   be the set of mappings
in $ {\cL C}_2^{\mathsf{c}}(\cL S, \cL T_{1}, \rho)$.
Note that $\hat H$ is a $\rr$-linear subspace of   $({\widehat { \mathfrak M}}_n^2)^n$.
We will define a $\rr$-linear projection $\pi$ from  $({\widehat { \mathfrak M}}_n^2)^n$
onto $\hat H$ such that $\pi$ preserves the degree of $F$ if $F$ is homogeneous. We will
 show that $\hat G=(\operatorname{I}-\pi)\hat H$  agrees with  ${\cL C}_2(\cL S, \cL T_{1}, \rho)$.
We will derive estimates on $\pi$ stated in \rl{fhg-}, from which we conclude the convergence of $H, G$.
The same argument will be applied  to  the second case of $\cL C(\cL T_1, 
\rho)$ and $\cL C^{\mathsf c}(\cL T_1,
\rho)$.

For the first case, let us define a projection $\pi \colon ({\widehat { \mathfrak M}}_n^2)^n 
\to \hat H$.
We decompose
$$
(U,V)=(U'+U'',V'+V''),\quad
\pi (U,V)=(U',V').
$$
 We first define
\eq{notinrj}
U'_{j,PQ}=U_{j,PQ}, \quad V_{j,PQ}'=V_{j,PQ}, \quad U_{j,PQ}''=0,
\quad V_{j,PQ}''=0, \eeq
for  $ (P,Q)\not\in\cL R_j.$
Suppose that $(P,Q)\in \cL R_e$. We have
$$
U_{e,PQ}=U_{e,PQ}'+U_{e,PQ}'', \quad U_{e,\rho_e(PQ)}=U_{e,\rho_e(PQ)}'+U_{e,\rho_e(PQ)}''.
$$
According to \re{uhus+3} and  \re{uenu}, we  need to seek solutions that satisfy
\eq{uepq'}
U_{e,PQ}'+\nu_{PQ}\ov U_{e,\rho_e(PQ)}'=0, \quad
U_{e,PQ}''-\nu_{PQ}\ov U_{e,\rho_e(PQ)}''=0.
\eeq
Hence,  for $(P,Q)\in \cL R_e$ we  choose
\begin{gather}\label{uepq'+}
\nonumber
 U_{e,PQ}'=\f{1}{2}(U_{e,PQ}-\nu_{PQ}\ov U_{e,\rho_e(PQ)}), \   U_{e,PQ}''=\f{1}{2}(U_{e,PQ}+\nu_{PQ}\ov U_{e,\rho_e(PQ)}).
\end{gather}
 We verify directly that the solutions satisfy \re{uepq'} as follows: 
\aln
U_{e,PQ}'+\nu_{PQ}\ov U_{e,\rho_e(PQ)}'&=
\f{1}{2}(U_{e,PQ}-\nu_{PQ}\ov U_{e,\rho_e(PQ)})\\ &\quad +
\f{1}{2}(\nu_{PQ}\ov U_{e,\rho_e(PQ)}-\nu_{PQ}\nu_{\rho_e(PQ)}\ov U_{e,PQ})=0.
\end{align*}
 Here we have used  that $\rho_e$ is an involution on  $\cL R_e$
and $\nu_{\rho_e(PQ)}\nu_{PQ}=1$ from \re{nupm1+}.

For $(P,Q)\in \cL R_h$,   we achieve   \re{uhus+1} and    the first identity in \re{uhus} by taking
\ga\label{uhpq'}\nonumber
 U_{h,PQ}'=\f{1}{2}(U_{h,PQ}- \ov U_{h,\rho(PQ)}),
\quad  U_{h,PQ}''=\f{1}{2}(U_{h,PQ}+ \ov U_{h,\rho(PQ)}).
\end{gather}
For $(P,Q)\in\cL R_{s+s_*}$,  we achieve  the second identity in \re{uhus} and \re{ussp} by taking
\ga\label{usspq'}\nonumber
 U_{s+s_*,PQ}'=\f{1}{2}(U_{s+s_*,PQ}-\ov U_{s,\rho(PQ)}),
\quad  U_{s+s_*,PQ}''=\f{1}{2}(U_{s+s_*,PQ}+ \ov U_{s,\rho(PQ)}).
\end{gather}

We have determined coefficients for $U_{j,PQ}', U_{j,PQ}''$ with $(P,Q)\in\cL R_j$.
Let us set for $(P,Q)\in\cL R_j$,
\begin{eqnarray}
V_{j,QP}'' &=& \la_j^{-1}\la^{P-Q}U_{j,PQ}''\label{v''},\\
V_{j,QP}'&=&V_{j,QP}-V_{j,QP}''.
\end{eqnarray}
This fulfills the conditions on $V_j'$ and $V_j''$ easily. Note that the first identity means that $(U'',V'')$
commutes with $T_1$. We have obtained the required formal decomposition.

To prove the convergence,  we start with
\begin{equation}\label{nupq}
\la_j^{-1}\la^{P-Q} = \nu_{PQ}=\pm1
\end{equation}
 for $(P,Q)\in\cL R_j$.
So $\pi$ is indeed an $\rr$-linear projection which preserves  degrees. Since
$|\nu_{PQ}|=1$, 
we have that
$$
|U_{PQ}'|\leq \max_{(P',Q')}|U_{P'Q'}|.
$$
Here $(P',Q')$ runs over all permutations of $(P,Q)$ in $2p$ coordinates.
The same holds for $V'$. Hence, with the notation of \rl{fhg-}, we have
$$
\{\pi (U,V)\}_{sym}\prec (U, V)_{sym}.
$$
The existence and uniqueness as well as the convergence also follow
 from \rl{fhg-}.

We now consider the second case of $\cL C(\cL T_1,\cL T_2,\rho)$ by  minor changes.
Let us define a projection $\pi\colon(({\widehat { \mathfrak M}}_n^2)^n\to \hat H$. Here $\hat H$
is the space associated with the mappings satisfying the normalized conditions \re{ujrjnjC}-\re{UssP}.
Let  $\hat G=(\operatorname{I}-\pi)\hat H.$ 
We decompose as above
$$
(U,V)=(U'+U'',V'+V''),\quad
\pi (U,V)=(U',V').
$$
Recalling that $\iota_e=(A_e,B_e)\circ\rho_e$ is  an involution on $\cL N_e$, we  choose~:
\begin{alignat}{4}
U_{j,PQ}''&=\f{1}{2}(U_{j,PQ}+  
\ov U_{j, 
\rho_j(PQ)}), \quad && (P,Q)\in\cL N_h,\\
\label{consist1} U_{j,PQ}'&=\f{1}{2}(U_{j,PQ}-  
\ov U_{j, 
\rho_j(PQ)}),\quad  && (P,Q)\in\cL N_h,\\ 
U_{e,PQ}''&=\f{1}{2}(U_{e,PQ}+\ov U_{e,\iota_e(PQ)}), &&\quad (P,Q)\in\cL N_e,\\
\label{consist2}
U_{e,PQ}'=&\f{1}{2}(U_{e,PQ}-U_{e,\iota_e(PQ)}),&&\quad (P,Q)\in\cL N_e,\\
U_{s+s_*,PQ}'' &= \f{1}{2}(U_{s+s_*,PQ}+
\ov U_{s, 
\rho(PQ)}),\quad && (P,Q)\in\cL N_{s+s_*},\\
\label{consist3}
U_{s+s_*,PQ}'&=\f{1}{2}(U_{s+s_*,PQ}- 
\ov U_{s, 
\rho(PQ)}),\quad && (P,Q)\in\cL N_{s+s_*}.
\end{alignat}
  We still use \re{notinrj} for $(P,Q)\not\in\cL R_j$.   For $(P,Q)\in R_j$,
define $V_{j,QP}''$  by \re{v''} and   $V_{j,QP}'=V_{j,QP}-V_{j,QP}''$,
 after  we set
\eq{ujrjnjCs}
U_{j, P Q}'' =\nu_{P Q}^+U''_{j,(A_j,B_j)( P, Q)},
\quad
U_{j, P Q}' =U_{j,PQ}-U_{j,PQ}'', \
(P,Q)\in\cL R_j\setminus\cL N_j.
\eeq

Let us   verify that $\pi (U,V)=(U',V')$ is in $\hat H$. 
To verify  \re{UePQ-} for  $j=e$,   via  \re{consist1}
we compute
\aln
U'_{e,PQ}+  
\ov U'_{e,(A_e,B_e)\circ\rho_e(PQ)}=
\f{1}{2}(U_{e,PQ}-   
\ov U_{e,\iota_e(PQ)})
  +
\f{1}{2}(U_{e,\iota_e(PQ)}-  
\ov U_{e,PQ})=0.
\end{align*}
We also know that  $\rho$ is an involution on $\cL N_h$ and it is a bijection from $\cL N_{s+s_*}$ onto $\cL N_{s}$.
Analogously, we verify \re{UhPQ-}    and  \re{UssP} via \re{consist1} and \re{consist3}.
Note that $(P,Q)\to (A_j,B_j)(P,Q)$ is a projection on   $\cL N_j$. Analogously, we verify \re{ujrjnjC} via \re{ujrjnjCs}.
This shows that $\pi (U,V)$ is in $\hat H$. We can also verify that $(U'',V'')=(\operatorname{I}-\pi
)(U,V)$ satisfies the conditions
on the centralizer,  i.e. it  is in $\hat G$.

As before, we have
$$
|U_{j,PQ}'|, |U_{j,PQ}''|\leq  C \max_i\max_{(P',Q')\text{permutation of } (P,Q)}|U_{i,P'Q'}|.
$$
Equations   \re{v''} lead to the same inequality for $V''$ and hence for $V'=V-V''$. Hence, again the result follows from \rl{fhg-}.
\end{proof}

\subsection{Convergence of linearizations.}
\begin{prop}\label{linearS} Assume that the family of involutions $\{\cL T_1, \cL T_2, \rho\}$ is
formally linearizable. Assume further that $\sigma_1,\ldots, \sigma_p$
 defined by \rea{sigma_i}-\rea{sigma_i+}, are linear.
 \bppp
\item There is
 a  biholomorphic mapping in the centralizer of $\{\cL S,\rho\}$
which linearizes $\tau_1$ and $\tau_2$.
\item
Assume further that $\tau_1=T_1$ and $\tau_2=T_2$.
Then $\{\tau_{11},\ldots, \tau_{1p},\rho\}$ is holomorphically linearizable.
\eppp
\end{prop}
\begin{proof} (i) Suppose that $\Psi$ is a formal mapping satisfying
\eq{conj-tj}\nonumber
\Psi^{-1}\tau_{1j}\Psi=T_{1i_j}, \quad \Psi\rho=\rho\Psi.
\eeq
Then $T_{1j}=(L\Psi)\circ T_{1i_j}\circ (L\Psi)^{-1}$, and $L\Psi$ commutes with $\rho$.
 Replacing $\Psi$ by $\Psi\circ L\Psi^{-1}$, we may assume that $\Psi$ is tangent to the identity and $i_j=j$.
We decompose
 $
\Psi=\Psi_1\Psi_0^{-1},
$
where $\Psi_1$ is normalized w.r.t. $\cL S,T_1,\rho$ and $\Psi_0$ is in the centralizer of $\cL S,T_1,\rho$.
Since $\Psi,\Psi_0$ commute with the $S_j$'s and $\rho$, then $\Psi_1$ commutes with the $S_j$'s and $\rho$ too.
We now let $\Psi$ denote $\Psi_1$.

To be more specific, let us write
$$
\tau_{1}\colon \begin{cases}
    \xi'_i= \lambda_{i}\eta_{i}+ f_{i}(\xi,\eta)
    \quad i=1,\ldots, p,\\ \eta'_i=\lambda_{i}^{-1}\xi_{i}+ g_{i}
    (\xi,\eta)\quad i=1,\ldots, p,\\
\end{cases}
$$
and
$$
\Psi\colon\begin{cases}
\xi'_i= \xi_{i}+U_{i}(\xi,\eta)\quad i=1,\ldots,p,\\
\eta'_i=\eta_{i}+V_{i}(\xi,\eta)\quad i=1,\ldots, p.\\
\end{cases}
$$
Let us write that $\Psi$ conjugates $\tau_{1}$ to
$$
T_1\colon
    \xi'_i= \lambda_{i}\eta_{i},\quad
      \eta'_i=\lambda_{i}^{-1}\xi_{i}, \quad i=1,\ldots,p.
$$
 We have $\Psi\circ T_{1}= \tau_{1} \circ
\Psi$; that is
\ga\label{lai-1-}
\lambda_{i}V_i-U_i\circ T_{1}=   - f_{i}\circ \Psi(\xi,\eta)\quad i=1,\ldots, p,\\
\label{lai-1}\lambda_{i}^{-1}U_i-V_i\circ T_{1}= -g_{i}\circ \Psi(\xi,\eta)\quad i=1,\ldots, p.
\end{gather}
Since $\Psi$ is normalized with respect $\{\cL S, T_1,\rho\}$, it satisfies \rd{dfnorm}   (iii).
Since  $\Psi$ commutes with each $S_j$, then $U_{j,PQ}=V_{j,QP}=0$ for $(P,Q)\not\in\cL R_j$.  Since it also commutes with $\rho$, then by \re{uhus} and \re{uhus+1}-\re{ussp} we obtain $U_{j,PQ}=0$ for $(P,Q)\in\cL R_j$
and $j=h,s,s+s_*$.

We need to majorize $U_{e,PQ},V_{e,QP}$ for $(P,Q)\in\cL R_e$.
 By   \re{lai-1} and \re{nupq}, we obtain
$$
U_{e,PQ}-\nu_{PQ}^{-1}V_{e,QP}=-\la_e\{g_e\circ\Psi\}_{PQ}.
$$
Using  \re{veqpo} and \re{uhus+3}, we obtain  $V_{e,QP}=\ov U_{e,\rho_e(PQ)}=-\nu_{PQ}^{-1}U_{e,PQ}$, and
hence
$$
U_{e,PQ}=-\f{1}{2}\la_e\{g_e\circ\Psi\}_{PQ},\quad
V_{e,QP}=\f{1}{2}\nu_{PQ}\la_e\{g_e\circ\Psi\}_{PQ}.
$$
Therefore, we have
$$
|V_{e,QP}|,|U_{e,PQ}|\leq C \left|\{g_j\circ\Psi\}_{PQ}\right|.
$$
The above holds for $(P,Q)\in\cL R_e$. It holds trivially for $(P,Q)\not\in\cL R_e$.
 In view of \re{FiGp},  
we then have
$$
\psi_{sym}\prec C g_{sym}\circ\Psi_{sym}= g_{sym}\circ(I_{sym}+\psi_{sym}).
$$
Therefore, $\psi_{sym}$ is convergent at the origin and so is $\Psi$.

(ii)
Assume now that $\sigma=S, \tau_1=T_1, \tau_2=T_2$ are linear. Suppose that $\Psi$ linearizes the $\{\tau_{ij}\}$ and
commutes with $\rho$. We decompose $\Psi=\Psi_1\Psi_0^{-1}$ with $\Psi_1$ being normalized w.r.t. $\cL T_1,\cL T_2,\rho$ and
with $\Psi_0$ being in the centralizer of $\cL T_1,\cL T_2,\rho$.
 Since $\Psi^{-1}\tau_{ij}\Psi=T_{ij}$, we have
$
\Psi_1^{-1}\tau_{i j}\Psi_1=\Psi_0^{-1}T_{i j}\Psi_0=T_{i j}.
$
Hence, $\Psi_1$ linearizes the $\tau_{i j}$ and is normalized w.r.t $\cL T_1,\cL T_2,\rho$. Since $\Psi,\Psi_{ 0}$ commute with $\cL S$, $T_1$ and $\rho$, so does $\Psi_1$.

We recall
$$
T_{1 j}\colon \begin{cases}
    \xi'_j= \lambda_{j}\eta_{j} \\
    \eta'_j=\lambda_{j}^{-1}\xi_{j} \\
    \xi'_k=\xi_{k},\quad k\neq j\\
    \eta'_k=\eta_{k},\quad k\neq j,\\
\end{cases}
\quad
\tau_{1 j}\colon \begin{cases}
    \xi'_j= \lambda_{j}\eta_{j}+ f_{j j}(\xi,\eta)\\
    \eta'_j=\lambda_{j}^{-1}\xi_{j}+ g_{j j}(\xi,\eta)\\
    \xi'_k=\xi_{k}+ f_{jk}(\xi,\eta),\quad k\neq j\\
    \eta'_k=\eta_{k}+ g_{jk}(\xi,\eta),\quad k\neq j.\\
\end{cases}
$$
Since we have $\Psi\circ T_{1j}= \tau_{1j} \circ \Psi$, we obtain the following relations
\begin{equation}\label{lin-invol}
\begin{cases}
\lambda_{j}V_j-U_j\circ T_{1 j}= - f_{j j}\circ \Psi\\
\lambda_{j}^{-1}U_j-V_j\circ T_{1 j}=-g_{j j}\circ \Psi\\
U_k-U_k\circ T_{1 j}=- f_{jk}\circ \Psi,\quad k\neq j\\
V_k-V_k\circ T_{1 j}=- g_{jk}\circ \Psi,\quad k\neq j.
\end{cases}
\end{equation}
Since $\Psi\in\cL C(\cL S,T_1,\rho)$, combining \re{uhus},   \re{vjlju} with
  the normalizing conditions \re{UhPQ-}, \re{UssP}, we find   that $U_{j,PQ}=0=V_{j,QP}$ for
  $(P,Q)\in \cL N_j$ and $j=h,s,s+s_*$.
  Using $\Psi\rho=\rho\Psi$, we get $V_{e}=\ov{ U_{e}\circ\rho}$.  By the first equation above, we get
  $$
  \lambda_e\ov{ U_e\circ\rho}\circ T_{1e}-U_e=-f_{ee}\circ \tau_{1e}\circ\Psi.
  $$
  For $(P,Q)\in\cL N_e$, we have  $(\lambda_e U_e\circ\rho\circ T_{1e})_{PQ}=\ov U_{e,A_eB_e(\rho_e(PQ))}$. By
  \re{UePQ-}, we get
  \eq{ueqpne}
  U_{e,PQ}=\f{1}{2}\left\{f_{e,e}\circ \tau_{1e}\circ\Psi\right\}_{PQ}, \quad   V_{e,QP}=\nu_{PQ}U_{e,PQ},
  \quad \ (PQ)\in\cL N_e.
  \eeq

We now majorize $U_{j,PQ}, V_{j,QP}$ for $(P,Q)\in\cL R_j\setminus\cL N_j$. Fix $(P,Q)\in\cL R_k\setminus\cL N_k$. Start with some $j$ such that $p_j<q_j$.
In the second last  identity  in \re{lin-invol}, let us compose on the right by  $T_{1 j'}$ with $j'\neq k,j$  to get
$$
  U_k\circ T_{1 j'}-U_k\circ T_{1 j}\circ T_{1 j'}=- f_{j,k}\circ \Psi\circ T_{1 j'} =f_{j,k}\circ \tau_{1 j'}\circ \Psi.
$$
Let   $\{\ell_1,\dots, \ell_d\}$ 
be the set of $i  \neq k$ such that
  $p_i<q_i$.
  Composing successively with the $T_{1l_j}$'s and adding, we get
$$
U_k - U_k\circ T_{1\ell_1}\circ\cdots T_{1\ell_d}=-\sum_{i=1}^{ d} f_{\ell_i, k}\circ\tau_{1\ell_{1}}\circ\cdots\circ \tau_{1\ell_{i-1}}\circ\Psi.
$$
Hence, if $PQ\in \cL R_k\setminus \cL N_k$, then
$$
U_{k,PQ}=\{U_k\circ T_{1\ell_1}\circ \dots\circ T_{1\ell_d}\}_{PQ}-\left\{\sum_{i=1}^d f_{\ell_i,k}\circ\tau_{1\ell_{ 1}}\circ\cdots\circ \tau_{1\ell_{i-1}}\circ\Psi\right\}_{PQ}.
$$
The first term on the right-hand side,  $U_{k,(A_k,B_k)(PQ)}$,  is either zero or majorized by \re{ueqpne}.
The summations have finitely many combinations. This shows that $U_{k}\prec a_k\circ\ov\Psi$.
By \re{vjlju}, we obtain $V_{j}\prec (U_j)_{sym}$. This shows that
$(U,V)\prec b\circ(I_{sym}+(U,V)_{sym})$ for some analytic mapping $b=O(2)$.  Using  \rl{fhg-}, we obtain
the convergence of $U_k,V_k$.
\end{proof}


\begin{thm}\label{rigidQ}
 Let $M$ be a germ of analytic submanifold that is an higher order perturbation of a product quadric $Q$ in $\cc^{2p}$. Assume that $M$ satisfies
 condition~J and it is 
  formally equivalent to  $Q$.
 Suppose that each hyperbolic component has an eigenvalue $\mu_h$ which
is either a root of unity or satisfies the Brjuno condition \rea{bnI} in which $\cL I=0$, $\ell=1,n=1, \mu_i=\mu_{i,j}=\mu_h$, and    each $\mu_s$ is not a root of unity and satisfies the Brjuno condition. Then $M$ is holomorphically equivalent
to the product quadric.
\end{thm}
\begin{proof}
We first apply a   theorem (with $\mathcal I=0$) in~\cite{stolo-bsmf} that linearize simultaneously and holomorphically the $\sigma_1,\ldots, \sigma_p$.
 Note that the small divisor condition in this special case
is equivalent that each $\mu_h$ is either a root of unity or a Brjuno number.
 Then, we apply successively the two assertions of \rp{linearS}. Hence, in good holomorphic coordinates, $\{\tau_{11},\ldots,\tau_{1p},\rho\}$ are linear. Then, by \rp{inmae}, the manifold is holomorphically equivalent to the quadric.
\end{proof}


As in the case of \rt{abelinv}, we can  also prove the first part of the above   proof by
 applying R\"ussmann's theorem \cite{Ru02} successively to each $\sigma_i$. This is due to the commutativity property and the special type of the linear parts that lead to the relatively simpler relations on $\cL C(S_i)$ and $\cL C^\mathsf{c}(S_i)$ for each fixed $i$.
\setcounter{thm}{0}\setcounter{equation}{0}
\section{
Existence of attached  complex manifolds}
\label{secideal}

We are interested in
complex submanifolds $K$  in $\cc^{2p}$ that    intersect the real
submanifold $M$ at the origin. Recall that $M$ has real dimension $2p$. Generically, the origin is
an isolated intersection point if $\dim K=p$.  Let us consider the situation when
the intersection has dimension $p$. Without further restrictions, there are many such
complex submanifolds; for instance, we can take a $p$-dimensional
 totally real and real analytic
submanifold $K_1$ of $M$. We then let $K$ be the complexification of $K_1$.
 To ensure the uniqueness or finiteness of  the complex submanifolds $K$,
we therefore introduce
the following.
\begin{defn} Let $M$ be a formal real submanifold of dimension $2p$ in $\cc^{2p}$.
We say that a formal complex submanifold $K$ is  {\it attached} to $M$ if $K\cap M$
 contains at least two germs of  totally real and formal submanifolds $K_1, K_2$  of dimension $p$
that intersect transversally at the origin. Such a pair $\{K_1,K_2\}$ are called a pair of {\it asymptotic} formal submanifolds of $M$.  \end{defn}

Before we present the details, let us describe the main steps to derive the results. 
We first  derive the results at the formal level.
We then apply the results of \cite{Po86} and \cite{stolo-bsmf}. The proof of the co-existence of convergent and divergent attached
submanifolds will rely on a theorem of
P\"oschel  on stable invariant submanifolds and Siegel's small divisor technique.

We now describe the formal results.
When $p=1$,  a non-resonant hyperbolic $M$
 admits a unique attached formal holomorphic curve~\cite{Kl85}.  When $p>1$,  new situations arise.
First, we show that there are obstructions to
attach formal submanifolds.
However, the  formal obstructions disappear when  
 $M$ admits the maximum number
of deck transformations and   $M$ is non-resonant. 
 These two conditions allow us to express $M$
in an 
equivalent form \re{masym}.  This equivalent form for $M$, which has not been used so far, will play an essential
role in our proof 
 for $p>1$.

 We will   consider  a real submanifold
$M$ which is   a higher order perturbation of a non-resonant  product quadrics.
By adapting the proof of Klingenberg~\cite{Kl85} to  the manifold $M$  \re{masym},
we will show the existence of a unique attached  formal submanifold for a prescribed non-resonance condition. As in~\cite{Kl85},
we  also show that  the complexification of $K$ in $\cL M$ is a pair of invariant formal submanifolds $\cL K_1,\cL K_2$ of $\sigma$.
Furthermore,  $K$ is convergent if and only if $\cL K_1$ is convergent.



Let us first recall the values of the Bishop invariants. The types of the invariants play an important role
 for the existence
and the convergence of  attached formal complex submanifolds. 
 From  
  \re{gaala1} and \re{msms-+},
we recall that
\ga\label{gens11}
 \gamma_e^{-1}= \la_e+\la_e^{-1}, \quad
\gamma_h^{-1}= \la_h+\ov\la_h,\quad
\gamma_s^{-1}= 1+\ov\la_s^{ 2},\\
\label{gehs11}
 0<\gaa_e<1/2, \quad
 \gaa_h>1/2, \quad
\gaa_s\in (-\infty, 1/2)+i(0,\infty), \quad
\gaa_{s+s_*}=1-\ov\gaa_s.
\end{gather}
Here we exclude the case that $\RE\gamma_s=1/2$ or $\gamma_s<1/2$ as we will assume that $\sigma$ has distinct eigenvalues.  We normalize
\ga\label{Larange}
\la_e>1,\quad |\la_h|=1,\quad |\la_s|>1,\quad\la_{s+s_*}=\ov\la_{s}^{-1};\\
\arg\la_{h}\in(0,\pi/2), \quad\arg\la_s\in(0,\pi/2).
\label{Larange+}
\end{gather}
Recall that $\mu_j=\la_j^2$. By \re{gens11},  we have
\eq{gsgssmu}
\gaa_j^2=\frac{\mu_j}{(1+\mu_j)^2}, \quad j=e,h; \qquad
\gaa_s\ov\gaa_{s+s_*}= \frac{\ov\mu_s}{(1+\ov\mu_s)^2}.
\eeq
We first verify the following.
\le{disga2} Let $\gaa_j,\la_j$ be given by \rea{gens11}-\rea{Larange+}. Let $\mu_j=\la_j^2$.
Assume that $\mu_1$,   $\mu_1^{-1}, \ldots$,  $\mu_p$,  $ \mu_p^{-1}$ are distinct. Then
  $\gaa_{e}^2,  \gaa_{h}^2,   \ov\gaa_s\gaa_{s+s_*},  \gaa_{s}\ov\gaa_{s+s_*}$
are distinct $p$ numbers.  The latter is equivalent to $\gaa_1, \ldots, \gaa_p$ being distinct.
\ele
\begin{proof} Note that $x^{-1}+x$ and $x^{-1}$
   decrease strictly  on $(0, 1)$.  So $\gaa_e^2,\gaa_h^2$ are distinct.
 We also have
 $
 \gaa_s\ov\gaa_{s+s_*}=\gaa_s-\gaa_s^2.
 $
If $a,b$ are complex numbers, then
 $a-a^2=b-b^2$ if and only if $a=b$ or $a+b=1$. Since $\gaa_s$ is not real,
 then $\gaa_s\ov\gaa_{s+s_*}$ are different from $\gaa_e^2$ and $\gaa_h^2$.
 For any distinct complex numbers $a_{1},   a_2$ in $(-\infty,1/2)+i(0,\infty)$.
 We have $1-a_2\neq 1-a_1, a_1, a_2$. The lemma
 is proved.
 \end{proof}

Let us 
first investigate the numbers  of pairs of formal asymptotic submanifolds and attached formal submanifolds.
\begin{lemma}\label{asynum}
Let $M$ be a formal submanifold that is  a third order perturbation of  a product quadric $Q$ in $\cc^{2p}$.
Assume that  $M$   
has distinct eigenvalues $$\mu_1,\dots, \mu_p, \quad\mu_1^{-1}, \ldots, \mu_p^{-1}.$$
  \bppp
\item If $M$ admits an attached formal submanifold,  its CR singularity has no elliptic component.
\item If $Q$ has no elliptic components, then
$Q$ has  at least  $2^{h_*+s_*-1}$ pairs of asymptotic totally real and real analytic submanifolds
that 
are contained in a single  attached complex submanifold.
\item  There is no formal  submanifold   attached to
$$
M\colon z_3=(z_1+2\gaa_1\ov z_1)^2+(z_2+2\gaa_2\ov z_2)^3,
\quad z_4=(z_2+2\gaa_2\ov z_2)^2.
$$
Here  $M$ has a hyperbolic complex tangent at the origin.
\item Assume that $M$ has no elliptic component and it admits the maximum number of formal deck transformations.
 Given $\e_h,\e_s=\pm1$, let $\nu=\nu_\epsilon:=( \nu_1,\ldots,\nu_p)$ with
\eq{nunue-}
\nu_h:=\mu_h^{\e_h},\quad  \nu_s:= \ov\mu_s^{\e_{s}},\quad  \nu_{s_*+s}:= \mu_s^{-\e_{s}}
\eeq
Suppose that
\eq{nunue+}
\nu^Q\neq\nu_j^{-1}, \quad \forall Q\in\nn^p, \quad |Q|>0, \quad 1\leq j\leq p.
\eeq
Then $M$ admits  a unique pair of asymptotic formal submanifolds $K_1,K_2$ such that  each  $K_i$
is defined by $z'=\rho_i(z')$ for   a formal anti-holomorphic involution $\rho_i$ and 
the linear part of  $\ov\rho_2^{-1}\ov\rho_1$  has eigenvalues   $\nu_1,\ldots, \nu_p$.
In particular, if \rea{nunue+} holds for each $\nu$ of the form \rea{nunue-} then
  $M$ admits exactly $2^{h_*+s_*-1}$ pairs
of asymptotic formal submanifolds.  \eppp
 \end{lemma}
\begin{proof} (i) Let $M$ be defined by
\eq{}
z_{p+j}=Q_j(z',\ov z')+H_j(z',\ov z'), \quad 1\leq j\leq p
\eeq
where $H_j(z',\ov{z'})=O(|z'|^3)$ and each $Q_j$ is quadratic.
Let $\{K_1,K_2\}$ be a pair of asymptotic formal submanifolds of $M$ 
intersecting  a formal complex
submanifold $K$.
We know that the totally real spaces $T_0K_1, T_0K_2$ are contained in $T_0M$, the $z'$-subspace. Let $K_i'$ be the projection of $K_i$
onto the $z'$-subspace, then $K_1',K_2' $ are still totally real.
Let $K_1'$ be defined by
\gan
K_1'\colon \ov z'={\mathbf A}z'+R(z'), \quad \ov {\mathbf A}{\mathbf A}={\mathbf I},  \quad R(z')=O(2)
\end{gather*}
such that $\rho_1(z'):=\ov{\mathbf A}\ov z'+\ov R(\ov z')$ defines an anti-holomorphic  formal involution.  Let $K_2$ be the (formal)
 fixed-point
set of  the anti-holomorphic
involution
$\rho_2(z')=\ov{\tilde{\mathbf A}}\ov z'+\ov{\tilde R(z')}$ with $\tilde R(z')=O(2)$.  Then $K_1,K_2$ intersect transversally at the origin
if and only if
$
\det(\tilde{\mathbf A}-\mathbf A)\neq0.
$
Let us define  holomorphic mappings
\begin{equation}\label{def-rhobar}
\ov\rho_i(z'):=\ov{\rho_i(z')}, \quad i=1,2. 
\end{equation}
Then $K$ is given by
\eq{dfKzp}
z_{p+j}''=Q_j(z',\ov\rho_{i}(z'))+H_j(z',\ov\rho_i(z')), \quad i=1,2, \quad j=1, \ldots, p.
\eeq
The  two equations agree, if and only if
\eq{Qzpb}
Q_j(z', \ov\rho_{1}(z'))+H_j(z',\ov\rho_1(z'))=Q_j(z',\ov\rho_{2}(z'))+H_j(z',\ov\rho_2(z')), \quad 1\leq j\leq p.
\eeq
 Then the asymptotic totally real submanifolds $\{K_1,K_2\}$ are defined by
\eq{dfki}
K_i\colon z_{p+j}=Q_j(z',\ov z')+H_j(z',\ov z'), \quad 1\leq j\leq p, \quad \rho_i(z')=z'.
\eeq

Recall that
\aln
Q_j(z',\ov z)&=(z_j+2\gaa_j\ov z_j)^2,\quad j=e,h;\\
Q_s(z',\ov z')&=(z_{s+s_*}+2\gaa_{s+s_*}\ov z_s)^2,\\
Q_{s+s_*}(z',\ov z')&=
(z_{s}+2\gaa_{s}\ov z_{s+s_*})^2.
\end{align*}
Let us first find necessary conditions on the linear parts of $\rho_i$ for \re{Qzpb} to  be solvable.
Let $w'={\mathbf A}z'$ and $\tilde w'=\tilde{\mathbf A}z'$.
Comparing the quadratic terms in \re{Qzpb}
for $i=1,2$, we see that
\gan
(z_j+2\gaa_jw_j)^2=(z_j+2\gaa_j\tilde
w_j)^2,\\
( z_{s+s_*}+2\gaa_{s+s_*}w_s)^2=( z_{s+s_*}+2\gaa_{s+s_*}\tilde
w_s)^2,\\
(z_{s}+2\gaa_{s}w_{s+s_*})^2=( z_{s}+2\gaa_{s}\tilde
w_{s+s_*})^2.
\end{gather*}
Here $\gaa_{s+s_*}=1-\ov\gaa_s$, by \re{gehs11}.
For each $j$, $w_j\neq\tilde w_j$. Otherwise, the fixed points of $\rho_1$ and $\rho_2$
do not intersect transversally.   Therefore, the above $3$ identities can be written as
\gan
z_j+2\gaa_jw_j=-(z_j+2\gaa_j\tilde
w_j),\\
z_{s+s_*}+2\gaa_{s+s_*}w_s=-( z_{s+s_*}+2\gaa_{s+s_*}\tilde
w_s),\\
z_{s}+2\gaa_{s}w_{s+s_*}=-( z_{s}+2\gaa_{s}\tilde
w_{s+s_*}).
\end{gather*}
In the matrix form, we get  
\begin{equation}\label{tildeAA}
\tilde {\mathbf{A}}=-{\boldsymbol{\gamma}}^{-1} -\mathbf{A},\text{ with}\quad
\boldsymbol{\gamma}:=\begin{pmatrix}
  \boldsymbol{\gaa}_{e_*}    &\mathbf{0} &\mathbf{0} &\mathbf{0}\\
    \mathbf{0}&\boldsymbol\gaa_{h_*}&\mathbf{0}& \mathbf{0}  \\
   \mathbf{0}&\mathbf{0}&\mathbf{0}&\boldsymbol{\gaa}_{s_*} \\
   \mathbf{0}&\mathbf{0}&{\boldsymbol{\tilde \gaa}}_{s_*}&\mathbf{0}
\end{pmatrix}.
\end{equation}
Here   in matrices $\tilde{\boldsymbol{\gaa}}_{s_*}=\mathbf I_{s_*}-\ov {\boldsymbol{\gaa}}_{s_*}$.
Let us express in block matrices
$$
{\mathbf A}=\begin{pmatrix}
{\mathbf A}_{e_*e_*}  &{\mathbf A}_{e_*h_*} &{\mathbf A}_{e_*s_*} &{\mathbf A}_{e_*(2s_*)} \\
{\mathbf A}_{h_*e_*}  &{\mathbf A}_{h_*h_*} &{\mathbf A}_{h_*s_*} &{\mathbf A}_{h_*(2s_*)} \\
{\mathbf A}_{s_*e_*}  &{\mathbf A}_{s_*h_*} &{\mathbf A}_{s_*s_*} &{\mathbf A}_{s_*(2s_*)} \\
{\mathbf A}_{(2s_*)e_*}  &{\mathbf A}_{(2s_*)h_*} &{\mathbf A}_{(2s_*)s_*} &{\mathbf A}_{(2s_*)(2s_*)}
\end{pmatrix}
$$
where the diagonal block matrices are of sizes $e_*\times e_*,h_*\times h_*,s_*\times s_*$, and $s_*\times s_*$, respectively.
By  $\mathbf{A}\ov {\mathbf{A}}=\mathbf I$,  $\tilde {\mathbf{A}}\ov{\tilde {\mathbf{A}}}=\mathbf I$ and \re{tildeAA}
we get $(\tilde {\mathbf{A}}\ov{\tilde {\mathbf{A}}}-\mathbf{A}\ov {\mathbf{A}})\ov\gaa=0$, i.e.
$
{\boldsymbol{\gamma}}^{-1}+\mathbf{A}+\boldsymbol{\gamma} ^{-1}\ov {\mathbf{A}}\ov{\boldsymbol{\gamma}}=0.
$
Recall that $\gaa_1^2,\ldots, \gaa_{e_*+h_*}^2$
 are real and distinct. It is easy to see that $\mathbf{A}_{e_*h_*}=0$, $\mathbf{A}_{h_*e_*}=0$, and $\mathbf{A}_{e_*e_*}, \mathbf{A}_{h_*h_*}$ are diagonal. Also,
\eq{aeses}
\mathbf{A}_{e_*e_*}+\ov {\mathbf{A}}_{e_*e_*}=-\boldsymbol{\gamma}_{e_*}^{-1}, \quad \mathbf{A}_{h_*h_*}+\ov {\mathbf{A}}_{h_*h_*}=-\boldsymbol{\gamma}_{h_*}^{-1}.
\eeq
In block matrices, we obtain
\ga 
{\boldsymbol{\gamma}}_j^{-1} \ov {\mathbf{A}}_{j(2s_*)}\ov{\tilde{\boldsymbol{\gamma}}}_{s_*}=- {\mathbf{A}}_{js_*},   \qquad
{\tilde{\boldsymbol{\gamma}}}_{s_*} ^{-1}\ov{\mathbf{A}}_{ (2s_*)j}\ov{\boldsymbol{\gamma}}_{j}=-{\mathbf{A}}_{s_*j}; \\ 
{\boldsymbol{\gamma}}_j^{-1}\ov {\mathbf{A}}_{js_*}\ov{\boldsymbol{\gamma}}_{s_*}=-{\mathbf{A}}_{j(2s_*)},
 \qquad{\boldsymbol{\gamma}}_{s_*}^{-1} \ov {\mathbf{A}}_{s_*j}{\boldsymbol{\gamma}}_{j}=- {\mathbf{A}}_{(2s_*)j}; \\ 
{\tilde{\boldsymbol{\gamma}}}_{s_*}^{-1}\ov {\mathbf{A}}_{(2s_*)(2s_*)}\ov{
\tilde{\boldsymbol{\gamma}}}_{s_*}=- {\mathbf{A}}_{s_*s_*},
 \qquad{\tilde{\boldsymbol{\gamma}}}_{s_*}^{-1}\ov {\mathbf{A}}_{(2s_*)s_*}\ov{\boldsymbol{\gamma}}_{s_*}=-{\mathbf{A}}_{s_*(2s_*)}-{\tilde{\boldsymbol{\gamma}}}_{s_*}^{-1},
\label{tgss}
\\
{\boldsymbol{\gamma}}_{s_*}^{-1}\ov {\mathbf{A}}_{s_*(2s_*)}\ov{\tilde{\boldsymbol{\gamma}}}_{s_*}=- {\mathbf{A}}_{(2s_*)s_*}-{\boldsymbol{\gamma}}_{s_*}^{-1},
 \qquad  {\boldsymbol{\gamma}}_{s_*}^{-1}\ov {\mathbf{A}}_{s_*s_*}\ov{\boldsymbol{\gamma}}_{s_*}=-{\mathbf{A}}_{(2s_*)(2s_*)}.
\end{gather} 
In the first $4$ equations,  we have $j=e_*,h_*$.

By \rl{disga2},  we know that $\gaa_e^2,\gaa_h^2$, and $\gaa_{s}\ov\gaa_{s+s_*}$ are distinct.
 Thus,  ${\mathbf{A}}_{js_*}={\mathbf{A}}_{j(2s)_*}=\mathbf 0$
and ${\mathbf{A}}_{s_*j}={\mathbf{A}}_{(2s_*)j}=\mathbf 0$ for $j=e_*,h_*$.  Since $\gaa_{s}\ov{{\gaa}}_{s+s_*}$ is
different from all $\gaa_{s+s_*}\ov\gaa_{s}$, then ${\mathbf{A}}_{s_*s_*}={\mathbf{A}}_{(2s_*)(2s_*)}=\mathbf 0$ while
${\mathbf{A}}_{s_*(2s_*)}$, ${\mathbf{A}}_{(2s_*)s_*}$ are diagonal.
Now ${\mathbf{A}}\ov {\mathbf{A}}=\mathbf I$ implies that
\eq{aehs}
{\mathbf{A}}_{e_*e_*}\ov {\mathbf{A}}_{e_*e_*}=\mathbf I, \quad
{\mathbf{A}}_{h_*h_*}\ov {\mathbf{A}}_{h_*h_*}=\mathbf I, \quad
{\mathbf{A}}_{s_*(2s_*)}\ov {\mathbf{A}}_{(2s_*)s_*}=\mathbf I.
\eeq
 Combining the first identities  in \re{aeses} and \re{aehs}, we know that  the  diagonal $e$th element $a_e$ of
 ${\mathbf{A}}_{e_*e_*}$ must satisfy
 $
 a_e+
 \ov a_e=-\gaa_e^{-1}, 
   a_e\ov a_e=1.
 $
 Since   $0<\gaa_e<1/2$,  there is no such solution $a_e$ if $e_*>0$.
  We have verified (i).

(ii)  For the hyperbolic components,  by \re{gens11} we have $\gaa_h^{-1}=\la_h+\ov\la_h$ with $|\la_h|=1$
  By the second identities in
  \re{aeses},  \re{aehs}, and by \re{tildeAA},  we obtain $(a_h,\tilde a_h)=(-\lambda_h,-\bar\lambda_h)$ or $(-\bar\lambda_h,-\lambda_h)$.
For the complex components, we  use  ${\mathbf{A}}_{s_*(2s_*)}\ov {\mathbf{A}}_{(2s_*)s_*}=\mathbf I$
and multiply the second identity in \re{tgss} by the diagonal matrix $\mathbf A_{s_*(2s_*)}$.  Thus the $ 
 s
 $th diagonal element $a_s$ of  $\mathbf A_{s_*(2s_*)}$ satisfies
$
a_s(a_s+\tilde\gaa_s^{-1})+\tilde \gaa_s^{-1}\ov\gaa_s=0.
$
By the last identity in \re{gsgssmu}, we get
$$
a_s^2+(1 + 
\mu_s^{-1})a_s+ 
   \mu_s^{-1}=0.
$$
 Obviously $a_s=-1, -\mu_s^{-1}$ are solutions.
By \re{tildeAA}, we get $(a_s,\tilde a_s)=(-1, 1- 
\tilde\gamma_s^{-1}) 
=(-1,-\mu_s^{-1})$ or $(- 
\mu_s^{-1},  
\mu_s^{-1}- 
\tilde\gamma_s^{-1})=(-\mu_s^{-1},-1)$. 
 Each tuple determines 
 a tuple $(b_s,\tilde b_s)$ by \re{aehs}, with $b_s$ being the diagonal entries of $\mathbf A_{(2s_*)s_*}$.
 We verify that
$(b_s,\tilde b_s)=(\ov a_s^{-1},\ov{\tilde a_s}^{-1})$.
There are exactly $2^{h_*+s_*-1}$ solutions for $\mathbf A$ and $\tilde {\mathbf A}$  since we can only determine the
pairs
$\{\mathbf A_{h_*h_*}, \tilde {\mathbf A}_{h_*h_*}\},  
\{\mathbf A_{s_*(2s_*)}, 
\tilde{\mathbf A}_{s_*(2s_*)}\}
$.
Indeed, we have
\begin{gather}
\mathbf A=\begin{pmatrix}
\diag (a_h) &{\mathbf{ 0}}&{\mathbf{ 0}} \\
{\mathbf{ 0}}&{\mathbf{ 0}} & \diag (a_s) \\
{\mathbf{ 0}}& \diag (b_s) &{\mathbf{ 0}}
\end{pmatrix}, \quad
\tilde A=\begin{pmatrix}
\diag (\tilde a_h) &{\mathbf{ 0}}&{\mathbf{ 0}} \\
{\mathbf{ 0}}&{\mathbf{ 0}} & \diag (\tilde a_s) \\
{\mathbf{ 0}}& \diag (\tilde b_s) &{\mathbf{ 0}}
\end{pmatrix},\nonumber\\
\label{ata-1}
\diag\nu:=\mathbf {\tilde A}^{-1}\mathbf A=\begin{pmatrix}
\diag (\tilde a_h^{-1}a_h) &{\mathbf{ 0}}&{\mathbf{ 0}} \\
{\mathbf{ 0}}&\diag (\tilde b_s^{-1}b_s) &{\mathbf{ 0}}  \\
{\mathbf{ 0}}&{\mathbf{ 0}}  &\diag (\tilde a_s^{-1}a_s)
\end{pmatrix},\\
\label{nunue}
 \nu=\mu_\epsilon=( 
 \mu_h^{\e_h},\ov\mu_s^{\e_{s}},\mu_{s}^{-\e_{s}}), \quad \e_h^2,\e_s^2=1, \quad \nu_{s+s_*}=\ov \nu_{s}^{-1},
\end{gather}
 where there are 
$2^{h_*+s_*}$ distinct combinations.
Thus,  we get exactly $2^{h_*+s_*-1}$ pairs $\{K_{\e}^1,K_{\e}^2\}$
of asymptotic linear submanifolds indexed by  $\e=(\e_1, \ldots,\e_{h_*+s_*})$ with $\e_j^2=1$
 for the product quadric.
 The attached formal submanifolds associated  to these linear asymptotic submanifolds are  unique and restricting to $\e_i=1$ for all $i$, it  is given by
\aln
z_{p+h}=(1-{4}\gaa_h^{2})z_h^2,\
z_{p+s}=(1-2\gaa_{s+s_*})^2z_{s+s_*}^2,\
z_{p+s+s_*}=(1-2\gaa_s)^2z_{s}^2.
\end{align*}
 Here we have used $(1-{4}\gaa_h^{2})=(1-2\gamma_h\lambda_h)^2$.

 In summary, we have shown that there are exactly $2^{h_*+s_*-1}$ pairs of linear anti-holomorphic involutions $\{\rho_1, \rho_2\}$.  In (iv) we show that under the non-resonant conditions on $\mu_1,\dots, \mu_p$,
 they are the only pairs of  anti-holomorphic involutions.
 This finishes the proof of  (ii).

(iii).
Let us continue the computation for the perturbations.
We have determined linear parts of antiholomorphic involutions $\rho_i$.
We expand components of $R(z')$ as
$
R_{j}(z')=\sum_{k=2}^{\infty}R_{j; k}(z'). 
$
Here $R_{j;k}$ are homogeneous terms of degree $k$.  We expand $\tilde R_j$ analogously.
Suppose that terms of order up to $k-1$ in $R_j,\tilde R_j$
have been determined. For the hyperbolic components, we need to solve the equations
\ga\label{zhrh}
4\sqrt{1-4\gaa_h^{2}}z_h(R_{h;k}(z')+{\tilde R_{h;k}}(z'))=\cdots,
\end{gather}
where the right-hand side has been determined.  Indeed, let us compute the  terms of degree $k$
in 
\re{Qzpb} to obtain
$$
(1-2\gaa_j\lambda_j)^2z_j^2 + 2(1-2\gaa_j\lambda_j)z_jR_{j;k}= (1-2\gaa_j\lambda_j^{-1})^2z_j^2 + 2(1-2\gaa_j\lambda_j^{-1})z_j\tilde R_{j;k}+ {\mathcal R}
$$
where ${\mathcal R}$ is a polynomial that depends on $\tilde R_{j;l},R_{j;l}$, $l<k$. Since $(1-2\gaa_j\lambda_j)=-(1-2\gaa_j\lambda_j^{-1})$, we obtain \re{zhrh}.

	When $p>1$, the system of equations \re{zhrh} cannot be solved
even formally, unless the right-hand side is divisible by $z_h$.  When $p=1$,
the equation \re{zhrh} is clearly solvable. In fact, under the non-resonant condition on $\mu_1$, the formal  anti-holomorphic  involutions $\{\rho_1,\rho_2\}$
can be uniquely determined.

 Let us keep the above notation and compute for the example stated in (iii).
We need to solve
\aln
(z_1+2\gaa_1\tilde  w_1)^2+(z_2+2\gaa_2\tilde  w_2)^3& =
(z_1+2\gaa_1 w_1)^2+(z_2+2\gaa_2 w_2)^3,\\
(z_2+2\gaa_2\tilde  w_2)^2&=(z_2+2\gaa_2 w_2)^2.
\end{align*}
Again $\tilde w_2- w_2$ cannot be identically zero. Thus
$
\tilde w_2=- w_2-\gaa_2^{-1} z_2.
$
Then we need to solve
\aln
(z_1+2\gaa_1\tilde  w_1)^2 =
(z_1+2\gaa_1 w_1)^2+2(z_2+2\gaa_2 w_2)^3.
\end{align*}
By (ii), we know that $ w_1=-\la_1z_1+R_1(z')$ and $ w_2= -\la_2 z_2+R_2(z')$ with $R_i(z')=O(2)$.
Also $\tilde  w_1=-\ov\la_1z_1+\tilde R_1(z')$ and $\tilde w_2= -\ov\la_2 z_2+\tilde R_2(z')$.
Comparing the cubic terms implies that $z_1$ must divide
$
2(1 -2{\gaa_2}\la_2)^3z_2^3,
$
which is a contradiction.

(iv)
For a general $M$, following Klingenberg \ci{Kl85} we reformulate the problem by considering the
following equations
\aln 
h(z')&=Q(z',\ov \rho_i(z'))+H(z',\ov \rho_i(z')),\quad i=1,2,\\
h^*(\ov\rho_i(z'))&=\ov Q( \ov\rho_i(z'), z')+\ov {H}(\ov\rho_i(z'),z'), \quad i=1,2.
\end{align*}
Here $h,h^*, \ov\rho_i$ are unknowns.  Initially, we require that
 $\ov\rho_1,\ov\rho_2$ be arbitrary biholomorphic maps, except
 their linear parts
match with $z'\to Az'$ and $z'\to\tilde Az'$.
This will ensure that the solutions $\ov\rho_i$ are unique and they are involutions.

 As demonstrated in (iii), in general there is no formal submanifold attached to $M$.
Thus we assume that $M$ is a higher order perturbation of
 product quadric without elliptic component and it admits the maximum number of deck transformation.


We  
 may assume that
\al
\label{zpsq}
z_{p+h}&=(z_h+2\gaa_h\ov z_h+E_{ h}(z',\ov z'))^2, \\
\quad z_{p+s}&=(z_s+2\gaa_s\ov z_{s+s_*}+  E_{ s}(z',\ov z'))^2,\\
z_{p+s+s_*}&=(
z_{s+s_*}+2\gaa_{s+s_*}\ov  z_s+  E_{ s+s_*}(z',\ov z'))^2.
\label{zpsq3}
\end{align}
For late references, we express the above in an abbreviated form:
\eq{defml}
M\subset\cc^{2p}\colon z_{p+j}=(L_j(z',\ov z')+E_j(z',\ov z'))^2, \quad 1\leq j\leq p.
\eeq
 We fix linear parts of $\rho_i$ such that
 $$
  \rho_1(z')=\ov A\ov z'+\ov R(\ov z'), \quad  \rho_2(z')=\ov {\tilde A}\ov z'+\ov {\tilde R}(\ov z').
   $$
For $i=1,2$ we then need to solve   $\ov\rho_i$ 
from
 \al\label{zh2g}
 z_h+2\gaa_h\ov \rho_{ih}+ E_h(z',\ov\rho_i)&=(-1)^if_h(z'),\\
 z_s+2\gaa_s\ov \rho_{is+s_*}+E_{s}(z',\ov\rho_i)&=(-1)^if_{s}(z'),\\
\label{zsss}
 z_{s+s_*}+2\gaa_{s+s_*}\ov\rho_{is}+E_{s+s_*}(z',\ov\rho_i)&=(-1)^if_{s+s_*} (z'),\\
 \label{2ghz}
  2\gaa_hz_h+\ov \rho_{ih}+ \ov E_h(\ov\rho_i,z')&=(-1)^i f^*_h(\ov\rho_i),\\
  2\ov\gaa_sz_{s+s_*}+\ov \rho_{is}+\ov E_{s}(\ov\rho_i,z')&=(-1)^if_s^*(\ov\rho_i),\\
  2\ov\gaa_{s+s_*}z_s+\ov\rho_{is+s_*}+\ov E_{s+s_*}(\ov\rho_i,z')&=
  (-1)^if_{s+s_*}^*(\ov\rho_i). \label{zh2g6}
  \end{align}
 Suppose that we have  determined terms of $R_{j},\tilde R_j,
 f_j,f_{j}^*$ of order $<k$.
 We have
 $$
 \ov\rho_1(z')=\mathbf Az'+R(z'), \quad \ov\rho_1^{-1}(z')=\mathbf A^{-1}z'-\mathbf A^{-1}R'(\mathbf A^{-1}z'),
 $$
 where  the terms in $R'-R$ of order $k$ depend only on terms of $R$ of order $<k$.
For terms of order $k$, by eliminating $f_j$, we    need to solve
 \ga\label{rjq}
 R_{jQ}+\tilde R_{jQ}=\cdots
 \end{gather}
 where the dots denote terms which have been determined.
 We compose from
 right in the last 3 identities for $i=1$ (resp. $i=2$)  by $\ov\rho_1^{-1}$ (resp. $\ov\rho_2^{-1}$).    Eliminating $f^*$ from the new identities, we obtain
 $$
 \mathbf A^{-1}R (\mathbf A^{-1}z') +
  \tilde {\mathbf A}^{-1}\tilde R (\tilde{\mathbf A}^{-1}z') =\cdots.
  $$
Recall that  $\tilde{\mathbf A}^{-1}{\mathbf A}=\diag\nu$  with $\nu:=\nu_{\e}$. 
Multiplying on the left by $\tilde {\mathbf A}$,  using $\tilde{\mathbf A}{\mathbf A}^{-1}=(\diag\nu)^{-1}$, and  evaluating at $z'= {\mathbf A}\tilde z'$,
we thus need to solve \re{rjq} and
$$
\nu_j^{-1}R_{j,Q}+\nu^Q {\tilde R}_{j,Q}=\cdots.$$
  This shows that   $R_{j}, \tilde R_{j}$ are uniquely determined
as   \eq{nuqn0}
  \nu^Q\neq\nu_j^{-1}, \quad Q\in\nn^p, \quad |Q|>1,\quad 1\leq j\leq p.
 \eeq

 To verify that $\rho_i$ are involutions,  we compose by $\ov \rho_i^{-1}$ from
 right
 in \re{zh2g}-\re{zsss}, and we apply complex conjugate to the coefficients of the new identities.
 This results in   \re{2ghz}-\re{zh2g6} in which $(\ov\rho_i, f_j^*)$ are replaced
 by $(\ov{(\ov\rho_i)^{-1}}, \ov f_i)$. We can also start with  \re{2ghz}-\re{zh2g6}
  and apply the same procedure to get \re{zh2g}-\re{zsss}, in which $(\ov\rho_i, f_i)$
 are replaced by $(\ov{(\ov\rho_i)^{-1}}, \ov f^*_i)$.  By the uniqueness of the solutions, we conclude
 that  $\ov {(\ov\rho_i)^{-1}}=\ov \rho_i$ as both sides have the same linear part.
 We now have
 $
 \ov{(\ov\rho_i)^{-1}(\ov z')}=\ov{\rho_i(z')}$. Hence, by \re{def-rhobar}, $ \ov z'=\ov\rho_i(\rho_i(z'))=\ov{\rho_i^2(z')}.
  $
   This shows that each $\rho_i$ is an involution.
 \end{proof}

We now can prove the following theorem.
\begin{thm}\label{invep}
  Let $M$ be a real analytic submanifold in $\cc^{2p}$ defined by \rea{defml}
 without 
 elliptic components.
  Assume that in $(\xi,\eta)$ coordinates,
   $D\sigma(0)$ is diagonal and has distinct eigenvalues $\mu_1,\ldots, \mu_p$,  $\mu_1^{-1}, \ldots, \mu_p^{-1}$.
Let $\nu=\nu_\e$ be of the form \rea{nunue} and satisfy \rea{nuqn0}.
  Then $M$ admits
a unique pair of 
formal asymptotic submanifold  $\{K^{\e}_1, K^{\e}_2\}$ such that the complexification
 of $K^{\e}_1$ in $\cL M$
is an invariant formal submanifold $\cL H_\e$
of $\sigma$ that is tangent to 
\eq{clhe}
\cap_{\e_j=1}\{\eta_j=0\}\cap\cap_{\e_j=-1}\{\xi_j=0\}.
\eeq
Furthermore, the complexification of $K_2^\e$ equals $\tau_1\cL H_\e$. \end{thm}
\begin{proof} Let $K_i=K_i^\e$.
We will follow Klingenberg's approach for $p=1$,  by using the deck transformations.
Suppose that
$K$ is an attached formal complex submanifold which intersects with $M$ at two
totally real formal submanifolds $K_1,K_2$.
We first embed
$K_{1}\cup K_{2}$ into $\cL M$ as $M$ is embedded into $\cL M$. Let $\cL K_{i}$ be
the complexification of $K_{i}$ in $\cL M$. Since $\rho$ fixes $K_{i}$ pointwise,
then $\rho\cL K_{i}=\cL K_{i}$.

 We want to show that $\tau_1(\cL K_{1})=\cL K_{2}$;
thus $\cL K_{i}$ is invariant under $\sigma$.  We can see that $\cL K_{i}$ is defined by
\eq{briz}
\ov\rho_i(z')=w'.
\eeq
On $\cL K_{1}$,  by \re{zh2g} and \re{zsss} we have
$\tilde L(z',w')+E(z',w')=-f(z')$.  The latter defines a complex
submanifold of dimension $p$. Thus it must be  $\cL K_{1}$.
On $\cL M$,
$$
(\tilde L_j(z',w')+E_j(z',w'))^2=z_{p+j}
$$
are invariant by $\tau_{1}$. Thus each $\tilde L_j(z',w')+E_j(z',w')$ is either
invariant or skew-invariant by $\tau_1$. Computing the linear part, we conclude
that they are all skew-invariant by $\tau_1$. Hence $\tau_1(\cL K_1)$ is defined by
$
\tilde L(z',w')+E(z',w')=f(z'),
$
which is the defining equations for $\cL K_2$.     We must identify the tangent space of $\cL K_1$ at the origin.
Let us verify \re{clhe} for $\e_j=1$ for all $j$, while the general case is analogous. Let $A, S$ be the linear parts
of $\ov\rho_1$ and $\sigma=\tau_1\tau_2$.  Define $e(z',w')=w'-A(z')$. The tangent space to $\cL K_1$ at the origin is $\{e(z',w')=0\}$. From the proof of \rl{asynum} (ii), the
matrix of $A$ is
 $$\mathbf A=\begin{pmatrix}
\diag (a_h) &{\mathbf{ 0}}&{\mathbf{ 0}} \\
{\mathbf{ 0}}&{\mathbf{ 0}} & \diag (a_s) \\
{\mathbf{ 0}}& \diag (b_s) &{\mathbf{ 0}}
\end{pmatrix}=
\begin{pmatrix}
-\diag (\la_h) &{\mathbf{ 0}}&{\mathbf{ 0}} \\
{\mathbf{ 0}}&{\mathbf{ 0}} & -\mathbf I \\
{\mathbf{ 0}}& -\mathbf I&{\mathbf{ 0}}
\end{pmatrix}.$$
Thus
$
e_h=w_h+\la_hz_h, e_s= w_s+z_{s+s_*}$,  and $  e_{s+s_*}=w_{s+s_*}+ z_s.
$
Using the formulas \re{tau10e} and \re{tau12s} of $\tau_1,\tau_2$ when $M$ is the product quadric, we can verify that
$e_{s+s_*}\circ S=\ov\mu_s e_{s+s_*}, e_s\circ S=\mu_s^{-1}e_s$,
and $e_h\circ S=\ov\mu_h e_h$. Therefore,  $e_j(z',w')=c_j\eta_j$.

Finally, if $\cL K_1$ is convergent, then \re{briz} implies that $\ov\rho_1$ is convergent. Hence $K_1$, the fixed point set of $\rho_1$,
 is convergent.\end{proof}

We now study the convergence of attached formal submanifolds.
Let us first recall a  theorem of P\"oschel~\cite{Po86}. Let $\nu$ and $\e$ be as in \re{nunue}.
Define
$$
\omega_{\nu}(k)=\min_{1<|P|\leq 2^k, P\in\nn^p}\min_{1\leq i\leq p}
\left\{|\nu^P-\nu_i|,
|\nu^P-\nu_i^{-1}|\right\}.
$$
Suppose that
\eq{omnu}
-\sum\frac{\log\omega_{\nu}(k)}{2^k}<\infty.
\eeq
Then the unique invariant formal submanifold of $\sigma$ that is tangent  to the $\cL H_\e$  defined by \re{clhe}
is convergent.

We now obtain a consequence of \rt{invep} and P\"oschel's theorem.
\begin{thm}\label{pocor}
 Let $M$ be as in \rta{invep}.
 Let $\nu=\mu_\e$ be given by \rea{nunue}.   Assume that $\nu=(\mu_1^{\epsilon_1}, \ldots, \mu_{p}^{\epsilon_{p}})$
satisfy  \rea{omnu}. Then $M$ admits an attached complex submanifold.
\end{thm}

 To study the
 convergence of all attached formal manifolds, we use 
  a
 theorem
in~\cite{stolo-bsmf} to conclude simultaneous convergence of all attached formal submanifolds. In fact the conclusion
is much more stronger.  
It is based on the simultaneous linearization of the $\sigma_j$'s on the resonant ideal, i.e. the ideal $\cL I$ generated by $\xi_1\eta_1,\ldots, \xi_p\eta_p$.
Define with $D\sigma_i(0):=\diag (\mu_{i,1}, \ldots, \mu_{i,n})$,
\begin{equation}\label{bnI}
\omega_{\cL S,{\cL  I}}(k)=\inf\left\{\max_{1\leq i\leq l}|\mu_{i}^Q-\mu_{i,j}|\neq 0\colon
\;2\leq |Q|\leq 2^k, 1\leq j\leq n,Q\in \nn^n, x^Q\not\in {\cL  I}\right\}
\end{equation}
where $\mu_i^Q:=\mu_{i,1}^{q_1}\cdots\mu_{i,n}^{q_n}$. As in ~\cite{stolo-bsmf}, we say that the family $\cL S$ is {\it Diophantine} on $\cL I$, if the sequence of numbers \re{bnI} satisfies \re{omnu}.

\begin{thm} 
Let $M$ be as in \rta{pocor}, given by \rea{defml}. 
Assume furthermore that $M$ is non resonant. Suppose that $\cL S$ is Diophantine on $\cL I$ or that
$M$ has an abelian CR singularity of  pure complex type. 
 Then
 all attached formal submanifolds are convergent.
Moreover, the complex submanifolds  $K_j$ 
 attached to   pairs of antiholomorphic involutions
$\{\rho_{j1},\rho_{j2}\}$ have the form
\eq{eqkjL}
K_j\colon z_{p+i}=(L_i(z',\ov\rho_{j1}(z'))+E_i(z',\ov\rho_{j1}(z')))^2, \quad 1\leq i\leq p.
\eeq
\end{thm}
\begin{proof}  When $M$ has an abelian CR singularity of pure complex type, from the normal form of $\tau_{ij}$
in \rt{abelinv} we know that all invariant submanifolds of $\sigma$ that are tangent to \re{clhe}
are convergent.
The non-abelian CR singularity case is a consequence of the theorem of simultaneous linearization of the $\sigma_j$'s along the resonant ideal $\cL I$~\cite{stolo-bsmf}[theorem 2.1] and \rt{invep}. 
Since, in good holomorphic coordinates, $\sigma$ is linear on the zero set of the resonant
 ideal, the solutions $\{\ov\rho_1,\ov\rho_2\}$
to \re{zh2g}-\re{zh2g6} are linear and there are $2^{h_*+s_*-1}$ pairs $\{\ov\rho_{j1},\ov\rho_{j2}\}$ of solutions.
The equation \re{eqkjL} is derived in  \re{dfKzp} for a general situation.
\end{proof}

\bibliographystyle{alpha}
\def\cprime{$'$}

\end{document}